\documentclass[a4paper,11pt]{amsart}
\usepackage[backend=bibtex,style=numeric,doi=false,isbn=false,url=false,eprint=false,giveninits=true]{biblatex}
\AtEveryBibitem{\clearfield{pages}}
\renewbibmacro*{issue+date}{%
  \ifboolexpr{not test {\iffieldundef{year}} or not test {\iffieldundef{issue}}}
    {\printtext[parens]{%
       \iffieldundef{issue}
         {\usebibmacro{date}}
         {\printfield{issue}%
          \setunit*{\addspace}%
          \usebibmacro{date}}}}
    {}%
  \newunit}
\renewbibmacro{in:}{}
\usepackage{amsmath}
\usepackage{amsthm,amsfonts,amssymb,graphics}
\usepackage{pgfplots}

\usepackage{hyperref}
\usepackage{color}
\usepackage[normalem]{ulem}
\usepackage{enumerate}
\usepackage{a4wide}
\usepackage{multirow}
\usepackage{booktabs}
\usepackage[foot]{amsaddr}
\usetikzlibrary{calc}
\usetikzlibrary{calc}

\usepackage[T1]{fontenc}
\usepackage[utf8]{inputenc}
\usepackage{subcaption}
\usepackage{dsfont}
\usepackage{tikz}

\newtheorem{theorem}{Theorem}[section]
\newtheorem{lemma}[theorem]{Lemma}
\newtheorem{proposition}[theorem]{Proposition}
\newtheorem{corollary}[theorem]{Corollary}

\newtheorem{assumption}[theorem]{Assumption}

\theoremstyle{remark}
\newtheorem{remark}[theorem]{Remark}
\newtheorem{definition}[theorem]{Definition}

\numberwithin{equation}{section}

\mathchardef\mhyphen="2D

\begin{document}

\title[Smoothness/monotonicity of the excursion set density of Gaussian fields]{Smoothness and monotonicity of the excursion set density of planar Gaussian fields}
\author{Dmitry Beliaev\textsuperscript{1}}
\address{\textsuperscript{1}Mathematical Institute, University of Oxford}
\email{belyaev@maths.ox.ac.uk}
\author{Michael McAuley\textsuperscript{1,2}}
\email{mcauley@maths.ox.ac.uk}
\address{\textsuperscript{2}Present address:  Department of Mathematics and Statistics, University of Helsinki.} \email{michael.mcauley@helsinki.fi}
\author{Stephen Muirhead\textsuperscript{3}}
\address{\textsuperscript{3}Department of Mathematics, King's College London\\
	\emph{Present address: School of Mathematical Sciences, Queen Mary University of London}}
\email{s.muirhead@qmul.ac.uk}
\thanks{The authors thank an anonymous referee for their careful reading of the manuscript, for making us aware of \cite{maxwell1870hills} and, in particular, for pointing out an error in Lemma~\ref{l:conditional one arm decay}. The authors also thank Igor Wigman and Ben Hambly for helpful comments on a slightly different version of this work. The first author was supported by the Engineering \& Physical Sciences Research Council (EPSRC) Fellowship EP/M002896/1.}


\begin{abstract}
Nazarov and Sodin have shown that the number of connected
components of the nodal set of a planar Gaussian field in a ball of
radius $R$, normalised by area, converges to a constant as $R\to \infty $.
This has been generalised to excursion/level sets at arbitrary levels,
implying the existence of functionals $c_{ES}(\ell )$ and $c_{LS}(\ell )$
that encode the density of excursion/level set components at the level
$\ell $. We prove that these functionals are continuously differentiable
for a wide class of fields. This follows from a more general result,
which derives differentiability of the functionals from the decay of
the probability of `four-arm events' for the field conditioned to have
a saddle point at the origin. For some fields, including the important
special cases of the Random Plane Wave and the Bargmann-Fock field, we
also derive stochastic monotonicity of the conditioned field, which allows
us to deduce regions on which $c_{ES}(\ell )$ and $c_{LS}(\ell )$ are
monotone.
\end{abstract}

\maketitle

\section{Introduction}
\label{s:introduction}

Let $f:\mathbb{R}^{2}\to \mathbb{R}$ be a continuous stationary Gaussian
field with zero mean and covariance function
$K:\mathbb{R}^{2}\to \mathbb{R}$ defined by
$K(x)=\mathbb{E}(f(x)f(0))$. We are interested in the geometric properties
of the (upper-)excursion sets and level sets of this field, defined respectively
as
\begin{displaymath}
\left \{  x\in \mathbb{R}^{2}:f(x)\geq \ell \right \}  \quad \text{and}
\quad \left \{  x\in \mathbb{R}^{2}:f(x)=\ell \right \}
\end{displaymath}
for $\ell \in \mathbb{R}$. Specifically, we are interested in the number
of connected components of these sets in a large domain.

Smooth Gaussian fields are used to model spatial phenomena across wide
ranging domains of science, such as quantum chaos, medical imaging and
oceanography (see
\cite{jain2017nodal,worsley1996unified,azais2009level} respectively). As
a particular example; cosmological theories predict that the Cosmic Microwave\vadjust{\goodbreak}
Background Radiation observed on Earth can be well modelled as a realisation
of a stationary Gaussian field on the two-dimensional sphere. Topological
and geometric quantities provide a useful way of testing this prediction,
which has important physical implications. In particular,
\cite{pranav2018unexpected} compares the number of excursion set components
of the observed background radiation to the corresponding number for simulated
Gaussian fields at a range of levels. A better understanding of the statistical
properties of the number of excursion sets of Gaussian fields could therefore
have consequences for such analysis.

Unlike certain other geometric functionals (e.g.\ the volume or Euler characteristic
of excursion sets), the number of connected components is inherently difficult
to study because it is non-local: the number of components in a domain
cannot be counted by partitioning the domain and simply counting the number
of components in each sub-domain. Nazarov and Sodin
\cite{SodinNazarov2015asymptotic} used an ergodic argument to study the
asymptotics of this quantity. Specifically, if $f$ is an ergodic Gaussian
field satisfying some regularity assumptions, $B(R)$ is the ball of radius
$R>0$ centred at the origin, and $N_{LS}(R,0)$ is the number of components
of the nodal set $\{x\in \mathbb{R}^{2}:f(x)=0\}$ contained in
$B(R)$, then there is a constant $c_{LS}(0) \ge 0$ such that
\begin{equation}
\label{e:ns}
\frac{N_{LS}(R,0)}{\pi R^{2}}\to c_{LS}(0)
\end{equation}
as $R\to \infty $, where convergence occurs in $L^{1}$ and almost surely.
Although this result was stated only for the nodal set, the arguments in
\cite{SodinNazarov2015asymptotic} go through verbatim for excursion/level
sets at arbitrary levels $\ell $; the respective limiting constants, denoted
by $c_{ES}(\ell )$ and $c_{LS}(\ell )$, can be interpreted as the density
of excursion/level set components per unit area.

In this paper we consider properties of $c_{ES}(\ell )$ and
$c_{LS}(\ell )$ viewed as functions of the level. It was shown in
\cite{Beliaev2018Number} that $c_{ES}$ and $c_{LS}$ are absolutely continuous.
Our main results (Theorems \ref{t:differentiability of c_{LS}},~\ref{t:Differentiability equivalence}
and Corollary~\ref{c:Differentiability equivalence}) show that, for a wide
class of fields, the continuous differentiability of $c_{ES}$ and
$c_{LS}$ at $\ell $ is equivalent to the statement that, if the field is
conditioned to have a saddle point at the origin at level $\ell $, then
almost surely the `arms' of the saddle (i.e.\ the four level lines that
emanate from the saddle point) do not connect the origin to infinity. Since
we can prove that the latter property holds for many fields, we deduce
the continuous differentiability of the density functionals.

Recent work has established that, in many circumstances, the geometry of
Gaussian excursion sets exhibits similar behaviour to that of discrete percolation
models \cite{beffara2017percolation}. In particular, for a wide class of
fields, it has been shown that the connectivity of the excursion sets exhibits
a sharp phase transition at $\ell =0$
\cite{rivera2017critical,Muirhead2018sharp}. Our results can therefore
be compared to what is known, and conjectured, about the analogous density
functionals for discrete percolation models. Consider Bernoulli bond percolation
on the integer lattice, defined by declaring the edges of
$\mathbb{Z}^{d}$ to be open independently with probability $p$ and closed
otherwise (see \cite{grimmett1999percolation} for background on this model).
Let $K_{n}$ denote the number of open clusters that are contained in
$[-n,n]^{d}$. Then it is known (\cite[Chapter 4]{grimmett1999percolation})
that
\begin{displaymath}
\frac{K_{n}}{(2n)^{d}} \to \kappa (p)
\end{displaymath}
as $n\to \infty $, almost surely and in $L^{1}$. This is a direct analogue
of \eqref{e:ns}, and is also proven using an ergodic argument. The smoothness
of $\kappa $ is of interest because it is related to the percolation phase
transition. Specifically, it is conjectured in the physics literature that
$\kappa $ is analytic on $[0,1]\backslash \{p_{c}\}$ and twice but not
three times differentiable at $p_{c}$, where $p_{c}\in (0,1)$ is the critical
probability for the model; this reflects the values of certain `critical
exponents' which are believed to be universal for percolation models (see
\cite[Chapter 9]{grimmett1999percolation}). What has been shown rigorously,
is that, for all $d \ge 2$, $\kappa $ is analytic on $[0,p_{c})$ and smooth
on $(p_{c},1]$, and in the case $d=2$ it is further known that
$\kappa $ is analytic on $(p_{c}, 1]$ and at least twice differentiable
at $p_{c}$ (see \cite[Chapter 4]{grimmett1999percolation}). Somewhat weaker
results have been derived for other percolation models, including the Poisson-Boolean
model and `spread-out' percolation models \cite{Bezuidenhout1998}.

Since the connectivity of the excursion sets of a wide class of planar
Gaussian fields is conjectured, and in some cases known, to undergo a phase
transition at $\ell = 0$ that is analogous to the phase transition at
$p_{c}$ for Bernoulli percolation (see
\cite{beffara2017percolation,beliaev2017russo,rivera2017quasi,rivera2017critical,Muirhead2018sharp}),
it is natural to conjecture that, for such fields, $c_{ES}$ and
$c_{LS}$ are also analytic on $\mathbb{R}\backslash \{0\}$ and twice but
not three times differentiable at $0$. Our proof of the continuous differentiability
of $c_{ES}$ and $c_{LS}$ can be seen as a first step in this direction.

Despite the connections to classical percolation theory, the method we
use to prove differentiability of the density functionals is quite different.
In Bernoulli percolation, the starting point is the equality
\begin{displaymath}
\kappa (p)=\mathbb{E}_{p} \left (\lvert C\rvert ^{-1} \right ),
\end{displaymath}
where $\lvert C\rvert $ is the number of vertices in the open cluster at
the origin. By enumerating clusters, this can be expressed as a power series
in $p$, and the smoothness of $\kappa $ can be deduced from bounds on the
coefficients in terms of connection probabilities for the cluster at the
origin.

This approach does not readily generalise to the setting of Gaussian fields:
whilst it can be shown that
\begin{displaymath}
c_{ES}(\ell )=\mathbb{E}\left (\mathrm{Vol}(C)^{-1}\mathds{1}_{f(0)>
	\ell }\right ) ,
\end{displaymath}
where $\mathrm{Vol}(C)$ is the volume of the component of
$\left \{  x\in \mathbb{R}^{2}:f(x)\geq \ell \right \}  $ containing the origin,
it is not known whether the density of $(\mathrm{Vol}(C), f(0))$ is jointly
continuous (\cite{beliaev2016volume} studies a kind of `ergodic' density
for $\mathrm{Vol}(C)$ at the zero level). Instead, our proof of differentiability
uses an integral representation for $c_{ES}$ and $c_{LS}$ that was developed
in \cite{Beliaev2018Number} (see Theorem~\ref{t:integral equality}), although
we still rely on the decay of certain `connection probabilities' for the
field $f$ conditioned to have a saddle point at the origin. These connections
are the equivalent of `four-arm events' in percolation, which play an important
role in this theory (e.g., in the analysis of noise sensitivity
\cite{garban2014noise}).

Our study of the integral representation for $c_{ES}$ and $c_{LS}$ also
allows us to derive certain montonicity properties of these functionals
(see Propositions~\ref{p:RPW monotonicity}--\ref{p:Isotropic monotonicity});
these results are of independent interest, and are a key input to proving
lower bounds on the variance of the number of excursion/level sets of Gaussian
fields (see Remark~\ref{r:Variance bound}).

\section{Main results}
\label{s:Main results}

Throughout the paper we consider a planar Gaussian field satisfying the
following assumption:
\begin{assumption}%
	\label{a:minimal}
	The Gaussian field $f:\mathbb{R}^{2}\to \mathbb{R}$ is stationary and
	centred with $\mathrm{Var}(f(0))=1$ and satisfies;
	\begin{enumerate}
		\item The covariance function $K\in C^{4+\eta ^{\prime }}$ for some
		$\eta ^{\prime }>0$,
		\item $\nabla ^{2} f(0)$ is a non-degenerate Gaussian vector (it is conventional
		to treat $\nabla ^{2} f(0)$ as a three-dimensional vector, ignoring degeneracy
		due to symmetry),
		\item For any $x\in \mathbb{R}^{2}$, if $f(x)-f(0)$ is a non-degenerate
		Gaussian variable then the Gaussian vector
		$(\nabla f(x),\nabla f(0),f(x)-f(0))$ is non-degenerate.
	\end{enumerate}
\end{assumption}
By Kolmogorov's theorem (\cite[Theorem~1.4.2]{RFG} and
\cite[Theorem~3.17]{hairer2009introduction}),
$K\in C^{4+\eta ^{\prime }}$ implies that
$f\in C^{2+\eta }_{\text{loc}}(\mathbb{R}^{2})$ almost surely for any
$\eta \in (0,\eta ^{\prime }/2)$, and we fix such an $\eta $ for our analysis.

Since $K$ is continuous and positive definite, Bochner's theorem
\cite{bochner1992monotone} states that it is the Fourier transform of a
measure $\mu $ which is known as the spectral measure of the field:
\begin{displaymath}
K(x)=\int _{\mathbb{R}^{2}}e^{it\cdot x}d\mu (t).
\end{displaymath}
We can alternatively state our assumptions in terms of the spectral measure:
$K\in C^{4+\eta ^{\prime }}$ is equivalent to
$\int _{\mathbb{R}^{2}}\lvert t\rvert ^{4+\eta ^{\prime }}d\mu (t)<
\infty $. The second and third parts of Assumption~\ref{a:minimal} are
equivalent to some non-degeneracy of the support of $\mu $ (see Appendix~\ref{a:nondeg}).

We have in mind two important examples of Gaussian fields satisfying Assumption~\ref{a:minimal}:
(1) The Random Plane Wave (RPW), with covariance
$K(x)=J_{0}(\lvert x\rvert )$, where $J_{0}$ is the $0$-th Bessel function,
and spectral measure equal to the normalised Lebesgue measure on the unit
circle; and (2) The Bargmann-Fock (BF) field, with covariance
$K(x)=\exp \left (-\lvert x\rvert ^{2}/2\right )$, and Gaussian spectral
measure. The RPW is a universal model for high energy eigenfunctions of
the Laplacian, see \cite{bogomolny2002percolation} for background. The
BF field can be viewed as a continuous analogue of Bernoulli percolation,
since it has rapid correlation decay and satisfies the FKG inequality,
see \cite{beffara2017percolation} for details and further motivation.

We now formally define the density functionals $c_{ES}$ and $c_{LS}$. Let
$N_{ES}(R,\ell )$ and $N_{LS}(R,\ell )$ denote respectively the number
of components of
$\left \{  x\in \mathbb{R}^{2}:f(x) \geq \ell \right \}  $ and
$\left \{  x\in \mathbb{R}^{2}:f(x)=\ell \right \}  $ contained in
$B(R)$ (i.e.\ the components which intersect $B(R)$ but not
$\mathbb{R}^{2}\backslash B(R)$). Then the following asymptotic laws are
known to hold:

\begin{theorem}[\cite{SodinNazarov2015asymptotic,kurlberg2017variation,Beliaev2018Number}]%
	\label{t:main level}
	Let $f$ be a Gaussian field satisfying Assumption~\ref{a:minimal}. For
	each $\ell \in \mathbb{R}$, there exist
	$c_{ES}(\ell ),c_{LS}(\ell )\geq 0$ such that
	\begin{align*}
	\mathbb{E}\left [N_{ES}(R,\ell )\right ]=c_{ES}(\ell )\cdot \pi R^{2}+O(R)
	\quad \text{and} \quad \mathbb{E}\left [N_{LS}(R,\ell )\right ]=c_{LS}(
	\ell )\cdot \pi R^{2}+O(R)
	\end{align*}
	as $R\to \infty $. The constants implied by the $O(\cdot )$ notation are
	independent of $\ell $. If $f$ is also ergodic, then
	\begin{align*}
	\frac{N_{ES}(R,\ell )}{\pi R^{2}} \rightarrow c_{ES}(\ell ) \quad
	\text{and} \quad \frac{N_{LS}(R,\ell )}{\pi R^{2}}\rightarrow c_{LS}(
	\ell )
	\end{align*}
	as $R\to \infty $, almost surely and in $L^{1}$.
\end{theorem}

\begin{remark}
	The notation in \cite{kurlberg2017variation} and elsewhere is slightly
	different: $c_{LS}$ in the present paper is denoted $c_{NS}$ in some previous
	papers.
\end{remark}

In \cite{Beliaev2018Number} a representation of $c_{ES}$ and
$c_{LS}$ was given in terms of the densities of certain types of critical
points. To state this we introduce upper/lower connected saddle points.

\begin{definition}%
	\label{d:lower connected aperiodic}
	Let $x_{0}$ be a saddle point of a $C^{2}$ function
	$g:\mathbb{R}^{2}\to \mathbb{R}$ such that there are no other critical
	points at the same level as $x_{0}$ (that is, if $x_{1}$ is another critical
	point of $g$, then $g(x_{1})\neq g(x_{0})$). We say that $x_{0}$ is
	\emph{upper connected} if it is in the closure of only one component of
	$\left \{  x\in \mathbb{R}^{2}:g(x)>g(x_{0})\right \}  $. We say that
	$x_{0}$ is \emph{lower connected} if it is in the closure of only one component
	of $\left \{  x\in \mathbb{R}^{2}:g(x)<g(x_{0})\right \}  $.
\end{definition}

Interestingly, this definition was used as far back as 1870, by Maxwell
\cite{maxwell1870hills}, to understand topographical properties of landscapes.

It was shown in \cite{dennis2007nodal,cheng2015expected} that the expected
number of local maxima, local minima or saddle points of a Gaussian field
with height in a certain range can be expressed as the integral of an explicit
continuous density function over the height range. In
\cite{Beliaev2018Number} this result was extended to upper and lower connected
saddle points without explicitly computing the corresponding density functions:

\begin{proposition}[{\cite[Proposition 1.8]{Beliaev2018Number}}]%
	\label{p:density existence}
	Let $f$ be a Gaussian field satisfying Assumption~\ref{a:minimal}. Then
	there exist non-negative functions
	$p_{m^{+}},p_{m^{-}},p_{s^{+}},p_{s^{-}},p_{s}\in L^{1}(\mathbb{R})$ such
	that the following holds. Let $\Omega \subset \mathbb{R}^{2}$ be compact
	and $\partial \Omega $ have finite Hausdorff-1 measure. Let
	$\ell \in \mathbb{R}$ and let $N_{m^{+}}(\ell )$, $N_{m^{-}}(\ell )$,
	$N_{s^{+}}(\ell )$, $N_{s^{-}}(\ell )$ and $N_{s}(\ell )$ denote the number
	of local maxima, local minima, upper connected saddles, lower connected
	saddles and saddles of $f$ in $\Omega $ with level above $\ell $ respectively.
	Then
	\begin{displaymath}
	\mathbb{E}\left [N_{h}(\ell )\right ]=\text{\emph{Area}}(\Omega )\int _{\ell }^{\infty }p_{h}(x) \, dx
	\end{displaymath}
	for $h=m^{+},m^{-},s^{+},s^{-},s$. Furthermore, these functions can be
	chosen to satisfy the relations $p_{m^{+}}(x) = p_{m^{-}}(-x)$,
	$p_{s^{+}}(x) = p_{s^{-}}(-x)$ and $p_{s^{-}}+p_{s^{+}}=p_{s}$, and such
	that $p_{m^{+}}$, $p_{m^{-}}$ and $p_{s}$ are continuous.%
\end{proposition}

We can now state the main result of \cite{Beliaev2018Number}, characterising
$c_{ES}$ and $c_{LS}$ in terms of the densities in Proposition~\ref{p:density existence}:

\begin{theorem}[{\cite[Theorem 1.9]{Beliaev2018Number}}]%
	\label{t:integral equality}
	Let $f$ be a Gaussian field satisfying Assumption~\ref{a:minimal}, and
	let $p_{m^{+}}$, $p_{m^{-}}$, $p_{s^{+}}$, $p_{s^{-}}$ denote the densities
	specified in Proposition~\ref{p:density existence}. Then
	\begin{equation}
	\label{e:integral equality2}
	c_{ES}(\ell )=\int _{\ell }^{\infty }p_{m^{+}}(x)-p_{s^{-}}(x) \, dx
	\end{equation}
	and
	\begin{equation}
	\label{e:integral equality1}
	c_{LS}(\ell )=\int _{\ell }^{\infty }p_{m^{+}}(x)-p_{s^{-}}(x)+p_{s^{+}}(x)-p_{m^{-}}(x)
	\, dx ,
	\end{equation}
	and hence $c_{ES}$ and $c_{LS}$ are absolutely continuous.
\end{theorem}

One of the motivations for Theorem~\ref{t:integral equality} was to provide
a tool with which to study the excursion/level set densities: since
$p_{m^{+}}$, $p_{m^{-}}$, and $p_{s}=p_{s^{+}}+p_{s^{-}}$ are explicitly
known for a wide class of fields, by establishing simple properties of
$p_{s^{-}}$ we can deduce results for $c_{ES}$ and $c_{LS}$. We expand
upon this method in this paper. Specifically, we consider the function
\begin{align}
\label{e:Main relation}
p_{s^{-}}^{*}(\ell )&:=p_{s}(\ell ) \, \mathbb{P}\left (\tilde{f}_{\ell }\text{ has a lower connected saddle point at the origin}\right ) ,
\end{align}
where $\tilde{f}_{\ell }$ is the field $f$ conditioned to have a saddle point
at the origin at level $\ell $ (in the sense of Palm distributions; see
Lemma~\ref{l:conditional distribution} for a formal definition). Under
mild conditions we show that $p_{s^{-}}^{*}$ defines a version of
$p_{s^{-}}$ (recall that the latter is defined only up to null sets). By
studying $\tilde{f}_{\ell }$ we are able to deduce properties of
$p^{\ast }_{s^{-}}$, and hence of $c_{ES}$ and~$c_{LS}$.

\subsection{Differentiability}
\label{ss:Differentiability}

Our first set of results concerns the differentiability of $c_{ES}$ and
$c_{LS}$. Let us begin by detailing the necessary assumptions on $f$.

\begin{assumption}%
	\label{a:non-degenerate gradient}
	For all $t\in \mathbb{R}^{2}\backslash \{0\}$,
	\begin{equation*}
	\mathrm{Cov}\left ( \left (f(t), \nabla f(t)\right ) \;\middle |f(0),
	\nabla f(0),\nabla ^{2} f(0)\right )
	\end{equation*}
	is non-degenerate (i.e.\ this $3\times 3$ matrix has non-zero determinant).
\end{assumption}

\begin{assumption}%
	\label{a:regularity}
	There exist $c,\nu >0$ such that, for all $|t| \ge 1$,
	\begin{displaymath}
	\max _{\lvert k\rvert \le 3} \, \left \lvert \partial ^{k} K(t)
	\right \rvert \leq c\lvert t\rvert ^{-(1+\nu )}.
	\end{displaymath}
	Moreover, there exists a neighbourhood $V$ of the origin on which the spectral
	measure $\mu $ has density $\rho $ with respect to the Lebesgue measure
	and $\inf _{V}\rho >0$.
\end{assumption}

\begin{assumption}%
	\label{a:Arm decay}
	For $0<r<R$, let $\mathrm{Arm}_{\ell }(r,R)$ denote the `one-arm event' that
	there exists a component of $\{f\geq \ell \}$ which intersects both
	$\partial B(r)$ and $\partial B(R)$. Then there exist
	$c_{1},c_{2}>0$ such that for any $1 <r<R$
	\begin{equation}
	\label{e:arm decay}
	\mathbb{P}\left (f\in \mathrm{Arm}_{0}(r,R)\right )\leq c_{1}(r/R)^{c_{2}}.
	\end{equation}
\end{assumption}

Assumption~\ref{a:non-degenerate gradient} is extremely mild; it is satisfied
whenever the support of the spectral measure $\mu $ is not too degenerate.
It is sufficient for this support to contain an open set or an ellipse/circle
(Lemma~\ref{a:nondegen2}), so in particular, it holds for the RPW and the
BF field.

Assumptions~\ref{a:regularity} and~\ref{a:Arm decay} are somewhat more
restrictive. Assumption~\ref{a:regularity} holds for any smooth field with
sufficiently nice correlation decay, and in particular holds for the BF
field, but it does not hold for the RPW (whose correlations decay only
as $|t|^{-1/2}$). It also implies Assumption~\ref{a:non-degenerate gradient},
by the previous remark.

Assumption~\ref{a:Arm decay} relates to the conjectured properties of the
`percolation universality class', and has been shown to hold for a wide
class of fields that includes the BF field
\cite{beffara2017percolation,rivera2017critical}. Moreover it is strongly
believed to hold for the RPW. We state our results directly in terms of
one-arm decay as it is likely that these bounds will be extended to more
fields over time.

Our first main result is that $c_{ES}$ and $c_{LS}$ are continuously differentiable
under the above assumptions:

\begin{theorem}%
	\label{t:differentiability of c_{LS}}
	Suppose $f$ is a Gaussian field satisfying Assumptions~\ref{a:minimal}
	and~\ref{a:regularity}--\ref{a:Arm decay} (e.g.\ the Bargmann-Fock field).
	Then $c_{ES}$ and $c_{LS}$ are continuously differentiable on
	$\mathbb{R}$. In other words, the functions $p_{s^{-}}$ and
	$p_{s^{+}}$ defined in Proposition~\ref{p:density existence} can be chosen
	to be continuous, and
	\begin{equation*}
	c'_{ES}(\ell ) = -p_{m^{+}}(\ell )+p_{s^{-}}(\ell )
	\end{equation*}
	and
	\begin{equation*}
	c'_{LS}(\ell )= - p_{m^{+}}(\ell ) + p_{s^{-}}(\ell ) - p_{s^{+}}(
	\ell ) + p_{m^{-}}(\ell ) .
	\end{equation*}
\end{theorem}

We emphasise that Theorem~\ref{t:differentiability of c_{LS}} applies to
a wide class of fields, including the important case of the BF field, but
does not apply to the RPW as stated (although we believe the conclusion
to be true).

\subsubsection{Four-arm saddle points}

Theorem~\ref{t:differentiability of c_{LS}} follows from a more general
result establishing that, under very mild conditions, the continuous differentiability
of $c_{ES}$ and $c_{LS}$ is implied by the decay of certain connection
probabilities involving `four-arm saddles'.

Let $D\subset \mathbb{R}^{2}$ be a simply connected domain with piecewise
$C^{1}$ boundary and let $x_{0}\in D$ be a saddle point of
$g\in C^{2}(\mathbb{R}^{2})$ such that $g$ has no other critical points
at the same level as $x_{0}$. We say that $x_{0}$ is
\textit{four-arm in $D$} if it is in the closure of two components of
$\left \{  x\in D:g(x)>g(x_{0})\right \}  $ and two components of
$\left \{  x\in D:g(x)<g(x_{0})\right \}  $ (see Figure~\ref{Fig_1}(a)); intuitively,
a saddle point is four-arm in $D$ if we cannot tell whether it is upper
or lower connected by looking at the values of $g$ in $D$. A saddle point
$x_{0}$ is said to be \textit{infinite four-arm} if it is in the closure
of two components of
$\left \{  x\in \mathbb{R}^{2}:g(x)>g(x_{0})\right \}  $ and two components
of $\left \{  x\in \mathbb{R}^{2}:g(x)<g(x_{0})\right \}  $ (see Figure~\ref{Fig_1}(b)).
As mentioned in Section~\ref{s:introduction}, four-arm saddle points are
analogous to four-arm events for percolation models.

\begin{figure}[ht!]
	\centering
	\begin{subfigure}[t]{0.45\textwidth}
		\resizebox{\linewidth}{!}{
			\begin{tikzpicture}[scale=0.04]
			\draw plot[smooth, tension=.7] coordinates {(51,44) (45,42) (42,39) (39,35) (38,32) (36,27) (33,24) (29,21) (22,20) (15,18) (9,14) (5,9) (4,4) (2,2) (0,0) (-4,1) (-7,3) (-11,4) (-16,8) (-21,7) (-28,10) (-32,10) (-37,9) (-46,11) (-51,2) (-54,-4) (-57,-11) (-52,-22) (-44,-27) (-44,-37) (-38,-41) (-31,-39) (-25,-31) (-20,-19) (-13,-11) (-6,-4) (0,0) (8,-3) (16,-9) (24,-15) (31,-23) (40,-28) (49,-33)};
			\draw[ dashed] (8,-9) circle (40);
			\draw[fill] (0,0) circle (5pt);
			\node at (-12,-1) {$-$};
			\node at (10,2) {$-$};
			\node at (-2,8) {$+$};
			\node at (1,-6) {$+$};
			\node[right] at (44,33) {$\left\{g=g(x_0)\right\}$};
			\node[right] at (48,-9) {$B(R)$};
			\clip plot[smooth, tension=.7] coordinates {(49,-33)(40,-28) (31,-23)(24,-15) (16,-9) (8,-3) (0,0) (-6,-4) (-13,-11) (-20,-19) (-25,-31) (-31,-39) (-38,-41) (-44,-37) (-44,-27) (-52,-22) (-57,-11) (-54,-4) (-51,2) (-46,11) (-37,9) (-32,10)(-28,10) (-21,7) (-16,8) (-11,4)(-7,3)(-4,1) (0,0) (2,2) (4,4) (5,9) (9,14)(15,18) (22,20) (29,21) (33,24) (36,27) (38,32)(39,35) (42,39) (45,42) (51,44) (-55,43)(-66,-10)(-46,-58)(20,-61)};
			\fill [gray, opacity=0.5] (8,-9) circle (40);
			\end{tikzpicture}}
		\caption{\centering An upper connected saddle point that is four-arm in $B(R)$.}%
		\label{Fig_1a}
	\end{subfigure}
	\begin{subfigure}[t]{0.45\textwidth}
		\resizebox{\linewidth}{!}{
			\begin{tikzpicture}[scale=0.05]
			\draw[white, fill=gray, opacity=0.5] plot[smooth, tension=.6] coordinates {(-27,17) (-20,10) (-11,6) (-3,2) (5,-3) (9,-8) (5,-13) (3,-17) (-2,-19) (-10,-23) (-16,-30) (-21,-29) (-27,-32) (-31,-36.5)(-37.5,-37.5) (-40,-22) (-40,-10) (-39,0) (-34,10) (-30,15)(-27,17)};
			\draw plot[smooth, tension=.6] coordinates {(-27,17) (-20,10) (-11,6) (-3,2) (5,-3) (9,-8) (5,-13) (3,-17) (-2,-19) (-10,-23) (-16,-30) (-21,-29) (-27,-32) (-31,-37)};
			\draw[white, fill=gray,opacity=0.5] plot[smooth, tension=.6] coordinates {(45,20) (41,21) (34,19) (28,13) (26,8) (25,0) (20,-2) (13,-5) (9,-8) (11,-11) (13,-16) (16,-18) (21,-21) (28,-23) (37,-28) (46,-35)(51,-36) (55,-36) (58,-30)(58,-21) (57,-14) (57,-5) (56,2) (55,7) (53,13) (50,17) (45,20)};
			\draw plot[smooth, tension=.6] coordinates {(45,20) (41,21) (34,19) (28,13) (26,8) (25,0) (20,-2) (13,-5) (9,-8) (11,-11) (13,-16) (16,-18) (21,-21) (28,-23) (37,-28) (46,-35)};
			\node at (-12,-9) {$+$};
			\node at (35,-10) {$+$};
			\node at (11,5) {$-$};
			\node at (8,-26) {$-$};
			\draw[->] (-29,19) -- (-34,24);
			\node [above left] at (-34,24) {$\infty$};
			\draw[->] (-33,-39) -- (-38,-44);
			\node [below left] at (-38,-44) {$\infty$};
			\draw[->] (49,-37) -- (54,-42);
			\node [below right] at (54,-42) {$\infty$};
			\draw[->] (47,20) -- (52,25);
			\node [above right] at (52,25) {$\infty$};
			\draw[fill] (9,-8)circle (5pt);
			\node[above, left] at (40,25) {$\left\{g=g(x_0)\right\}$};
			\end{tikzpicture}}
		\caption{\centering An infinite-four-arm saddle point.}%
		\label{Fig_1b}
	\end{subfigure}
	\caption{}
	\label{Fig_1}%
\end{figure}

Recall the conditional field $\tilde{f}_{\ell }$ (to be formally defined
in Lemma~\ref{l:conditional distribution}) and the functions
$p_{s^{-}}^{\ast }$ and $p_{s^{-}}$ defined in~\eqref{e:Main relation} and
Proposition~\ref{p:density existence} respectively.

\begin{samepage}
	\begin{theorem}%
		\label{t:Differentiability equivalence}
		Let $f$ be a Gaussian field satisfying Assumptions~\ref{a:minimal} and~\ref{a:non-degenerate gradient}.
		Then $p_{s^{-}}^{\ast }= p_{s^{-}}$ almost everywhere. Moreover, let
		$a<b$ and suppose that for all $\ell \in (a,b)$
		\begin{equation}
		\label{e:fourarm}
		\mathbb{P} \left (\tilde{f}_{\ell
		}\text{\textup{ has an infinite four-arm saddle at the origin}}\right )=0 .
		\end{equation}
		Then $p_{s^{-}}^{\ast }|_{(a,b)}$ is continuous, and so $c_{ES}$ and
		$c_{LS}$ are continuously differentiable on $(a,b)$.
	\end{theorem}
\end{samepage}

Theorem~\ref{t:differentiability of c_{LS}} follows from Theorem~\ref{t:Differentiability equivalence}
once we verify condition \eqref{e:fourarm} under Assumptions~\ref{a:regularity}
and~\ref{a:Arm decay}. To do so, we use Assumption~\ref{a:regularity} and
a Cameron-Martin argument to treat the conditional field
$\tilde{f}_{\ell }$ away from the origin as a perturbation of the unconditioned
field $f$. We then use Assumption~\ref{a:Arm decay} to bound the relevant
connection probabilities for the unconditioned field.

As a corollary of Theorem~\ref{t:Differentiability equivalence} (actually
of its proof), we deduce a bound on the number of saddle points of a Gaussian
field that are four-arm inside a ball and whose level lies in a narrow
range. This improves a bound that was previously established in
\cite{Beliaev2018Number}, and is also a key ingredient in proving lower
bounds on the variance of the number of excursion/level set components
(see Remark~\ref{r:Variance bound}).

\begin{corollary}%
	\label{c:Four arm}
	Let $f$ be a Gaussian field satisfying all the assumptions of Theorem~\ref{t:Differentiability equivalence}.
	Then there exists a function $\delta _{R}\to 0$ as $R\to \infty $ and a
	constant $c>0$ such that, for each $R>1$ and
	$a\leq a_{R}\leq b_{R}\leq b$,
	\begin{displaymath}
	\mathbb{E}\left (N_{\mathrm{4\mhyphen arm}}\left (R,\left [a_{R},b_{R}\right ]
	\right )\right )\leq c\min \left \{  \delta _{R} R^{2}\left (b_{R}-a_{R}
	\right ),R\right \}
	\end{displaymath}
	where $N_{\mathrm{4\mhyphen arm}}\left (R,\left [a_{R},b_{R}\right ]\right )$ is the
	number of saddle points of $f$ which are four-arm in $B(R)$ and have level
	in $[a_{R},b_{R}]$.
\end{corollary}

\begin{remark}
	In \cite{Beliaev2018Number} it was shown that
	$\mathbb{E}\left (N_{\mathrm{4\mhyphen arm}}(R)\right )=O(R) $; Corollary~\ref{c:Four arm}
	supersedes this bound whenever $b_{R}-a_{R}=O\left (R^{-1}\right )$. It
	is possible to improve the conclusion of Corollary~\ref{c:Four arm} further
	by imposing stronger assumptions on the field. For example, suppose we
	assume the exponential decay of arm probabilities at non-zero levels: for
	some $\ell ^{\ast }> 0$ and $\delta \in (0,1)$, there exist
	$c_{1}, c_{2} > 0$ such that
	\begin{equation}
	\label{e:Exponential decay assumption}
	\mathbb{P}\left (f\in \mathrm{Arm}_{\ell ^{*}}(\delta R,R)\right )
	\leq c_{1} e^{-c_{2}R} .
	\end{equation}
	Then for any $a>\ell ^{*}$ (or $b<-\ell ^{*}$), it is possible to prove
	that there exists $c>0$ such that
	\begin{displaymath}
	\mathbb{E}\left (N_{\mathrm{4\mhyphen arm}}\left (R,\left [a_{R},b_{R}\right ]
	\right )\right )\leq c\min \left \{  R\log (R)(b_{R}-a_{R}),R\right \}  .
	\end{displaymath}
	In \cite{Muirhead2018sharp}, it is shown that a wide class of fields satisfy \eqref{e:Exponential decay assumption}, so this assumption is reasonable.
	We do not prove this result formally here because Proposition~\ref{c:Four arm}
	is simpler to prove, holds for a wider class of fields, and suffices for
	its intended purpose (see Remark~\ref{r:Variance bound}).
	
	In light of Theorem~\ref{t:Differentiability equivalence}, the fact that \eqref{e:Exponential decay assumption} is expected to hold for a wide class
	of fields also suggests it should be much easier to prove the differentiability
	of $c_{ES}$ away from zero, since the probability of four-arm saddles in
	$B(R)$ should decay exponentially at non-zero levels.
\end{remark}

\subsubsection{The positivity of the level set density}

In order for Theorem~\ref{t:main level} to describe the leading-order asymptotics
of the number of excursion/level set components, it is crucial that the
limiting constants are positive; if they are not, then it can be shown
that $f$ almost surely has no compact excursion/level sets. One nice consequence
of the differentiability of $c_{ES}$ and $c_{LS}$ is that it gives a new,
short proof of their positivity in the delicate case $\ell = 0$:

\begin{proposition}%
	\label{p:positivity of c_{LS}}
	Let $f$ be a Gaussian field satisfying Assumption~\ref{a:minimal}. Suppose
	either $c_{ES}$ or $c_{LS}$ is differentiable at $0$. Then
	$c_{ES}(0) > 0$ and $c_{LS}(0)>0$.
\end{proposition}

The positivity of $c_{ES}(0)$ and $c_{LS}(0)$ are already known quite generally
(\cite{SodinNazarov2015asymptotic,ingremeau2016lower} give a variety of
sufficient conditions, whose union can be checked to exhaust fields satisfying
Assumptions~\ref{a:minimal} and \ref{a:non-degenerate gradient}). We restate
this result because it uses a very different method of proof; in particular,
it does not rely on the `barrier method'.

The positivity of $c_{ES}(\ell )$ and $c_{LS}(\ell )$ for $\ell > 0$ is
simpler to establish, even without differentiability (see
\cite{Beliaev2018Number}). On the other hand, our arguments apparently
do not extend to $c_{ES}(\ell )$ for $\ell < 0$ (although this case can
still be treated via the `barrier method'; see Lemma~\ref{p:Positivity RPW}).

\subsubsection{Fields outside the `percolation universality class'}

Although in general we expect the properties of $c_{ES}$ and
$c_{LS}$ to match those of the analogous density functional
$\kappa $ from percolation theory, this can fail for fields outside the
`percolation universality class'.

To demonstrate this, we consider the one non-trivial case in which
$c_{ES}$ and $c_{LS}$ are explicitly known: fields with spectral measure
supported on four or five points (see
\cite[Proposition~1.20]{Beliaev2018Number}). In
\cite{Beliaev2018Number} it was shown that, in the `five point case',
$c_{ES}$ and $c_{LS}$ are smooth everywhere, whereas in the `four point
case', $c_{ES}$ and $c_{LS}$ are smooth everywhere except zero, at which
point they are continuous but not differentiable (see Figure~\ref{Fig_2}).
Hence, in both cases, the smoothness of $c_{ES}$ and $c_{LS}$ differs from
the conjectured properties of $\kappa $ (and in different ways). However,
these fields do not fall within the scope of the present paper (they do
not satisfy Assumptions~\ref{a:minimal} and~\ref{a:non-degenerate gradient}).
Moreover, being periodic, their large-scale properties cannot be expected
to match those of Bernoulli percolation.

\begin{figure}[h!]
	\centering
	\begin{subfigure}[t]{0.45\textwidth}
		\resizebox{\linewidth}{!}{
			\includegraphics[scale=0.5]{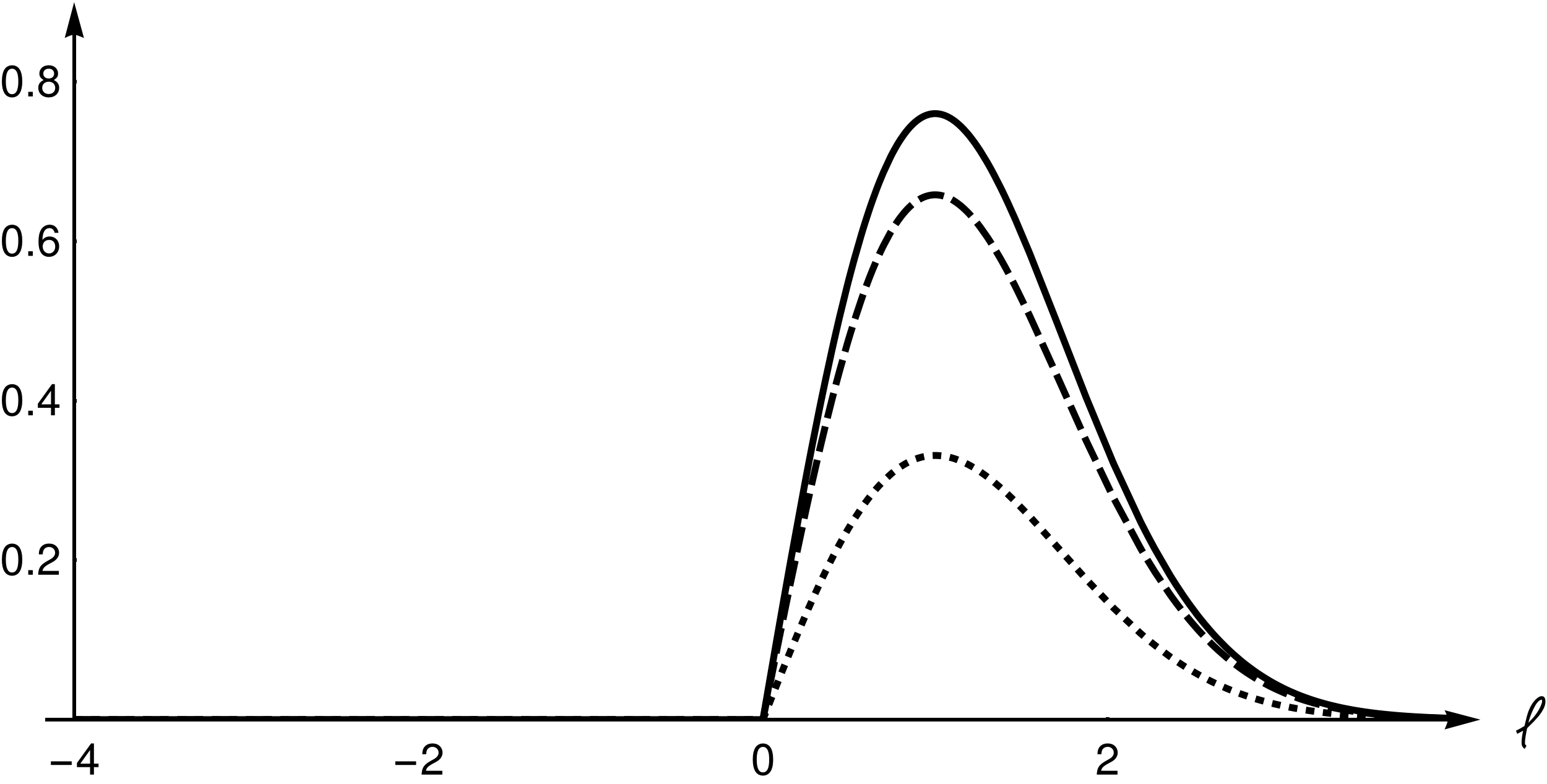}}
		\caption{}%
		\label{Fig_2a}
	\end{subfigure}
	\begin{subfigure}[t]{0.45\textwidth}
		\resizebox{\linewidth}{!}{
			\includegraphics[scale=0.5]{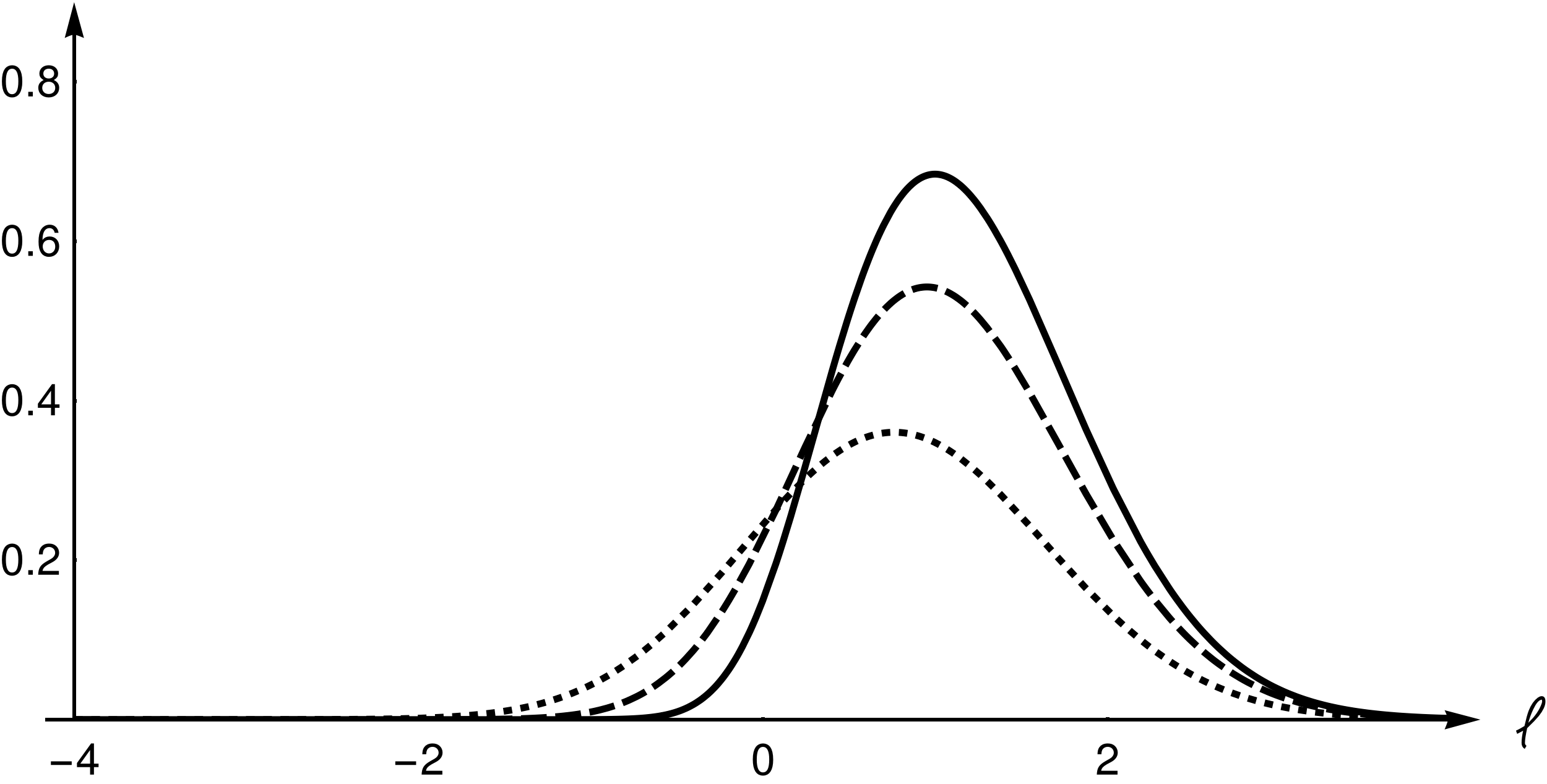}}
		\caption{}%
		\label{Fig_2b}
	\end{subfigure}
	\caption{The functional $c_{ES}(\ell )$ for fields with spectral measure
		supported on four (left) or five (right) points. The different lines correspond
		to different measures.}%
	\label{Fig_2}
\end{figure}

On the other hand, the non-differentiability of $c_{ES}$ at zero in the
`four point case' does reflect a different kind of phase transition: for
$\ell \leq 0$, $\{f\geq \ell \}$ almost surely has no bounded components
($c_{ES}(\ell ) = 0$), whereas for $\ell >0$, the number of components
is of order $R^{2}$ ($c_{ES}(\ell )> 0$); see Figure~\ref{Fig_3}. Moreover,
a Gaussian field in the `five point case' can be represented as a field
in the `four point case' plus an independent Gaussian level shift. Hence
the same phase transition occurs, although it does so at a random level
and so the discontinuity is averaged out.

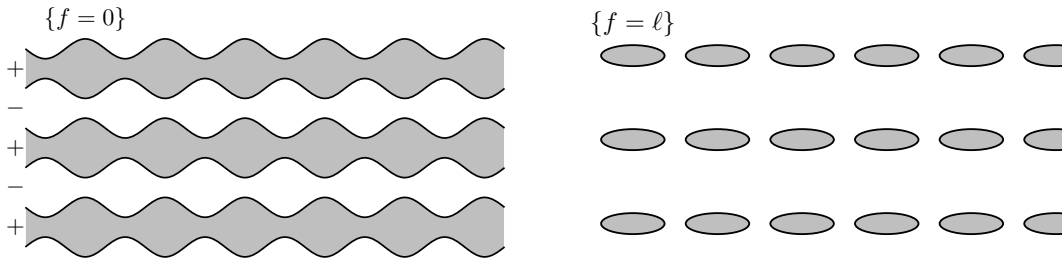
\begin{figure}[h!]
	\centering
	\begin{subfigure}[t]{0.45\textwidth}
		\resizebox{\linewidth}{!}{
			\begin{tikzpicture}[scale=0.16]
			\foreach \y in {0,8,16}{
				\draw[white] [fill=gray,opacity=0.5]
				(0,0+\y) sin (2,1+\y) cos (4,0+\y) sin (6,-1+\y) cos (8,0+\y)
				sin (10,1+\y) cos (12,0+\y) sin (14,-1+\y) cos (16,0+\y) sin (18,1+\y)
				cos (20,0+\y) sin (22,-1+\y) cos (24,0+\y)
				sin (26,1+\y) cos (28,0+\y) sin (30,-1+\y) cos (32,0+\y)
				sin (34,1+\y) cos (36,0+\y) sin (38,-1+\y) cos (40,0+\y)
				sin (42,1+\y) cos (44,0+\y) sin (46,-1+\y) cos (48,0+\y)--(48,4+\y)
				
				sin (46,1+\y+4)cos (44,0+\y+4)sin (42,-1+\y+4)cos (40,0+\y+4)
				sin (38,1+\y+4) cos (36,0+\y+4)sin (34,-1+\y+4)cos (32,0+\y+4)
				sin (30,1+\y+4) cos (28,0+\y+4) sin (26,-1+\y+4) cos (24,0+\y+4)
				sin (22,1+\y+4)cos (20,0+\y+4) sin (18,-1+\y+4)cos (16,0+\y+4)
				sin (14,1+\y+4) cos (12,0+\y+4) sin (10,-1+\y+4) cos (8,0+\y+4)
				sin (6,1+\y+4) cos (4,0+\y+4) sin (2,-1+\y+4)cos (0,\y+4)--(0,\y);
				\draw[ thick]
				(0,\y+4) sin (2,-1+\y+4) cos (4,0+\y+4) sin (6,1+\y+4) cos (8,0+\y+4)
				sin (10,-1+\y+4) cos (12,0+\y+4) sin (14,1+\y+4) cos (16,0+\y+4)
				sin (18,-1+\y+4) cos (20,0+\y+4) sin (22,1+\y+4) cos (24,0+\y+4)
				sin (26,-1+\y+4) cos (28,0+\y+4) sin (30,1+\y+4) cos (32,0+\y+4)
				sin (34,-1+\y+4) cos (36,0+\y+4) sin (38,1+\y+4) cos (40,0+\y+4)
				sin (42,-1+\y+4) cos (44,0+\y+4) sin (46,1+\y+4) cos (48,0+\y+4);
				\draw[ thick]
				(0,0+\y) sin (2,1+\y) cos (4,0+\y) sin (6,-1+\y) cos (8,0+\y)
				sin (10,1+\y) cos (12,0+\y) sin (14,-1+\y) cos (16,0+\y) sin (18,1+\y)
				cos (20,0+\y) sin (22,-1+\y) cos (24,0+\y)
				sin (26,1+\y) cos (28,0+\y) sin (30,-1+\y) cos (32,0+\y)
				sin (34,1+\y) cos (36,0+\y) sin (38,-1+\y) cos (40,0+\y)
				sin (42,1+\y) cos (44,0+\y) sin (46,-1+\y) cos (48,0+\y);
			}
			\node[above] at (6,21) {$\{f=0\}$};
			\node at (-1,18) {$+$};
			\node at (-1,14) {$-$};
			\node at (-1,10) {$+$};
			\node at (-1,6) {$-$};
			\node at (-1,2) {$+$};
			\clip (-3,25) rectangle (49.5,-4);
			\end{tikzpicture}}
	\end{subfigure}
	\begin{subfigure}[t]{0.45\textwidth}
		\resizebox{\linewidth}{!}{
			\begin{tikzpicture}[scale=0.16]
			\clip (-9,23) rectangle (40.5,-6);
			\foreach \y in {0,8,16} {
				\foreach \x in {0,8,...,40}{
					\fill [fill=gray,opacity=0.5] (\x,\y) ellipse (3 and 1);
					\draw[thick] (\x,\y) ellipse (3 and 1);
			}}
			\node[above] at (0,17) {$\{f=\ell\}$};
			\end{tikzpicture}}
	\end{subfigure}
	\caption{Stylised excursion sets for fields with spectral measure supported
		on four points, at the zero level (left) and at a positive level (right).}%
	\label{Fig_3}
\end{figure}

\subsection{Monotonicity}
\label{ss:Monotonicity}

We next consider monotonicity properties of $c_{ES}$ and $c_{LS}$. We begin
by analysing the ratio
\begin{equation*}
p^{\ast }_{s^{-}}(\ell )/p_{s}(\ell ) = \mathbb{P}\left (\tilde{f}_{\ell }\text{ has a lower connected saddle point at the origin}\right ) ,
\end{equation*}
which we intuitively expect to be non-decreasing: if we condition on the
origin being a saddle point at increasing heights, it seems more likely
that it should be lower connected. This can be made rigorous under some
additional assumptions, and allows us to deduce regions on which
$c_{ES}(\ell )$ and $c_{LS}(\ell )$ are monotone.

\begin{assumption}%
	\label{a:Monotonicity}
	The field $f$ is isotropic (i.e.\ its law is invariant under rotations)
	and hence its covariance function can be expressed as
	$K(x)=k(\lvert x\rvert ^{2})$. Then
	\begin{equation}
	\label{e:mon3}
	\chi :=\frac{-k^{\prime }(0)}{\sqrt{k^{\prime \prime }(0)}}\geq 1.
	\end{equation}
	Furthermore, the Gaussian vector $(f(0),\nabla ^{2} f(0))$ is non-degenerate,
	and for all $x\in \mathbb{R}^{2}$,
	\begin{gather}
	\label{e:mon}
	\mathbb{E}\left (f(x) \, \middle | \, f(0)=0,\nabla ^{2} f(0)=
	\begin{pmatrix}
	1 &0
	\\
	0 &0
	\end{pmatrix}
	\right )\geq 0,
	\\
	\label{e:mon2}
	\mathbb{E}\left (f(x) \, \middle | \, f(0)=1,\nabla ^{2} f(0)=
	\begin{pmatrix}
	0 &0
	\\
	0 &0
	\end{pmatrix}
	\right )\leq 1.
	\end{gather}
\end{assumption}
The parameter $\chi $ is used in \cite{cheng2015expected} to parameterise
the density of eigenvalues of $\nabla ^{2} f$ at critical points and is
shown to take values in $(0,\sqrt{2}]$. We note that \eqref{e:mon3} can
be replaced by a weaker condition (see Remark~\ref{r:general monotonicity conditions})
however we do not state the general condition here as \eqref{e:mon3} is
much simpler to verify.

In Section~\ref{ss:mongen} we explain how \eqref{e:mon} and \eqref{e:mon2} can be translated into explicit properties of the conditional
field $\tilde{f}_{\ell }$. We can also give equivalent versions of \eqref{e:mon} and \eqref{e:mon2} that are easier to check in practice.
If we rescale the domain of $f$ so that $k^{\prime }(0)=-1$ (note that this
does not affect the value of $\chi $), then it can be shown by Gaussian
regression that \eqref{e:mon} is equivalent to
\begin{displaymath}
\forall x\in \mathbb{R}^{2},\qquad \left (k\left (\lvert x\rvert ^{2}
\right ) + k^{\prime }\left (\lvert x\rvert ^{2}\right )\right )k^{
	\prime \prime }(0) + \left (x_{1}^{2}\left (3k^{\prime \prime }(0)-1
\right ) + x_{2}^{2}\left (1-k^{\prime \prime }(0)\right )\right ) k^{
	\prime \prime }\left (\lvert x\rvert ^{2}\right ) \geq 0
\end{displaymath}
and \eqref{e:mon2} is equivalent to
\begin{displaymath}
\forall y\geq 0,\qquad
\frac{2k^{\prime \prime }(0)k\left (y\right )+yk^{\prime \prime }\left (y\right )+k^{\prime }\left (y\right )}{2k^{\prime \prime }(0)-1}
\leq 1.
\end{displaymath}
From this, it can be verified that specific fields satisfy Assumption~\ref{a:Monotonicity},
including the BF field. The RPW does not satisfy Assumption~\ref{a:Monotonicity};
in this case $\chi =\sqrt{2}$ but $(f(0),\nabla ^{2} f(0))$ is non-degenerate
so the conditional expectations in \eqref{e:mon} and \eqref{e:mon2} are
not defined. However we are able to prove the monotonicity of
$p^{\ast }_{s^{-}}/p_{s}$ in this case too:

\begin{theorem}%
	\label{t:Monotonicity}
	Let $f$ be the Random Plane Wave or a field satisfying Assumptions~\ref{a:minimal},~\ref{a:non-degenerate gradient}
	and~\ref{a:Monotonicity} (e.g.\ the Bargmann-Fock field). Then
	$p_{s^{-}}^{*}(\ell )/p_{s}(\ell )$ is non-decreasing in $\ell $.
\end{theorem}

Given the definition of $p_{s^{-}}^{*}$, we will show that Theorem~\ref{t:Monotonicity}
is a consequence of $\tilde{f}_{\ell }-\ell $ being stochastically decreasing
in $\ell $. Our proof of the latter fact differs for the RPW and for fields
satisfying Assumption~\ref{a:Monotonicity} (in the former case it is somewhat
simpler, because of the degeneracies in the RPW; see e.g.
\cite{wigman2012nodal}).

The monotonicity of $p_{s^{-}}^{*}/p_{s}$ has some implications for the
smoothness of $c_{ES}$ and $c_{LS}$:
\begin{corollary}%
	\label{c:Twice differentiable}
	Let $f$ satisfy the conditions of Theorem~\ref{t:Monotonicity}, then
	$p_{s^{-}}^{*}$ has at most a countable set of discontinuities, all of
	which are jump discontinuities. In particular, $c_{ES}$ and $c_{LS}$ are
	twice differentiable almost everywhere.
\end{corollary}

Another consequence of monotonicity is a converse of Theorem~\ref{t:Differentiability equivalence}:

\begin{corollary}%
	\label{c:Differentiability equivalence}
	Let $f$ satisfy the conditions of Theorem~\ref{t:Monotonicity}, then for
	every $a<b$ the following are equivalent:
	\begin{enumerate}
		\item For all $\ell \in (a,b)$
		\begin{displaymath}
		\mathbb{P} \left (\tilde{f}_{\ell
		}\text{\textup{ has an infinite four-arm saddle at the origin}}\right )=0;
		\end{displaymath}
		\item There exists a version of $p_{s^{-}}$ which is continuous on
		$(a,b)$;
		\item $c_{ES}(\cdot )$ is continuously differentiable on $(a,b)$;
		\item $c_{LS}(\cdot )$ is continuously differentiable on $(a,b)$.
	\end{enumerate}
\end{corollary}

\begin{remark}
	Clearly, if any of (1)--(4) hold in Corollary~\ref{c:Differentiability equivalence},
	then by Theorem~\ref{t:Differentiability equivalence}, the version of
	$p_{s^{-}}|_{(a,b)}$ which is continuous is equal to
	$p^{\ast }_{s^{-}}|_{(a,b)}$.
\end{remark}

Finally we use Theorem~\ref{t:Monotonicity} to deduce intervals on which
$c_{ES}$ and $c_{LS}$ are monotone. We shall state the strongest form of
our results only in the case of the RPW and BF field. Let $D_{+}$ and
$D^{+}$ respectively denote the lower and upper, right Dini derivatives,
that is, for $g:\mathbb{R}\to \mathbb{R}$,
\begin{displaymath}
D_{+}g(x)=\liminf _{\epsilon \to 0^{+}}
\frac{g(x+\epsilon )-g(x)}{\epsilon }\quad \text{and}\quad D^{+}g(x)=
\limsup _{\epsilon \to 0^{+}}\frac{g(x+\epsilon )-g(x)}{\epsilon }.
\end{displaymath}

\begin{proposition}%
	\label{p:RPW monotonicity}
	Let $f$ be the Random Plane Wave. Then
	\begin{align*}
	D_{+}c_{ES}(\ell )>0\quad &\text{for }\ell \in (-\infty ,0.87]
	\\
	D^{+}c_{ES}(\ell )<0\quad &\text{for }\ell \in [1,\infty )
	\end{align*}
	and
	\begin{align*}
	D^{+}c_{LS}(\ell )<0\quad \text{for }\ell \in [1,\infty ).
	\end{align*}
\end{proposition}

\begin{proposition}%
	\label{p:BF monotonicity}
	Let $f$ be the Bargmann-Fock field. Then there exists $\epsilon >0$ such
	that
	\begin{displaymath}
	c_{ES}^{\prime }(\ell )
	\begin{cases}
	>0 &\text{for }\ell \in (-\epsilon ,0.64]
	\\
	<0 &\text{for }\ell \in [1.03,\infty )
	\end{cases}
	\end{displaymath}
	and
	\begin{displaymath}
	c^{\prime }_{LS}(\ell )<0\quad \text{for }\ell \in [1.03,\infty ).
	\end{displaymath}
\end{proposition}

We also present weaker results for general isotropic fields. Recall that
the covariance function of an isotropic $f$ may be expressed as
$K(x)=k(\lvert x\rvert ^{2})$ for some
$k:[0,\infty )\to \mathbb{R}$. We also recall the parameter
$\chi =-k^{\prime }(0)/\sqrt{k^{\prime \prime }(0)}$ which takes values in
$(0,\sqrt{2}]$ (see \cite{cheng2015expected} for details on this parameter).

\begin{proposition}%
	\label{p:Isotropic monotonicity}
	Let $f$ be an isotropic field satisfying Assumptions~\ref{a:minimal},~\ref{a:regularity}--\ref{a:Arm decay}
	and~\ref{a:Monotonicity}. Then there exists $\epsilon >0$ and an explicit
	constant $C>0$ such that
	\begin{displaymath}
	c_{ES}^{\prime }(\ell )
	\begin{cases}
	>0 &\text{for }\ell \in (-\epsilon ,C)
	\\
	<0 &\text{for }\ell \in \left (\sqrt{2}/\chi ,\infty \right )
	\end{cases}
	\end{displaymath}
	and
	\begin{displaymath}
	c^{\prime }_{LS}(\ell )<0\;\text{ for }\ell \in \left (\sqrt{2}/\chi ,
	\infty \right ).
	\end{displaymath}
\end{proposition}

The explicit formula for the constant $C$ is quite complicated and is given
in the proof of this proposition. However it is straightforward to apply
this formula to any particular field (as we have done for the RPW and Bargmann-Fock
field in Propositions~\ref{p:RPW monotonicity} and~\ref{p:BF monotonicity}).

As an intermediate result to Proposition~\ref{p:RPW monotonicity} we require
that, for the RPW, $c_{ES}(\ell )>0$ for $\ell \le 0$. Since this result
is not stated elsewhere in the literature, we do so here. The proof uses
the `barrier method' and is near-identical to that in
\cite{SodinNazarov2015asymptotic} in the case $\ell =0$.

\begin{proposition}%
	\label{p:Positivity RPW}
	Let $f$ be the Random Plane Wave. Then $c_{ES}(\ell )>0$ for all
	$\ell \in \mathbb{R}$.
\end{proposition}

\begin{remark}%
	\label{r:Variance bound}
	Many of the results in this work are built upon by
	\cite{Beliaev2019fluctuations} in order to prove lower bounds on the variance
	of the number of level/excursion set components in $B(R)$ as
	$R\to \infty $. Specifically, it is shown that if $f$ has sufficiently
	nice correlation decay (such as the BF field), and if $c_{ES}$ has a non-zero
	derivative at $\ell $, then
	\begin{displaymath}
	\mathrm{Var}(N_{ES}(R,\ell ))\geq cR^{2}
	\end{displaymath}
	for some $c>0$ and all $R$ sufficiently large. Moreover, if $f$ is the
	RPW and one of the Dini derivatives of $c_{ES}$ is non-zero for
	$\ell \neq 0$, then
	\begin{displaymath}
	\mathrm{Var}(N_{ES}(R,\ell ))\geq cR^{3}
	\end{displaymath}
	for some $c>0$ and all $R$ sufficiently large. Analogous results hold in
	both cases for level sets and $c_{LS}$. A key step in proving these results
	is to estimate the order of
	\begin{displaymath}
	\mathbb{E}(N_{ES}(R,\ell )-N_{ES}(R,\ell +\epsilon _{R})),
	\end{displaymath}
	which is made possible by Theorem~\ref{t:differentiability of c_{LS}} and
	Corollary~\ref{c:Four arm}. Since the lower bounds also require that
	$c_{ES}$ has a non-zero derivative/Dini derivative at $\ell $, Propositions~\ref{p:RPW monotonicity}--\ref{p:Isotropic monotonicity}
	are crucial for ensuring that they are widely applicable.
\end{remark}

\subsection{Outline of the remainder of the paper}

In Section~\ref{s:palm} we give a formal definition of
$\tilde{f}_{\ell }$, the field $f$ conditioned to have a saddle point at
the origin at level $\ell $, and derive explicit representations for
$\tilde{f}_{\ell }$ in special cases. In Section~\ref{s:Continuity} we study
topological properties of $\tilde{f}_{\ell }$, and use this to deduce the
results outlined in Section~\ref{ss:Differentiability}. In Section~\ref{s:Monotonicity}
we consider stochastic monotonicity properties of $\tilde{f}_{\ell }$, and
complete the proofs of the results in Section~\ref{ss:Monotonicity}. Appendix~\ref{a:nondeg}
contains miscellaneous results on the non-degeneracy of Gaussian fields.

\section{The field conditioned to have a saddle at the origin}
\label{s:palm}

In this section we consider $\tilde{f}_{\ell }$, the field $f$ conditioned
to have a saddle point at the origin at level $\ell $. Using the theory
of Palm distributions we give an explicit representation for
$\tilde{f}_{\ell }$, and in the isotropic case we derive simple expressions
for its distribution.

We begin with a general statement expressing $\tilde{f}_{\ell }$ as (a limit
of) a Palm distribution relative to a point process defined by the saddle
points of $f$. Let us first recall the relevant theory of Palm distributions
(see \cite[Chapter 11]{kallenberg2006foundations} for background). We define
a point process $\zeta $ to be a random measure on $\mathbb{R}^{d}$ such
that $\zeta (B)$ is integer-valued for every bounded Borel set $B$. We
say that $\zeta $ is simple if, with probability one,
$\zeta (\{s\})\leq 1$ for every $s\in \mathbb{R}^{d}$. We say that it
is non-degenerate if $\mathbb{E}(\zeta (B))>0$ for every Borel set
$B$ with positive Lebesgue measure. Let
$g:\mathbb{R}^{2} \to \mathbb{R}$ be a planar random field and
$\mathcal{S}$ a non-degenerate, simple point process on
$\mathbb{R}^{2}$, and suppose that $(g, \mathcal{S})$ are jointly stationary
(i.e.\ this joint distribution is invariant under translations). Fix a
bounded Borel set $B \subset \mathbb{R}^{2}$ such that
$0<\mathbb{E}\left (\#\{s\in B:s \in \mathcal{S} \} \right ) <\infty $.
Then the \textit{Palm distribution of $g$ relative to $\mathcal{S}$} is
defined as the random field $\tilde{g}$ satisfying, for any Borel cylinder
set $A$,
\begin{equation}
\label{e:palm}
\mathbb{P}\left (\tilde{g}(x) \in A \right ) =
\frac{\mathbb{E}\left (\#\{s\in B:s \in \mathcal{S}, g(x - s) \in A \} \right )}{\mathbb{E}\left (\#\{s\in B:s \in \mathcal{S} \} \right ) }
.
\end{equation}
This definition is independent of the reference set $B$, and so we may write ${\tilde{g} = (g \,|\, \{0\} \in \mathcal{S})}$.

\begin{lemma}%
	\label{l:palm}
	Let $f$ be a Gaussian field satisfying Assumption~\ref{a:minimal}. For
	$\ell \in \mathbb{R}$ and $\epsilon >0$, let
	\begin{displaymath}
	\tilde{f}_{[\ell ,\ell +\epsilon ]}=(f\;|\{0\}\in S[\ell ,\ell +
	\epsilon ])
	\end{displaymath}
	be the Palm distribution of $f$ relative to
	$S[\ell ,\ell +\epsilon ]$, the point process of saddle points with level
	in $[\ell , \ell +\epsilon ]$ (this point process is non-degenerate by
	Lemma~\ref{a:nondegen3}). Then there exists a random field
	$\tilde{f}_{\ell }$ such that, as $\epsilon \to 0$,
	$\tilde{f}_{[\ell ,\ell +\epsilon ]}$ converges in distribution to
	$\tilde{f}_{\ell }$ in the topology of uniform $C^{2+\eta }$ convergence on
	compacts.
\end{lemma}

It is important to distinguish $\tilde{f}_{[\ell ,\ell +\epsilon ]}$ from
the conditioned field
\begin{align*}
& (f(t)|\nabla f(0)= {0},\det \nabla ^{2} f(0)<0,f(0)\in [\ell ,\ell +
\epsilon ]),
\end{align*}
which is defined via the distributional limit
\begin{equation*}
\lim _{\delta \to 0}(f(t)| f_{1}(0),f_{2}(0)\in [0,\delta ),\det
\nabla ^{2} f(0)<0,f(0)\in [\ell ,\ell +\epsilon ]) .
\end{equation*}
The latter is sometimes known as `vertical window conditioning', and is
the standard way of conditioning on part of a random vector (for a Gaussian
vector, this conditioning is given explicitly by Gaussian regression, see
\cite[Proposition~1.2]{azais2009level}). By constrast, the former can be
thought of as `horizontal window conditioning', and corresponds to sampling
a `typical' saddle point (i.e.\ via the counting measure). The difference
between these forms of conditioning is elegantly explained in
\cite{kac1959}.

Using basic properties of Gaussian fields, we can derive explicit representations
for $\tilde{f}_{[\ell ,\ell +\epsilon ]}$ and $\tilde{f}_{\ell }$:

\begin{lemma}%
	\label{l:conditional distribution}
	Let $f$ be a Gaussian field satisfying Assumption~\ref{a:minimal} such
	that $(f(0),\nabla ^{2} f(0))$ is a non-degenerate Gaussian vector. Define
	$\alpha :\mathbb{R}^{2}\to \mathbb{R}$ and
	$\beta =(\beta _{11},\beta _{22},\beta _{12}):\mathbb{R}^{2}\to
	\mathbb{R}^{3}$ to be the unique functions satisfying
	\begin{displaymath}
	\mathbb{E}\left (f(t)\middle |f(0)=u,\nabla ^{2} f(0)={U}\right )=
	\alpha (t)u+\beta (t)\cdot U
	\end{displaymath}
	for all $u\in \mathbb{R}$, $U\in \mathbb{R}^{3}$, and define
	\begin{displaymath}
	\gamma (s, t) := \mathbb{E}\left (f(s)f(t)\middle |f(0)=0,\nabla f(0)={0},
	\nabla ^{2} f(0)={0}\right ) .
	\end{displaymath}
	Then
	\begin{displaymath}
	\tilde{f}_{[\ell ,\ell +\epsilon ]}\overset{d}{=}g+ z_{[\ell ,\ell +
		\epsilon ]}\alpha +Z_{[\ell ,\ell +\epsilon ]}\cdot \beta ,
	\end{displaymath}
	where $g$ is a centred Gaussian field with covariance function
	$\gamma $, and
	$\left (z_{[\ell ,\ell +\epsilon ]},Z_{[\ell ,\ell +\epsilon ]}
	\right )$ is an independent random vector with density\footnote{Here and
		in the proof of this lemma we treat $Z_{[\ell ,\ell +\epsilon ]}$ interchangeably
		as the three-dimensional column vector
		$(Z_{[\ell ,\ell +\epsilon ],11},Z_{[\ell ,\ell +\epsilon ],22},Z_{[
			\ell ,\ell +\epsilon ],12})$ and the symmetric $2\times 2$ matrix
		$
		\Big(\begin{smallmatrix}
		Z_{[\ell ,\ell +\epsilon ],11} &Z_{[\ell ,\ell +\epsilon ],12}
		\\
		Z_{[\ell ,\ell +\epsilon ],12} &Z_{[\ell ,\ell +\epsilon ],22}
		\end{smallmatrix}\Big)
		$; which form we are using will always be clear from context. We also use
		this convention for $Z_{\ell }$ and $X$ (introduced below).}
	\begin{displaymath}
	p_{\left (z_{[\ell ,\ell +\epsilon ]},Z_{[\ell ,\ell +\epsilon ]}
		\right )}(x,X)\propto \lvert \det X\rvert \; p_{f(0),\nabla ^{2} f(0)}(x,X)
	\;\mathds{1}_{x\in [\ell ,\ell +\epsilon ]}\mathds{1}_{\det X<0}.
	\end{displaymath}
	Moreover,
	\begin{displaymath}
	\tilde{f}_{\ell }\overset{d}{=}g+ \ell \alpha +Z_{\ell }\cdot \beta
	\end{displaymath}
	where $Z_{\ell }$ is a random vector, independent of $g$, with density
	\begin{displaymath}
	p_{Z_{\ell }}(X)\propto \lvert \det X\rvert \; p_{f(0),\nabla ^{2}f(0)}(
	\ell ,X)\mathds{1}_{\det X<0}.
	\end{displaymath}
\end{lemma}

The functions $\alpha $, $\beta $ and $\gamma $ in Lemma~\ref{l:conditional distribution}
can be computed explicitly via Gaussian regression (see
\cite[Proposition 1.2]{azais2009level}). Specifically, define
${v}_{0}=(f(0),\partial _{xx}f(0),\partial _{yy}f(0),\partial _{xy}f(0))$
and
\begin{equation*}
{v}=(f(0),\nabla f(0),\partial _{xx}f(0),\partial _{yy}f(0),\partial _{xy}f(0))
,
\end{equation*}
and let $\Sigma _{0}$ and $\Sigma $ be the respective covariance matrices
of these vectors. Then
\begin{align*}
(\alpha (t),\beta _{11}(t),\beta _{22}(t),\beta _{12}(t))=
\mathrm{Cov}\left (f(t),{v}_{0}\right )\Sigma _{0}^{-1}
\end{align*}
and
\begin{align*}
\gamma (s,t)=\mathrm{Cov}\left (f(s),f(t)\right )-\mathrm{Cov}\left (f(s),{v}
\right )\Sigma ^{-1}\mathrm{Cov}\left (f(t),{v}\right )^{\prime }.
\end{align*}

In the case that $(f(0),\nabla ^{2} f(0))$ is degenerate (which includes
the RPW), the representations of
$\tilde{f}_{[\ell ,\ell +\epsilon ]}$ and $\tilde{f}_{\ell }$ in Lemma~\ref{l:conditional distribution}
must be modified to accommodate this degeneracy; in particular,
$\nabla ^{2} f(0)$ should be considered as a vector consisting of two of
its coordinates, chosen so that they are non-degenerate with $f(0)$, and
$\alpha , \beta $ and $\gamma $ defined accordingly. For simplicity we
will not state this representation formally; for the RPW we state a more
precise description below (in Proposition~\ref{p:ftildeRPW}).

Lemmas~\ref{l:palm} and~\ref{l:conditional distribution} are essentially
derived in \cite[Chapter 6]{Adler07applicationsof}; we repeat this here
for completeness, and so that we can extend the arguments slightly.

\begin{proof}[Proof of Lemmas~\ref{l:palm} and~\ref{l:conditional distribution}]
	We assume that $(f(0),\nabla ^{2} f(0))$ is non-degenerate, since the proof
	in the degenerate case is almost identical. Let
	$T=(t_{1},\dots ,t_{m})\in \mathbb{R}^{2m}$ and
	$y_{1},\dots ,y_{m}\in \mathbb{R}$. Then by the definition of
	$\tilde{f}_{[\ell ,\ell +\epsilon ]}$, and the Kac-Rice theorem (\cite[Corollary~11.2.2]{RFG}),
	\begin{align*}
	\mathbb{P}&\left (\tilde{f}_{[\ell ,\ell +\epsilon ]}(t_{1})\leq y_{1},
	\dots ,\tilde{f}_{[\ell ,\ell +\epsilon ]}(t_{m})\leq y_{m}\right )
	\\
	&\quad \quad \quad =
	\frac{\mathbb{E}\left (\#\{s\in B(1):\nabla f(s)={0},\det \nabla ^{2}f(s)<0,f(s)\in [\ell ,\ell +\epsilon ],f(s+t_{i})\leq y_{i}\;\forall i\right )}{\mathbb{E}\left (\#\{s\in B(1):\nabla f(s)={0},\det \nabla ^{2}f(s)<0,f(s)\in [\ell ,\ell +\epsilon ]\right )}
	\\
	&\quad \quad \quad =
	\frac{\mathbb{E}\left (\lvert \det \nabla ^{2}f(0)\rvert \mathds{1}_{f(0)\in [\ell ,\ell +\epsilon ]}\mathds{1}_{\det \nabla ^{2}f(0)<0}\prod _{i=1}^{m}\mathds{1}_{f(t_{i})\leq y_{i}}\middle |\nabla f(0)={0}\right )}{\mathbb{E}\left (\lvert \det \nabla ^{2} f(0)\rvert \mathds{1}_{f(0)\in [\ell ,\ell +\epsilon ]}\mathds{1}_{\det \nabla ^{2}f(0)<0}\middle |\nabla f(0)={0}\right )}
	\\
	&\quad \quad \quad =\int _{-\infty }^{y_{1}}\dots \int _{-\infty }^{y_{m}}
	\frac{\int _{\mathbb{R}^{3}}\int _{\ell }^{\ell +\epsilon }\lvert \det X\rvert p_{T}(x,{0},X,U)\mathds{1}_{\det X<0}\;dx\;dX}{\int _{\mathbb{R}^{3}}\int _{\ell }^{\ell +\epsilon }\lvert \det X\rvert p(x,{0},X)\mathds{1}_{\det X<0}\;dx\;dX}
	\;dU_{m}\;\dots \;dU_{1},
	\end{align*}
	where $p_{T}$ and $p$ denote respectively the densities of
	$(f(0),\nabla f(0),\nabla ^{2} f(0),f(t_{1}),\dots , f(t_{m}))$ and
	$(f(0),\nabla f(0),\nabla ^{2} f(0))$. We note that $p$ is non-degenerate
	since $\nabla f(0)$ is independent of $(f(0),\nabla ^{2} f(0))$ (this is
	a standard fact for Gaussian fields with constant variance, see
	\cite[Chapter 5]{RFG}) and these vectors are non-degenerate by assumption.
	The density $p_{T}$ may be degenerate, in which case we think of it as
	having atomic mass. Rearranging these terms slightly, we can express the
	joint density of
	$\left (\tilde{f}_{[\ell ,\ell +\epsilon ]}(t_{i})\;:i=1,\dots ,m
	\right )$ as
	\begin{displaymath}
	\varphi _{T}^{[\ell ,\ell +\epsilon ]}(U):=\int _{\mathbb{R}^{3}}
	\int _{\ell }^{\ell +\epsilon }\psi _{x}(X)p_{T}(x,{0},X,U)/p(x,{0},X)\;dx
	\;dX
	\end{displaymath}
	where
	\begin{displaymath}
	\psi _{x}(X)=
	\frac{\lvert \det X\rvert p(x,{0},X)\mathds{1}_{\det X<0}}{\int _{\mathbb{R}^{3}}\int _{\ell }^{\ell +\epsilon }\lvert \det X\rvert p(x,{0},X)\mathds{1}_{\det X<0}\;dx\;dX}.
	\end{displaymath}
	Then the characteristic function of
	$\left (\tilde{f}_{[\ell ,\ell +\epsilon ]}(t_{1}),\dots ,\tilde{f}_{[
		\ell ,\ell +\epsilon ]}(t_{m})\right )$ is given by
	\begin{equation}
	\label{e:Characteristic function}
	\hat{\varphi }_{T}^{[\ell ,\ell +\epsilon ]}(\theta )=\int _{
		\mathbb{R}^{3}}\int _{\ell }^{\ell +\epsilon }\psi _{x}(X)\int _{
		\mathbb{R}^{m}}e^{i\theta \cdot U}p_{T}(x,{0},X,U)/p(x,{0},X)\;dU\;dx
	\;dX.
	\end{equation}
	The inner integral of equation \eqref{e:Characteristic function} can be
	calculated by Gaussian regression (see
	\cite[Proposition~1.2]{azais2009level}). Specifically, let
	$A=(\alpha (t_{1}),\dots ,\alpha (t_{m}))^{\prime }$,
	$B=(\beta (t_{1}),\dots ,\beta (t_{m}))^{\prime }$ and
	$\Gamma =(\gamma (t_{i},t_{j}))_{i,j=1,\dots ,m}$, then
	\begin{align*}
	\left (f(t_{1}),\dots ,f(t_{m})\;\middle |f(0)=x,\nabla f(0)={0},
	\nabla ^{2} f(0)=X\right )\sim \mathcal{N}(Ax+BX,\Gamma ).
	\end{align*}
	Since $p_{T}(x,0,X,U)/p(x,0,X)$ is the probability density of this random
	variable, we can substitute the characteristic function of a Gaussian vector
	into \eqref{e:Characteristic function} to give
	\begin{align*}
	\hat{\varphi }_{T}^{[\ell ,\ell +\epsilon ]}(\theta )&=\int _{
		\mathbb{R}^{3}}\int _{\ell }^{\ell +\epsilon }\psi _{x}(X)e^{i\theta
		\cdot (Ax+BX)-\frac{1}{2}\theta ^{\prime }\Gamma \theta }\;dx\;dX
	\\
	&=e^{-\frac{1}{2}\theta ^{\prime }\Gamma \theta }
	\frac{\int _{\mathbb{R}^{4}}e^{i\theta \cdot (Ax+BX)}\lvert \det X\rvert p(x,{0},X)\mathds{1}_{x\in [\ell ,\ell +\epsilon ]}\mathds{1}_{\det X<0}\;dx\;dX}{\int _{\mathbb{R}^{4}}\lvert \det X\rvert p(x,{0},X)\mathds{1}_{x\in [\ell ,\ell +\epsilon ]}\mathds{1}_{\det X<0}\;dx\;dX}.
	\end{align*}
	Since the characteristic function of a random vector uniquely specifies
	its distribution, we identify the distribution of
	$\tilde{f}_{[\ell ,\ell +\epsilon ]}$ as that given in the statement of
	Lemma~\ref{l:conditional distribution} (using the fact that
	$p(x,{0},X)=p_{f(0),\nabla ^{2}f(0)}(x,X)$ since $\nabla f(0)$ is independent
	of $\left (f(0),\nabla ^{2} f(0)\right )$).
	
	By inspecting their joint distribution, it is clear that
	$z_{[\ell ,\ell +\epsilon ]}\xrightarrow{d}\ell $ and
	$Z_{[\ell ,\ell +\epsilon ]}\xrightarrow{d}Z_{\ell }$ as
	$\epsilon \rightarrow 0$. We now fix a sequence
	$\epsilon _{i}\downarrow 0$, and create a coupling of
	$\tilde{f}_{[\ell ,\ell +\epsilon _{i}]}$ for each $i$ such that each field
	consists of the same realisation of $g$ and the sequences
	$\{z_{[\ell ,\ell +\epsilon _{i}]}\}_{i\in \mathbb{N}}$ and
	$\{Z_{[\ell ,\ell +\epsilon _{i}]}\}_{i\in \mathbb{N}}$ converge almost
	surely. Since $K\in C^{4+\eta ^{\prime }}_{\text{loc}}(\mathbb{R}^{2})$, the
	same is true of $\alpha $, $\beta $ and $\gamma $ and hence
	$g\in C^{2+\eta }_{\text{loc}}(\mathbb{R}^{2})$ almost surely for the choice
	of $\eta \in (0,\eta ^{\prime }/2)$ made at the beginning of Section~\ref{s:Main results}
	(Kolmogorov's theorem \cite[Appendix~A]{SodinNazarov2015asymptotic}). It
	is therefore clear that the coupled fields
	$\tilde{f}_{[\ell ,\ell +\epsilon _{i}]}$ converge almost surely in the
	$C^{2+\eta }$ topology uniformly on compact sets as
	$i\rightarrow \infty $ to $\tilde{f}_{\ell }$. This completes the proof of
	the lemmas.
\end{proof}

We now present simpler descriptions for $\tilde{f}_{\ell }$ in the case of
isotropic fields. In this case it is quite natural to express the Hessian
component $Z_{\ell }$ in terms of its eigenvalues
$\lambda _{1}<\lambda _{2}$ and the argument $\theta $ of the first eigenvector.
Recall the parameter
$\chi =-k^{\prime }(0)/\sqrt{k^{\prime \prime }(0)}\in (0,\sqrt{2}]$, where
$K(t)=k(\lvert t\rvert ^{2})$. Again we must distinguish the case in which
$(f(0),\nabla ^{2} f(0))$ is degenerate, which corresponds to
$\chi = \sqrt{2}$ and implies that $f$ is (a rescaled version of) the RPW.

\begin{proposition}%
	\label{p:ftildeiso}
	Let $f$ be an isotropic field satisfying Assumption~\ref{a:minimal} such
	that $\chi < \sqrt{2}$. Then
	\begin{displaymath}
	\tilde{f}_{\ell }(\cdot )\overset{d}{=}g(\cdot )+\ell \alpha (\cdot )+
	\lambda _{1} b_{1}(\cdot ,\theta )+\lambda _{2} b_{2}(\cdot ,\theta ) ,
	\end{displaymath}
	where $g, \alpha $ and $\beta $ are as in Lemma~\ref{l:conditional distribution},
	\begin{align*}
	b_{1}(t,\theta )&=\cos ^{2}(\theta )\beta _{11}(t)+\sin ^{2}(\theta )
	\beta _{22}(t)+\sin (\theta )\cos (\theta )\beta _{12}(t),
	\\
	b_{2}(t,\theta )&=\sin ^{2}(\theta )\beta _{11}(t)+\cos ^{2}(\theta )
	\beta _{22}(t)+\sin (\theta )\cos (\theta )\beta _{12}(t),
	\end{align*}
	$\theta $ is an independent random variable uniform on $[0, 2\pi )$, and
	$(\lambda _{1},\lambda _{2})$ is an independent random vector with density
	proportional to
	\begin{displaymath}
	q_{\ell }(x,y):=\lvert x\rvert y(y-x)\mathds{1}_{y>0>x}\exp \left (-
	\frac{1}{2\sigma ^{2}}\left ((x-\mu \ell )^{2}+(y-\mu \ell )^{2}+2
	\tau (x-\mu \ell )(y-\mu \ell )\right ) \right ),
	\end{displaymath}
	where
	\begin{equation}
	\label{e:density parameters}
	\mu =2 k^{\prime }(0) \ , \quad \sigma ^{2}=
	\frac{16k^{\prime \prime }(0)(2-\chi ^{2})}{3-\chi ^{2}} \quad
	\text{and} \quad \tau =\frac{\chi ^{2}-1}{3-\chi ^{2}}.
	\end{equation}
\end{proposition}

\begin{proof}
	Recall the random vector $Z_{\ell }$ from Lemma~\ref{l:conditional distribution},
	which we view as a $2 \times 2$ symmetric matrix. Let
	$\lambda _{1}<\lambda _{2}$ be the eigenvalues of $Z_{\ell }$, and let
	$\theta $ be the argument of the eigenvector associated to
	$\lambda _{1}$. If $h$ denotes the bijection which maps $Z_{\ell }$ to
	$\Lambda :=(\lambda _{1},\lambda _{2},\theta )$, then for any Borel set
	$A$
	\begin{displaymath}
	\mathbb{P}\left (\Lambda \in h(A)\right )=
	\frac{\mathbb{E}\left (\left \lvert \det \nabla ^{2} f(0)\right \rvert \;\mathds{1}_{\det \nabla ^{2} f(0)<0,\nabla ^{2} f(0)\in A}\middle |f(0)=\ell \right )}{\mathbb{E}\left (\left \lvert \det \nabla ^{2} f(0)\right \rvert \;\mathds{1}_{\det \nabla ^{2} f(0)<0}\middle |f(0)=\ell \right )}.
	\end{displaymath}
	Since $f$ is isotropic and $(f(0),\nabla ^{2} f(0))$ is non-degenerate,
	\cite{cheng2015expected} derives the density of the ordered eigenvalues
	of $(\nabla ^{2} f(0)|f(0)=\ell )$ and the argument of the corresponding
	eigenvectors as that given above.
\end{proof}

\begin{proposition}%
	\label{p:ftildeRPW}
	Let $f$ be the Random Plane Wave. Then
	\begin{displaymath}
	\tilde{f}_{\ell }\overset{d}{=} g + \ell \alpha + Z^{\ell }\cdot \beta ,
	\end{displaymath}
	where $g$ is a centred Gaussian field with covariance function
	$\gamma $ (defined as in Lemma~\ref{l:conditional distribution}),
	$\alpha ,\beta $ are defined as
	\begin{align*}
	& \qquad \quad \alpha (t)=J_{0}(\lvert t\rvert )+2
	\frac{t_{1}^{2}-t_{2}^{2}}{\lvert t\rvert ^{2}}J_{2}\left (\lvert t
	\rvert \right ),
	\\
	& \beta _{11}(t)=4\frac{t_{1}^{2}-t_{2}^{2}}{\lvert t\rvert ^{2}}J_{2}(
	\lvert t\rvert ) , \qquad \beta _{12}(t)=8
	\frac{t_{1}t_{2}}{\lvert t\rvert ^{2}}J_{2}(\lvert t\rvert ),
	\end{align*}
	and $Z^{\ell }=(Z_{11}^{\ell },Z_{12}^{\ell })^{t}$ is an independent random
	vector with density
	\begin{displaymath}
	\psi _{\ell }(x,y)\propto \left (x(x+\ell )+y^{2}\right )\mathds{1}_{
		\left (x(x+\ell )+y^{2}\right )>0}\; p_{f(0),f_{11}(0),f_{12}(0)}(
	\ell ,x,y) .
	\end{displaymath}
	Alternatively, $\tilde{f}_{\ell }$ has the representation
	\begin{displaymath}
	\tilde{f}_{\ell }(t) = g(t) + \ell \cdot \left [J_{0}(\lvert t\rvert )+2
	\cos (2(\theta -\arg t))J_{2}(\lvert t\rvert )\right ]+\lambda \cdot 4
	\cos (2(\theta -\arg t))J_{2}(\lvert t\rvert ),
	\end{displaymath}
	where $\arg t$ denotes the argument of $t$ and
	$(\theta , \lambda )=(\theta ,\lambda _{\ell })$ is a random vector, independent
	of $g$, with density
	\begin{equation}
	\label{e:Lambda density}
	p_{\lambda _{\ell },\theta }(x,y) \propto x(x+\ell )(2x+\ell ) e^{-4x(x+
		\ell )}\mathds{1}_{x>\max \{0,-\ell \}}\mathds{1}_{y\in [0,2\pi )} .
	\end{equation}
\end{proposition}

\begin{proof}
	The representation
	\begin{displaymath}
	\tilde{f}_{\ell }=g+\ell \alpha +Z^{\ell }\cdot \beta
	\end{displaymath}
	follows from an argument similar to that used to prove Lemma~\ref{l:conditional distribution}.
	The functions $\alpha $ and $\beta $ can be explicitly calculated using
	Gaussian regression (see \cite[Proposition~1.2]{azais2009level} for example).
	Next we note that $(Z_{11}^{\ell },Z_{12}^{\ell })$ is supported on the region
	for which
	\begin{displaymath}
	\det
	\begin{pmatrix}
	Z_{11}^{\ell }&Z_{12}^{\ell
	}\\
	Z_{12}^{\ell }&-Z_{11}^{\ell }-\ell .
	\end{pmatrix}
	<0.
	\end{displaymath}
	Therefore this matrix almost surely has a unique, positive eigenvalue
	$\lambda $ and corresponding eigenvector with argument $\theta $. By explicitly
	diagonalising this matrix, we obtain a formula for $\lambda $ and
	$\theta $:
	\begin{align*}
	Z_{11}^{\ell }+\ell /2&=(\lambda +\ell /2)\cos (2\theta )
	\\
	Z_{12}^{\ell }&=(\lambda +\ell /2)\sin (2\theta ).
	\end{align*}
	By the standard change of variable formula (and explicitly evaluating
	$\psi _{\ell }$ in terms of the covariance of the RPW) we can calculate the
	joint density of $(\lambda ,\theta )$ to be equal to the expression in \eqref{e:Lambda density}.
\end{proof}

\section{Differentiability of excursion/level set functionals}
\label{s:Continuity}

In this section we prove the results stated in Section~\ref{ss:Differentiability}.
We begin by studying the space $C^{2+\eta }_{\text{Reg}}$ of functions
$h\in C^{2+\eta }_{\text{loc}}\left (\mathbb{R}^{2}\right )$ which have a
non-degenerate critical point at the origin and no other critical points
at level $h(0)$. We will also use the space
$C^{2+\eta }_{\text{Reg}}(R)$ which is the set of all
$h\in C^{2+\eta }_{\text{Reg}}$ such that $h(0)$ is not a critical level of
$h|_{\partial B(R)}$. We endow these spaces with the
$C^{2+\eta }_{\text{loc}}$ topology.

By showing that $\tilde{f}_{\ell }\in C^{2+\eta }_{\text{Reg}}(R)$ almost surely,
we prove that $\tilde{f}_{\ell }$ having an upper (or lower) connected saddle
point in a compact region is a continuity event, from this we deduce Theorem~\ref{t:Differentiability equivalence}
(with the other results following as consequences).

\begin{lemma}%
	\label{l:trichotomy conditional field}
	If $h \in C^{2+\eta }_{\text{Reg}}$ has a saddle point at the origin, then
	this saddle point is either upper connected, lower connected or an infinite
	four-arm saddle.
\end{lemma}
\begin{proof}
	For a small enough neighbourhood $B$ of the origin, the level set
	$\{h = h(0)\}$ in $B \setminus \{0\}$ consists of four curves that connect
	$0$ to $\partial B$ (the Morse lemma
	\cite[Lemma 2.2]{milnor1963morse}). If the connected components of these
	curves in $\mathbb{R}^{2} \setminus B$ are all unbounded, the saddle point
	must be infinite four-arm. If one of them is finite, then by the implicit
	function theorem it is a simple $C^{1}$ curve joining two points on
	$\partial B$. Hence the saddle point is either upper connected (if the
	field takes values larger than $h(0)$ on the outer boundary of the loop)
	or lower connected (if the field takes values smaller the $h(0)$ on the
	outer boundary of the loop).
\end{proof}

We now consider saddle points which are upper or lower connected in a compact
domain. Specifically, for a $C^{2}$ function $h$ with a saddle point
$x_{0}$ we say that $x_{0}$ is $R$\textit{-lower connected} if it is in
the closure of only one component of
$\{x\in B(x_{0},R):h(x)<h(x_{0})\}$. We make an analogous definition for
$R$-upper connected saddles.

\begin{lemma}%
	\label{l:continuity event}
	Let $s^{-}(R)$ be the subset of functions
	$h\in C^{2+\eta }_{\text{Reg}}(R)$ such that the origin is an $R$-lower connected
	saddle point of $h$, then $s^{-}(R)$ is open and closed in
	$C^{2+\eta }_{\text{Reg}}(R)$. The same is true for the set $s^{+}(R)$ of functions
	with $R$-upper connected saddle points.
\end{lemma}
\begin{proof}
	Let $h\in C^{2+\eta }_{\text{Reg}}(R)$ have a saddle point at the origin which
	is $R$-lower connected; we will find a neighbourhood around $h$ which contains
	only functions with such saddle points at the origin. First we choose
	$r\in (0,1)$ sufficiently small that $h$ has a four-arm saddle in
	$B(r)$. Since the origin is a non-degenerate saddle point for $h$,
	$\nabla ^{2} h(0)$ has eigenvalues $\lambda _{1}<0<\lambda _{2}$ and corresponding
	eigenvectors $v_{1},v_{2}$. We now choose a neighbourhood
	$N_{1}\subset C^{2+\eta }_{\text{Reg}}(R)$ of $h$ (in the topology of uniform
	$C^{2+\eta }$ convergence) such that for all $g\in N_{1}$,
	\begin{displaymath}
	\partial _{v_{1},v_{1}} g(0)<\lambda _{1}/2\quad \text{and}\quad
	\partial _{v_{2},v_{2}} g(0)>\lambda _{2}/2.
	\end{displaymath}
	This ensures that each function in $N_{1}$ also has a saddle point at the
	origin.
	
	Next we choose $N_{2}\subset C^{2+\eta }_{\text{Reg}}(R)$ such that for each
	$g\in N_{2}$,
	\begin{displaymath}
	\|g\|_{C^{2+\eta }(B(R))}\leq 2\|h\|_{C^{2+\eta }(B(R))}.
	\end{displaymath}
	We consider the four line segments joining $0$ to $\partial B(r)$ parallel
	to $v_{1}$ and $v_{2}$ and we reduce $r$ relative to
	$\|h\|_{C^{2+\eta }(B(R))}$ so that for each $g\in N_{1}\cap N_{2}$, the
	directional derivative of $g$ on this line segment (parallel to the line
	segment) has constant sign. This ensures that for each such $g$ the saddle
	point at the origin is four-arm in $B(r)$.
	
	There exist two connected subsets $A_{1},A_{2}$ of $\partial B(r)$ such
	that $h<h(0)-3\epsilon $ on $A_{1}\cup A_{2}$ for some $\epsilon >0$ and
	$A_{1}$ and $A_{2}$ are in different components of
	$\overline{B(r)}\cap \{h<h(0)\}$ (see Figure~\ref{Fig_4}). We next choose
	a neighbourhood $N_{3}\subset C^{2+\eta }_{\text{Reg}}(R)$ of $h$ such that
	$A_{1}$ and $A_{2}$ have the same properties for any function
	$g\in N_{3}$, with $3\epsilon $ replaced by $2\epsilon $, and
	$\lvert g(0)-h(0)\rvert <\epsilon $.
	
	\begin{figure}[h!]
		\centering
		\begin{tikzpicture}[scale=0.05]
		\draw[dashed] (0,0) node (v1) {} circle (20);
		\node[right] at (20,0) {$B(r)$};
		\draw[thick] plot[smooth, tension=.7] coordinates {(44,38) (41,34) (38,30) (33,28) (26,26) (16,20) (14,13) (14,6) (9,3) (3,1) (v1) (-4,2) (-9,3) (-12,7) (-16,8) (-22,9) (-26,13) (-32,15) (-37,19) (-44,20) (-50,14) (-57,10) (-52,0) (-52,-7) (-50,-18) (-45,-25) (-37,-24) (-31,-21) (-19,-17) (-13,-13) (-11,-10) (-6,-3) (v1) (7,-4) (17,-10) (23,-14) (27,-20) (32,-27) (41,-28)};
		\draw (0,0) circle (5pt);
		\draw[very thick] (70:20) arc (70:120:20);
		\draw[very thick] (-70:20) arc (-70:-120:20);
		\node at (0,7) {$-$};
		\node at (0,-8) {$-$};
		\node at (-12,-1) {$+$};
		\node at (13,-2) {$+$};
		\node[right] at (42,29) {$\{h=h(0)\}$};
		\node[above] at (95:20) {$A_1$};
		\node[below] at (-95:20) {$A_2$};
		\draw plot[smooth, tension=.7] coordinates {(-3,20) (-6,21) (-9,22) (-10,26) (-10,30) (-9,33) (-9,36)(-9,39) (-15,41) (-23,41) (-40,35) (-50,33) (-61,27) (-69,18) (-75,4) (-74,-8) (-63,-29) (-44,-41) (-29,-47) (-16,-47) (-12,-42) (-10,-39)(-9,-32) (-8,-25) (-7,-21) (-5,-20) (-4,-20)};
		\node[above] at (-40,37) {$\gamma$};
		\end{tikzpicture}
		\caption{Approximating an $R$-lower connected saddle point in the
			$C^{2+\eta }\left (\overline{B(R)}\right )$ topology.}%
		\label{Fig_4}
	\end{figure}
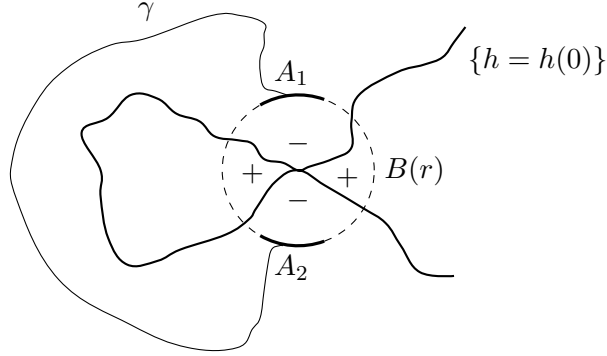
	
	By definition of a saddle being $R$-lower connected, there is a curve
	$\gamma $ in $B(R)$ joining $A_{1}$ to $A_{2}$ in $\{h<h(0)\}$ and
	$h$ is bounded above by $h(0)-3\delta $ on $\gamma $ for some
	$\delta >0$. Since $\gamma $ is compact we can find a neighbourhood
	$N_{4}$ such that $g<h(0)-2\delta $ on $\gamma $ for all
	$g\in N_{4}$ and $\lvert h(0)-g(0)\rvert <\delta $. Combining these observation,
	we see that $N:=N_{1}\cap N_{2}\cap N_{3}\cap N_{4}$ is a neighbourhood
	of $h$ (in $C^{2+\eta }_{\text{Reg}}(R)$) and any $g\in N$ has a saddle point
	at the origin which is lower connected in $B(R)$ and so the set of functions
	with such saddle points is open, as required.
	
	The set $C^{2+\eta }_{\text{Reg}}$ can be partitioned into sets of functions
	which have either a local maximum, a local minimum, a saddle point which
	is four-arm in $B(R)$ or a saddle point which is $R$-upper/lower connected
	at the origin. Arguments which are very similar to those above show that
	each of these subsets is open, hence proving the statement of the lemma.
	(For saddle points which are four-arm in $B(R)$, we use the fact that
	$h(0)$ is not a critical level of $h|_{\partial B(R)}$ which implies that
	the four level lines emanating from the origin intersect
	$\partial B(R)$ at different points.)
\end{proof}

We next confirm that $\tilde{f}_{\ell }\in C^{2+\eta }_{\text{Reg}}(R)$ almost
surely:

\begin{lemma}%
	\label{l:no four arm saddles conditional field}
	If $f$ is a Gaussian field satisfying Assumptions~\ref{a:minimal} and~\ref{a:non-degenerate gradient},
	then for any $\ell \in \mathbb{R}$ and $R>0$,
	$\tilde{f}_{\ell }\in C^{2+\eta }_{\text{Reg}}(R)$ almost surely.
\end{lemma}

\begin{proof}
	To simplify the presentation we assume that
	$(f(0),\nabla ^{2} f(0))$ is non-degenerate; the proof in the degenerate
	case is similar. Recall the representation of $\tilde{f}_{\ell }$ in Lemma~\ref{l:conditional distribution}.
	By the definitions of $\alpha ,\beta $ and $\gamma $,
	$\tilde{f}_{\ell }$ is almost surely in
	$C^{2+\eta }_{\text{loc}}(\mathbb{R}^{2})$, and has a critical point at the
	origin at level $\ell $. By evaluating the second order derivatives of
	$\alpha ,\beta $ and $\gamma $, it follows that
	$\nabla ^{2}\tilde{f}_{\ell }(0)=Z_{\ell }$. Since the density of
	$Z_{\ell }$ is identically zero on the region where its determinant is zero,
	$\det \nabla ^{2}\tilde{f}_{\ell }(0)\neq 0$ almost surely, and so the critical
	point at the origin is non-degenerate.
	
	Next we show that $\tilde{f}_{\ell }$ almost surely has no other critical
	points at level $\ell $. Let
	$T_{n}=\overline{B(n)}\backslash B(\frac{1}{n})$ and consider
	$(\nabla \tilde{f}_{\ell },\tilde{f}_{\ell }-\ell ):T_{n}\rightarrow
	\mathbb{R}^{3}$. Bulinskaya's lemma (\cite[Lemma~11.2.10]{RFG}) states
	that this function almost surely has no zeroes in $T_{n}$ provided the
	univariate densities of
	$(\nabla \tilde{f}_{\ell }(t),\tilde{f}_{\ell }(t))$ are bounded in a neighbourhood
	of $(0,\ell )$ uniformly over $t\in T_{n}$. Since $g$ and $Z_{\ell }$ are
	independent, the density of
	$(\nabla \tilde{f}_{\ell }(t),\tilde{f}_{\ell }(t))$ is given by
	\begin{align*}
	p_{\nabla \tilde{f}_{\ell }(t),\tilde{f}_{\ell }(t)}({x})&=\int _{
		\mathbb{R}^{3}} p_{\nabla g(t),g(t)}({x}-{u})p_{\nabla (Z_{\ell
		}\cdot \beta (t)+\ell \alpha (t)),Z_{\ell }\cdot \beta (t)+\ell \alpha (t)}({u})
	\;d{u}
	\\
	&\leq \sup _{{x}\in \mathbb{R}^{3}}p_{\nabla g(t),g(t)}({x})\int _{
		\mathbb{R}^{3}} p_{\nabla (Z_{\ell }\cdot \beta (t)+\ell \alpha (t)),Z_{\ell }\cdot \beta (t)+\ell \alpha (t)}({u})\;d{u}
	\\
	&=\sup _{{x}\in \mathbb{R}^{3}}p_{\nabla g(t),g(t)}({x}).
	\end{align*}
	Therefore, to show that
	$p_{\nabla \tilde{f}_{\ell }(t),\tilde{f}_{\ell }(t)}$ is bounded, it is sufficient
	to show that the density of $(\nabla g(t),g(t))$ is bounded uniformly in
	$t$. Since these densities are Gaussian, this is equivalent to showing
	that the determinant of the covariance matrix of
	$(\nabla g(t),g(t))$ is bounded away from $0$ on $T_{n}$. However this
	is the determinant of
	\begin{displaymath}
	\mathrm{Cov}\left ( \nabla f(t), f(t) \;\middle |f(0),\nabla f(0),
	\nabla ^{2} f(0)\right )
	\end{displaymath}
	which is non-degenerate for each $t\in T_{n}$ by Assumption~\ref{a:non-degenerate gradient}.
	Since this determinant is continuous in $t$, it is bounded away from
	$0$ on the compact set $T_{n}$. Taking the countable union of
	$T_{n}$ for $n\in \mathbb{N}$ then shows that $\tilde{f}_{\ell }$ almost
	surely has no critical points at level $\ell $ in
	$\mathbb{R}^{2}\backslash \{0\}$.
	
	To verify that $\tilde{f}_{\ell }|_{\partial B(R)}$ almost surely has no
	critical points at level $\ell $, we apply an identical argument to
	\begin{displaymath}
	\left (
	\begin{pmatrix}
	-\sin (\theta )
	\\
	\cos (\theta )
	\end{pmatrix}
	\cdot \nabla \tilde{f}_{\ell }(y), \tilde{f}_{\ell }(y)\right )
	\end{displaymath}
	where $y=(R\cos (\theta ),R\sin (\theta ))$. This completes the proof that
	$\tilde{f}_{\ell }\in C^{2+\eta }_{\text{Reg}}(R)$ almost surely.
\end{proof}

We are now ready to prove Theorem~\ref{t:Differentiability equivalence}.
Let $N_{s^{-}}^{(R)}[\ell _{1},\ell _{2}]$ denote the number of $R$-lower
connected saddle points of $f$ in $B(1)$ with height in
$[\ell _{1},\ell _{2}]$. If $N_{s^{-}}^{(R)}$ is replaced with
$N_{s^{-}}$ or $N_{s}$, we make a corresponding definition for lower connected
saddle points or saddle points respectively. Recall that $s^{-}(R)$ is
the subset of functions in $C^{2+\eta }_{\text{Reg}}(R)$ with an $R$-lower
connected saddle point at the origin. We also define $s^{-}$ and
$s^{+}$ to be the subsets of $C^{2+\eta }_{\text{Reg}}$ with lower and upper
connected saddle points at the origin respectively.

\begin{proof}[Proof of Theorem~\ref{t:Differentiability equivalence}]
	Let $f$ be a field satisfying Assumptions~\ref{a:minimal} and~\ref{a:non-degenerate gradient}.
	The first step is to show that $p_{s^{-}}^{*}$ is lower semi-continuous
	by expressing it as the pointwise supremum of a sequence of continuous
	functions. Let $\ell \in \mathbb{R}$ and $\epsilon >0$ and we fix
	$R>0$. We now claim that
	\begin{displaymath}
	\frac{\mathbb{E}\left (N_{s^{-}}^{(R)}[\ell ,\ell +\epsilon ]\right )}{\mathbb{E}(N_{s}[\ell ,\ell +\epsilon ])}=
	\mathbb{P}\left (\tilde{f}_{[\ell ,\ell +\epsilon ]}\in s^{-}(R)
	\right ).
	\end{displaymath}
	We first note that, by Lemma~\ref{l:continuity event}, the event
	$s^{-}(R)$ is contained in the Borel $\sigma $-algebra generated by the
	$C^{2+\eta }_{\text{Reg}}(R)$ topology and that
	$\tilde{f}_{[\ell ,\ell +\epsilon ]}$ is measurable with respect to this
	$\sigma $-algebra. Furthermore, by an elementary argument (see, for example,
	\cite[Lemma~A.1]{SodinNazarov2015asymptotic}) this $\sigma $-algebra is
	generated by cylinder sets; those which depend on the value of the function
	at only finitely many points. Since the distribution of
	$\tilde{f}_{[\ell ,\ell +\epsilon ]}$ on cylinder sets is defined in \eqref{e:palm} as an empirical measure, these two measures must coincide
	on the $\sigma$-algebra generated by this $\pi $-system. This verifies
	the claim.
	
	By Lemma~\ref{l:palm}, $\tilde{f}_{[\ell ,\ell +\epsilon ]}$ converges
	in distribution to $\tilde{f}_{\ell }$ (in the $C^{2+\eta }_\text{loc}$ topology)
	as $\epsilon \to 0$, and since having a saddle point at the origin which
	is $R$-lower connected is a continuity event for $\tilde{f}_{\ell }$ (Lemmas~\ref{l:continuity event}
	and~\ref{l:no four arm saddles conditional field}), the portmanteau lemma
	implies that
	\begin{displaymath}
	\mathbb{P}\left (\tilde{f}_{[\ell ,\ell +\epsilon ]}\in s^{-}(R)
	\right )\to \mathbb{P}\left (\tilde{f}_{\ell }\in s^{-}(R)\right )
	\end{displaymath}
	as $\epsilon \to 0$. By inspecting the form of $p_{Z_{\ell }}$ it is clear
	that $\tilde{f}_{\ell }\xrightarrow{d}\tilde{f}_{\ell _{0}}$ as
	$\ell \rightarrow \ell _{0}$ in the $C^{2+\eta }_\text{loc}$ topology. So by
	applying the portmanteau lemma again, we see that
	$\mathbb{P}(\tilde{f}_{\ell }\in s^{-}(R))$ is continuous in $\ell $. Hence
	the function
	\begin{equation}
	\label{e:Differentiability equivalent}
	p_{s^{-}}^{(R)}(\ell ):=p_{s}(\ell )\mathbb{P}\left (\tilde{f}_{\ell
	}\in s^{-}(R)\right ),
	\end{equation}
	is continuous in $\ell $. Now note that
	\begin{displaymath}
	\frac{1}{\epsilon }\mathbb{E}\left (N_{s^{-}}^{(R)}[\ell ,\ell +
	\epsilon ]\right )=
	\frac{\mathbb{E}\left (N_{s}[\ell ,\ell +\epsilon ]\right )}{\epsilon }
	\frac{\mathbb{E}\left (N_{s^{-}}^{(R)}[\ell ,\ell +\epsilon ]\right )}{\mathbb{E}(N_{s}[\ell ,\ell +\epsilon ])}
	\to p_{s^{-}}^{(R)}(\ell )
	\end{displaymath}
	as $\epsilon \to 0$ (by Proposition~\ref{p:density existence}). Hence
	$\mathbb{E}(N_{s^{-}}^{(R)}[-\infty ,\ell ])$ is differentiable in
	$\ell $ with derivative $p_{s^{-}}^{(R)}(\ell )$. We now allow $R$ to vary;
	since $s^{-}(R)$ is non-decreasing in $R$ and
	$\cup _{R>0}s^{-}(R)=s^{-}$, taking the limit of \eqref{e:Differentiability equivalent} shows that
	\begin{displaymath}
	p_{s^{-}}^{*}(\ell )=\lim _{R\to \infty }p_{s^{-}}^{(R)}(\ell )
	\end{displaymath}
	for each $\ell \in \mathbb{R}$. Hence $p_{s^{-}}^{*}$ is indeed a pointwise
	supremum of continuous functions, and so is lower semi-continuous.
	
	We next prove that $p_{s^{-}}^{*}=p_{s^{-}}$ almost everywhere. Let
	$a<b$, then since $\lvert p_{s^{-}}^{*}-p_{s^{-}}\rvert $ is bounded, by
	dominated convergence
	\begin{align*}
	\int _{a}^{b}p_{s^{-}}(x)-p_{s^{-}}^{*}(x)\;dx&=\lim _{R\to \infty }
	\int _{a}^{b}p_{s^{-}}(x)-p_{s^{-}}^{(R)}(x)\;dx
	\\
	&=\lim _{R\to \infty }\mathbb{E}\left (N_{s^{-}}[a,b]-N_{s^{-}}^{(R)}[a,b]
	\right )=0,
	\end{align*}
	where in the last line we have used the definition of $p_{s^{-}}$, the
	fundamental theorem of calculus applied to
	$\mathbb{E}(N_{s^{-}}^{(R)}[-\infty ,\ell ])$ (along with the differentiability
	proven above), and then dominated convergence once again. Since $a$ and
	$b$ are arbitrary, we conclude that $p_{s^{-}}^{*}=p_{s^{-}}$ almost everywhere.
	
	To finish the proof we show that \eqref{e:fourarm}, the condition that
	$\tilde{f}_{\ell }$ does not have an infinite four-arm saddle, implies the
	continuity of $p_{s^{-}}^{*}$. Observe that, by repeating the arguments
	above, we may define the lower semi-continuous function
	\begin{displaymath}
	p_{s^{+}}^{*}(\ell ):=p_{s}(\ell )\mathbb{P} \left (\tilde{f}_{\ell
	}\in s^{+} \right )
	\end{displaymath}
	which is a version of $p_{s^{+}}$. By Lemmas~\ref{l:trichotomy conditional field}
	and~\ref{l:no four arm saddles conditional field}, the saddle point of
	$\tilde{f}_{\ell }$ at the origin must be either upper connected, lower connected
	or an infinite four-arm saddle. Therefore
	\begin{equation}
	\label{e:Differentiability equivalent 1}
	1-\mathbb{P} \left (\tilde{f}_{\ell
	}\text{ has an infinite four-arm saddle} \right )=\mathbb{P} \left (
	\tilde{f}_{\ell }\in s^{+}\cup s^{-}\right )=
	\frac{p_{s^{+}}^{*}(\ell )}{p_{s}(\ell )}+
	\frac{p_{s^{-}}^{*}(\ell )}{p_{s}(\ell )}.
	\end{equation}
	(Note that $p_{s}(\ell )>0$ by Lemma~\ref{a:nondegen3}.) Now suppose that \eqref{e:fourarm} holds, that is, for all $\ell \in (a,b)$,
	$\tilde{f}_{\ell }$ almost surely does not have an infinite four-arm saddle
	point at the origin. By \eqref{e:Differentiability equivalent 1} we see
	that $p_{s^{+}}^{*}(\ell )=p_{s}(\ell )-p_{s^{-}}^{*}(\ell )$ for all
	$\ell \in (a,b)$. Since $p_{s^{+}}^{*}$ is lower semi-continuous (and
	$p_{s}$ is continuous), we deduce that $p_{s^{-}}^{*}$ is upper semi-continuous
	on $(a,b)$. Hence we have shown that $p_{s^{-}}^{*}$ is both upper and
	lower semi-continuous on $(a,b)$, which completes the result.
\end{proof}

As mentioned previously, Theorem~\ref{t:differentiability of c_{LS}} follows
from Theorem~\ref{t:Differentiability equivalence} once we verify condition~\eqref{e:fourarm}.
This is done in the next lemma:

\begin{lemma}%
	\label{l:conditional one arm decay}
	Let $f$ be a Gaussian field satisfying Assumptions~\ref{a:minimal} and~\ref{a:regularity}--\ref{a:Arm decay}.
	Then for every $\ell \geq 0$ and $r>0$,
	\begin{equation}
	\label{e:arm lemma}
	\mathbb{P} \left (\tilde{f}_{\ell }\in \mathrm{Arm}_{\ell }(r,R) \right )
	\to 0
	\end{equation}
	as $R\to \infty $. In particular, for all $\ell \in \mathbb{R}$,
	$\tilde{f}_{\ell }$ almost surely does not have an infinite four-arm saddle
	point at the origin.
\end{lemma}

\begin{proof}
	We first note that if $\tilde{f}_{\ell }$ has an infinite four-arm saddle
	at the origin, then both $\{\tilde{f}_{\ell }\geq \ell \}$ and
	$\{\tilde{f}_{\ell }\leq \ell \}=\{-\tilde{f}_{\ell }\geq -\ell \}$ have unbounded
	components containing the origin. Then, since $\tilde{f}_{\ell }$ and
	$-\tilde{f}_{-\ell }$ have the same distribution by Lemma~\ref{l:conditional distribution},
	the second claim of this lemma follows from the first.
	
	Since the event $\mathrm{Arm}_{\ell }(r,R)$ is weakly increasing in $r$, it
	is sufficient to prove \eqref{e:arm lemma} for a sequence
	$r_{R}\to \infty $ as $R\to \infty $. This allows us to make use of the
	fact that, far from the origin, the distribution of
	$\tilde{f}_{\ell }$ is close to that of $f$.
	
	By Assumption~\ref{a:regularity} and Lemma~\ref{a:nondegen2} we know that
	$(f(0),\nabla ^{2} f(0))$ is non-degenerate. Recall the representation
	for $\tilde{f}_{\ell }$ in Lemma~\ref{l:conditional distribution}
	\begin{align*}
	\tilde{f}_{\ell }&=g+\ell \alpha +Z_{\ell }\cdot \beta ,
	\end{align*}
	and recall also the explicit expressions for $\alpha , \beta $ and the
	covariance of $g$ derived after this lemma. Since this covariance is expressed
	as the difference of two positive definite functions, if we let
	$f_{1}$ be a centred Gaussian field with covariance
	\begin{displaymath}
	K_{1}(s,t)=\text{Cov}\left (f(s),{v}\right )\Sigma ^{-1}\text{Cov}
	\left (f(t),{v}\right )^{\prime },
	\end{displaymath}
	then we can decompose $f = g+ f_{1}$, where $f_{1}$ and $g$ are independent.
	Since $K_{1}$ can be expressed as a linear combination of
	$\partial ^{k_{1}}K(s) \partial ^{k_{2}}K(t)$, for
	$|k_{1}|,|k_{2}| \le 2$, by Assumption~\ref{a:regularity} there exists
	$c_{1},\nu >0$ such that, for all $r > 1$,
	\begin{equation}
	\label{e:armdecay1}
	\sup _{s,t \notin B(r)} \sup _{|k| \le 2} \left \lvert \partial ^{k} K_{1}(s,
	t) \right \rvert \leq c_{1} r^{-2(1+\nu )}.
	\end{equation}
	Moreover, since $\alpha , \beta $ can be expressed as a linear combination
	of $\partial ^{k}K(t)$, for $|k| \le 2$, by Assumption~\ref{a:regularity}
	there exists $c_{2},\nu >0$ such that
	\begin{equation}
	\label{e:armdecay2}
	\sup _{\lvert t\rvert >r}\lvert \alpha (t)\rvert \leq c_{2} r^{-(1+
		\nu )}\quad \text{and}\quad \sup _{\lvert t\rvert >r}\|\beta (t)\|_{\infty }\leq c_{2}r^{-(1+\nu )}.
	\end{equation}
	
	Next, we fix $\ell \geq 0$ and apply a Cameron-Martin argument to the unconditional
	field $f$. Specifically, by \cite[Corollary~3.7]{Muirhead2018sharp} (valid
	by the condition on the spectral density in Assumption~\ref{a:regularity},
	and since $\text{Arm}_{\ell }(r,R)$ is an increasing event with respect to
	the field) there exists $c_{3},r_{0}>0$ such that for all $r>r_{0}$ the
	following holds: if $F:\mathbb{R}^{2}\to \mathbb{R}$ is a continuous
	random field coupled with $f$ such that
	\begin{displaymath}
	\mathbb{P}\left (\|f-F\|_{\infty ,A(r,R)}\geq \epsilon \right )\leq
	\delta
	\end{displaymath}
	where $\|\cdot \|_{\infty ,A(r,R)}$ denotes the supremum norm on
	$A(r,R)$ the centred annulus of inner radius $r$ and outer radius
	$R$, then
	\begin{equation}
	\label{e:Conditional one arm}
	\mathbb{P}\left (F\in \text{Arm}_{\ell }(r,R)\right )\leq \mathbb{P}
	\left (f\in \text{Arm}_{\ell }(r,R)\right )+\delta +c_{3}R\epsilon .
	\end{equation}
	We will apply this bound to $F = \tilde{f}_{\ell }$. Note that, by the union
	bound,
	\begin{equation}
	\label{e:conditional one arm 1}
	\begin{aligned}
	\mathbb{P}&\left (\left \|  f-\tilde{f}_{\ell }\right \|  _{\infty ,A(r,R)}
	\geq \epsilon \right )
	\\
	&\quad \quad \quad \quad \quad \quad \leq \mathbb{P}\left (\|\ell
	\alpha +Z_{\ell }\cdot \beta \|_{\infty ,A(r,R)}\geq \epsilon /2\right )+
	\mathbb{P}\left (\|f_{1}\|_{\infty ,A(r,R)}\geq \epsilon /2\right )
	\\
	&\quad \quad \quad \quad \quad \quad \leq
	\mathds{1}\left \{  \ell \|\alpha \|_{\infty ,A(r,R)}\ge \epsilon /8
	\right \}  +\sum _{i\in \{11,12,22\}}\mathbb{P}\left (\lvert Z_{\ell ,i}
	\rvert \|\beta _{i}\|_{\infty ,A(r,R)}\geq \epsilon /8\right )
	\\
	&\quad \quad \quad \quad \quad \quad \quad +\mathbb{P}\left (\|f_{1}
	\|_{\infty ,A(r,R)}\geq \epsilon /2\right )
	\end{aligned}
	\end{equation}
	where $Z_{\ell ,i}$ denotes the elements of the random vector
	$Z_{\ell }$. We now show that, for a suitable choice of
	$r = r_{R} \to \infty $ and $\epsilon = \epsilon _{R} \to 0$, the three
	terms in \eqref{e:conditional one arm 1} all decay to zero as
	$R \to \infty $.
	
	By \eqref{e:armdecay2}, and since $Z_{\ell ,i}$ is almost surely finite,
	the first two terms in~\eqref{e:conditional one arm 1} converge to zero
	as long as $\epsilon r^{1+\nu }\to \infty $. If we assume this convergence
	is sufficiently fast (to be specified below) then it is a standard estimate
	for the norm of a Gaussian field that the third term of \eqref{e:conditional one arm 1} also converges to zero. This argument is
	essentially the same as \cite[Lemma 3.12]{Muirhead2018sharp}, but our setting
	is slightly different so we give a complete proof.
	
	Let $B_{x}(1)$ denote the ball of radius $1$ centred at $x$. Covering
	$A(r, R)$ with $O(R^{2})$ unit balls, and by the union bound,
	\begin{equation*}
	\mathbb{P}\left (\|f_{1}\|_{\infty ,A(r,R)}\geq \epsilon /2\right )
	\le c_{3} R^{2} \sup _{x \in A(r, R)} \mathbb{P}\left (\|f_{1}\|_{
		\infty ,B_{x}(1))}\geq \epsilon /2\right ) .
	\end{equation*}
	By the Borell--TIS inequality (\cite[Theorem 2.1.1]{RFG}), for all
	$u > 0$,
	\begin{equation*}
	\mathbb{P}\left (\|f_{1}\|_{\infty ,B_{x}(1))} \ge m_{x} + u \right )
	\le 2e^{-u^{2} / (2 \sigma _{x}^{2}) },
	\end{equation*}
	where
	\begin{equation*}
	m_{x} = \mathbb{E}[ \| f_{1} \|_{\infty ,B_{x}(1)} ] \quad \text{and}
	\quad \sigma _{x}^{2} = \sup _{y \in B_{x}(1)} K_{1}(y,y) .
	\end{equation*}
	By Kolmogorov's theorem
	\cite[Appendix~A.9]{SodinNazarov2015asymptotic}, there is a
	$c_{4} > 0$ such that
	\begin{equation*}
	m_{x} < c_{4} \sup _{s,t \in B_{x}(1)} \sup _{|\alpha _{1}|, |\alpha _{2}|
		\le 1} \left ( \partial ^{\alpha _{1}, \alpha _{2}} K_{1}(s, t)
	\right )^{1/2} .
	\end{equation*}
	Therefore, by \eqref{e:armdecay1},
	\begin{equation*}
	\sup _{x \in A(r, R)} m_{x} < c_{5} r^{-1-\nu } \quad \text{and} \sup _{x
		\in A(r, R)} \sigma _{x}^{2} < c_{5} r^{-2-2\nu }.
	\end{equation*}
	Taking $u=\epsilon /4$ and assuming that
	$\epsilon /4>c_{5}r^{-1-\nu }$ we have
	\begin{equation*}
	\mathbb{P}\left (\|f_{1}\|_{\infty ,A(r,R)}\geq \epsilon /2\right )
	\le 2 c_{6} R^{2} \exp ({-c_{7} \epsilon ^{2}r^{2+2\nu }}).
	\end{equation*}
	To finish, we take
	\begin{displaymath}
	r=\frac{R}{\log (R)}\quad \text{and} \quad \epsilon =
	\frac{1}{R\log (R)}
	\end{displaymath}
	and observe that for this choice the right hand side of the estimate above
	converges to $0$ as $R\to \infty $. Combining all of these estimates together
	we have that the right hand side of \eqref{e:conditional one arm 1} tends
	to zero as $R\to \infty $.
	
	Substituting into \eqref{e:Conditional one arm}, and noting that
	$r/R\to 0$ and $R\epsilon \to 0$ as $R\to \infty $, proves that
	$\mathbb{P}(\tilde{f}_{\ell }\in \text{Arm}_{\ell }(r,R))$ can be made arbitrarily
	small, which completes the proof of the lemma.
\end{proof}

\begin{proof}[Proof of Theorem~\ref{t:differentiability of c_{LS}}]
	This is immediate from Theorem~\ref{t:Differentiability equivalence} and
	Lemma~\ref{l:conditional one arm decay}.
\end{proof}

To end the section we prove the remaining results stated in Section~\ref{ss:Differentiability},
namely Corollary~\ref{c:Four arm} and Proposition~\ref{p:positivity of c_{LS}}.

\begin{proof}[Proof of Corollary~\ref{c:Four arm}]
	By \cite[Lemmas~2.4 and~4.5]{Beliaev2018Number},
	$\mathbb{E}(N_{\mathrm{4\mhyphen arm}}(R))=O(R)$ as $R\to \infty $, so it suffices
	to prove the other bound here. Recall that $A(R-r,R)$ denotes the annulus
	of inner radius $R-r$ and outer radius $R$. We first note that for any
	$1<r<R$
	\begin{align*}
	N_{\mathrm{4\mhyphen arm}}(R,[a_{R},b_{R}])\leq & N_{\mathrm{c}}\left (A\left (R-r,R
	\right ),[a_{R},b_{R}]\right )+N_{\mathrm{4\mhyphen arm},r}\left (B\left (R-r
	\right ),[a_{R},b_{R}]\right )
	\end{align*}
	where, by a slight abuse of notation,
	$N_{\mathrm{c}}\left (A\left (R-r,R\right ),[a_{R},b_{R}]\right )$ denotes
	the number of critical points in $A\left (R-r,R\right )$ which have level
	in $[a_{R}, b_{R}]$, and
	$N_{\mathrm{4\mhyphen arm},r}(B(R-r),[a_{R},b_{R}])$ denotes the number of saddle
	points $t\in B(R-r)$ which are four-arm in $B(t,r)$ and have level in
	$[a_{R},b_{R}]$. Using the Kac-Rice theorem (\cite[Corollary 11.2.2]{RFG})
	and the independence of $(f(0),\nabla ^{2}f(0))$ and $\nabla f(0)$
	\begin{equation}
	\label{e:Four arm 1}
	\begin{aligned}
	&\mathbb{E}\left (N_{c}(A(R-r,R) ,[a_{R},b_{R}])\right )
	\\
	& \qquad \qquad =\int _{A(R-r,R)}\mathbb{E}\left (\left \lvert \det
	\left (\nabla ^{2} f(0)\right )\right \rvert \mathds{1}_{f(0)\in [a_{R},b_{R}]}
	\middle |\nabla f(0)={0}\right )p_{\nabla f(0)}(0)\;dt
	\\
	&\qquad \qquad =c_{1}\left (R^{2}-(R-r)^{2}\right )\int _{a_{R}}^{b_{R}}
	\mathbb{E}\left (\left \lvert \det \left (\nabla ^{2} f(0)\right )
	\right \rvert \middle |f(0)=x\right )p_{f(0)}(x)\;dx
	\\
	&\qquad \qquad \leq c_{2} Rr\cdot (b_{R}-a_{R})
	\end{aligned}
	\end{equation}
	for some $c_{1},c_{2}>0$ independent of $R$. By stationarity of $f$
	\begin{equation}
	\label{e:Four arm 2}
	\begin{aligned}
	\mathbb{E}(N_{\mathrm{4\mhyphen arm},r}(B(R-r),[a_{R},b_{R}]))&\leq R^{2}\;
	\mathbb{E}\left (N_{\mathrm{4\mhyphen arm},r}(B(1),[a_{R},b_{R}])\right )
	\\
	&=\pi R^{2}\;\int _{a_{R}}^{b_{R}}p_{s}(x)-p_{s^{-}}^{(r)}(x)-p_{s^{+}}^{(r)}(x)
	\;dx
	\end{aligned}
	\end{equation}
	where $p_{s^{-}}^{(r)}$ and $p_{s^{+}}^{(r)}$ are the continuous functions
	defined as in the proof of Theorem~\ref{t:Differentiability equivalence}.
	In this proof it is shown that as $r\to \infty $,
	$p_{s^{-}}^{(r)}+p_{s^{+}}^{(r)}$ converges pointwise monotonically to
	$p_{s^{-}}^{*}+p_{s^{+}}^{*}=p_{s}$ which is continuous. Therefore by Dini's
	theorem this convergence is uniform on $[a,b]$ and so for any
	$\epsilon >0$ taking $r$ sufficiently large relative to $\epsilon $ ensures
	that the right hand side of \eqref{e:Four arm 2} is bounded above by
	$\epsilon R^{2}(b_{R}-a_{R})$. If we choose $r$ depending on $R$ such that
	$r\to \infty $ but $r/R\to 0$ as $R\to \infty $, then combining \eqref{e:Four arm 1} and \eqref{e:Four arm 2} proves the corollary.
\end{proof}

\begin{proof}[Proof of Proposition~\ref{p:positivity of c_{LS}}]
	By Theorem~\ref{t:integral equality} and the identities in Proposition~\ref{p:density existence},
	\begin{equation}
	\label{e:pos2}
	c_{LS}(\ell ) = c_{ES}(\ell )+c_{ES}(-\ell ) .
	\end{equation}
	Let us also consider the function
	\begin{equation}
	\label{e:pos}
	h(\ell ) =c_{ES}(\ell )-c_{ES}(-\ell ).
	\end{equation}
	In \cite[Corollary~1.12]{Beliaev2018Number}, this is interpreted as the
	asymptotic mean Euler characteristic of the excursion set at level
	$\ell $, and hence shown via explicit calculation to be equal to the
	$C^{1}$ function
	\begin{equation*}
	h(\ell ) =\sqrt{\det \nabla ^{2}K(0)}\frac{\ell }{(2\pi )^{3/2}}e^{-
		\ell ^{2}/2} .
	\end{equation*}
	If $c_{LS}(0) = 0$, then it follows from \eqref{e:pos2} that
	$c_{ES}(0)= 0$. Similarly, if $c_{LS}$ is differentiable at $0$, then by \eqref{e:pos2} and the differentiability of $h$, $c_{ES}$ is also differentiable
	at $0$.
	
	It remains to show that if $c_{ES}$ is differentiable at $0$ then
	$c_{ES}(0)\neq 0$. Suppose for the sake of contradiction that
	$c_{ES}(0) = 0$. Then by the non-negativity of $c_{ES}$, we have
	$c_{ES}'(0) = 0$. Hence, by \eqref{e:pos}, $h'(0)=0$. Since $h$ has critical
	points only at $\ell = \pm 1$, we have derived the necessary contradiction.
\end{proof}

\begin{remark}
	Assuming differentiability of $c_{ES}$ or $c_{LS}$ at $\ell $, the above
	argument actually shows that $c_{LS}(\ell )>0$ for all
	$\ell \neq \pm 1$ (although it apparently says nothing about the positivity
	of $c_{ES}(\ell )$ for $\ell \neq 0$).
\end{remark}

\section{Monotonicity results}
\label{s:Monotonicity}

In this section we prove the monotonicity results stated in Section~\ref{ss:Monotonicity}.
The main intermediate step is to show that the finite-dimensional projections
of $\tilde{f}_{\ell }-\ell $ are stochastically decreasing in $\ell $, which
we do in the next subsection.

\subsection{Stochastic monotonicity}

Our analysis differs depending on whether we deal with the RPW or a general
isotropic field satisfying Assumption~\ref{a:Monotonicity}, the RPW case
being somewhat simpler.

\subsubsection{Stochastic monotonicity for the RPW}

Let $f$ be the RPW. The first step is to show, via explicit calculation,
that $\tilde{f}_{\ell }-\ell $ is stochastically decreasing in $\ell $ at
every point.

By Proposition~\ref{p:ftildeRPW}, $\tilde{f}_{\ell }$ has the distribution
\begin{equation}
\label{e:rpwrep}
\tilde{f}_{\ell }(t) = g(t) + \ell \cdot \left [J_{0}(\lvert t\rvert )+2
\cos (2(\theta -\arg t))J_{2}(\lvert t\rvert )\right ]+\lambda \cdot 4
\cos (2(\theta -\arg t))J_{2}(\lvert t\rvert )
\end{equation}
for the random vector $(\theta , \lambda )$ defined in that proposition.
To simplify notation, we define
\begin{align*}
a&:=a(t,\theta )=1-J_{0}(\lvert t\rvert )-2\cos (2(\theta -\arg t))J_{2}(
\lvert t\rvert )
\\
b&:=b(t,\theta )=4\cos (2(\theta -\arg t))J_{2}(\lvert t\rvert ).
\end{align*}
The key fact leading to stochastic monotonicity is that, by Lemma~\ref{l:bessel}
below, $a(t,\theta ) \ge 0$ for all $t$ and $\theta $. This is equivalent
to the statement that for all $t\in \mathbb{R}^{2}$
\begin{displaymath}
\alpha (t)=\mathbb{E}\left (f(t) \, \middle | \, f(0)=1,f_{11}(0)=f_{12}(0)=0
\right )\leq 1.
\end{displaymath}
For general isotropic fields, we show in Lemma~\ref{l:genmonassump} that
Assumption~\ref{a:Monotonicity} implies $\alpha (t)\leq 1$ (recall that
$\alpha $ has a slightly different definition in the general case, see
Lemma~\ref{l:conditional distribution}).

\begin{lemma}%
	\label{l:bessel}
	For all $t \in \mathbb{R}^{2}$ and $\theta \in \mathbb{R}$,
	\begin{equation*}
	a = a(t, \theta ) = 1-J_{0}(\lvert t\rvert )-2 \cos (2 (\theta -\arg t))
	J_{2}(\lvert t\rvert ) \ge 0 .
	\end{equation*}
\end{lemma}
\begin{proof}
	It is sufficient to prove that, for $s \ge 0$,
	\begin{equation*}
	1-J_{0}(s)-2J_{2}(s) \ge 0 \quad \text{and} \quad 1-J_{0}(s)+2J_{2}(s)
	\ge 0 .
	\end{equation*}
	By the identity $2J_{1}(s)/s=J_{0}(s)+J_{2}(s)$ and an explicit uniform
	bound on $\sqrt{s}\lvert J_{n}(s)\rvert $ given in
	\cite[Theorem 2.1]{olenko2006upper}, the first inequality holds for all
	$s>4$. Hence, since $1-J_{0}(0)-2J_{2}(0)=0$ and
	$\frac{d}{ds}(1-J_{0}(s)-2J_{2}(s))=J_{3}(s)$ (which is non-negative for
	$s\in [0,4]$), the first inequality holds for all $s\geq 0$.
	
	The same bound from \cite{olenko2006upper} shows that the second inequality
	holds for $s\geq 11$. Since $\lvert J_{0}\rvert \leq 1$ everywhere and
	$J_{2}(s)\geq 0$ for $s\in [0,5]\cup [9,11]$, the inequality also holds
	on these intervals. We verify the second inequality on the remaining compact
	set $[5,9]$ by inspection. More precisely, since
	\begin{equation*}
	\left \lvert \frac{d}{ds}(1-J_{0}(s)+2J_{2}(s)) \right \rvert = | 2J_{1}(s)
	- J_{3}(s) | \le 2|J_{1}(s)| + |J_{3}(s)| < 2 ,
	\end{equation*}
	it suffices to check that $1-J_{0}(s)+2J_{2}(s) > 0.08$ for all
	$s \in \{5 + 4i/100 : i = 0, 1, \ldots , 100\}$.
\end{proof}

\begin{remark}
	We prove the above lemma by somewhat explicit computations. We believe
	that there might be a more conceptual proof of this statement.
\end{remark}

We shall actually show the slightly stronger statement that
$\tilde{f}_{\ell }-\ell $ is pointwise stochastically decreasing conditional
on all values of $(g, \theta )$:

\begin{lemma}%
	\label{l:RPW pointwise stochastic decreasing}
	Let $f$ be the RPW. For $t\in \mathbb{R}^{2}$ and
	$c\in \mathbb{R}$
	\begin{displaymath}
	\mathbb{P}\left (\tilde{f}_{\ell }(t)-\ell \leq c \middle | g, \theta
	\right )
	\end{displaymath}
	is non-decreasing in $\ell \in \mathbb{R}$.
\end{lemma}
\begin{proof}
	Given the representation in \eqref{e:rpwrep}, we have
	\begin{equation}
	\label{e:RPW pointwise}
	\mathbb{P}\left (\tilde{f}_{\ell }(t)-\ell \leq c\middle |g,\theta
	\right )=
	\begin{cases}
	\mathbb{P}\left (\lambda \leq (c-g(t)+a\ell )/b\right ) &\text{if }b(t,
	\theta )>0,
	\\
	\mathbb{P}\left (\lambda \geq (c-g(t)+a\ell )/b\right )&\text{if }b(t,
	\theta )<0,
	\\
	\mathds{1}_{a\ell +c-g(t)\geq 0} &\text{if }b(t,\theta )=0.
	\end{cases}
	\end{equation}
	It remains to show that each of the expressions on the right-hand side
	of \eqref{e:RPW pointwise} are non-decreasing in $\ell $ for all values
	of $g$, $\theta $ and $c$. Recall that $a = a(t, \theta ) \ge 0$ by Lemma~\ref{l:bessel}.
	Hence $\mathds{1}_{a\ell +c-g(t)\geq 0}$ is clearly non-decreasing in
	$\ell $. Moreover, after integrating \eqref{e:Lambda density}, we see that
	for a differentiable function $h:\mathbb{R}\to \mathbb{R}$
	\begin{equation}
	\label{e:RPW pointwise 2}
	\frac{d}{d\ell }\mathbb{P}\left (\lambda \leq h(\ell )\right )=p_{\lambda }(h)\left (h^{\prime }(\ell )+\frac{h(\ell )}{2h(\ell )+\ell }
	\right )\leq p_{\lambda }(h)\left (h^{\prime }(\ell )+1\right )
	\end{equation}
	where the last inequality follows from the fact that
	$p_{\lambda }(h)$ is zero unless $h>0\vee (-\ell )$. Now let
	$h(\ell )=(c-g(t)+a\ell )/b$, and first suppose $b>0$. Then
	$h^{\prime }(\ell )=a/b>0$, $h/(2h+\ell )\geq 0$ on the region
	$h>0\vee (-\ell )$ and $p_{\lambda }(h)\geq 0$ (as a probability density)
	so \eqref{e:RPW pointwise 2} shows that the left hand side of \eqref{e:RPW pointwise} is non-decreasing whenever $b>0$. Finally we suppose
	$b<0$ and note that
	\begin{displaymath}
	h^{\prime }(\ell )+1=\frac{a(t,\theta )+b(t,\theta )}{b(t,\theta )}=
	\frac{1-J_{0}(\lvert t\rvert )+2\cos (2(\theta -\arg t))J_{2}(\lvert t\rvert )}{b(t,\theta )}=
	\frac{a(t,\theta +\pi /2)}{b(t,\theta )}\leq 0.
	\end{displaymath}
	So once again, \eqref{e:RPW pointwise 2} shows the left hand side of \eqref{e:RPW pointwise} is non-decreasing whenever $b<0$, completing the
	proof of the lemma.
\end{proof}

We now extend this result to finite-dimensional projections of
$\tilde{f}_{\ell }-\ell $. Recall that a random vector
$X=(X_{1},\dots ,X_{n})$ is said to \textit{stochastically dominate} a random
vector $Y=(Y_{1},\dots ,Y_{n})$, written $X \succ Y$, if
$\mathbb{E}(g(X))\geq \mathbb{E}(g(Y))$ for any coordinate-wise increasing
$g:\mathbb{R}^{n}\to \mathbb{R}$. Clearly, if $X \succ Y$ then
$X_{i} \succ Y_{i}$ for each $i=1,\dots ,n$. The converse is not true in
general, but a useful sufficient condition can be formulated using the
notion of copulas.

Let $X=(X_{1},\dots ,X_{n})$, where $X_{i}$ has cumulative density function
$F_{i}$ and induced probability measure $\mathbb{P}_{i}$. Then Sklar's
theorem states that there exists a (unique on
$\Pi _{i=1}^{n}\text{Range}(\mathbb{P}_{i})$) function
$\text{Cop}_{X}:[0,1]^{n}\to [0,1]$, known as the \textit{copula} of
$X$, such that
\begin{displaymath}
\mathbb{P}\left (X\in A_{1}\times \dots \times A_{n}\right )=
\text{Cop}_{X}(\mathbb{P}_{1}(A_{1}),\dots ,\mathbb{P}_{n}(A_{n}))
\end{displaymath}
for all $A_{1},\dots ,A_{n}\in \mathcal{B}(\mathbb{R})$. The copula is
equivalently specified by
\begin{equation}
\label{e:copula}
\text{Cop}_{X}(u_{1},\dots ,u_{n})=\mathbb{P}\left (F_{1}(X_{1})\leq u_{1},
\dots ,F_{n}(X_{n})\leq u_{n}\right ),
\end{equation}
i.e.\ $\text{Cop}_{X}$ is the joint cumulative density function of the collection
of uniform-[0,1] random variables $F_{1}(X_{1}),\dots ,F_{n}(X_{n})$.

\begin{theorem}[{\cite[Theorem~2]{scarsini1988multivariate}}]%
	\label{t:Stochastic domination marginal}
	Let $X=(X_{1},\dots ,X_{n})$ and $Y=(Y_{1},\dots ,Y_{n})$ be random vectors
	with induced marginal probability measures
	$\mathbb{P}_{1},\dots ,\mathbb{P}_{n}$ and
	$\mathbb{Q}_{1},\dots ,\mathbb{Q}_{n}$ respectively. If
	$\mathrm{Cop}_{X} = \mathrm{Cop}_{Y}$,
	$\Pi _{i=1}^{n} \mathrm{Range}(\mathbb{P}_{i})=\Pi _{i=1}^{n}
	\mathrm{Range}(\mathbb{Q}_{i})$, and $X_{i} \succ Y_{i}$ for each $i$, then
	$X \succ Y$.
\end{theorem}

Using this theorem, we extend Lemma~\ref{l:RPW pointwise stochastic decreasing}
to show the stochastic monotonicity of the finite-dimensional projections
$\tilde{f}_{\ell }-\ell $, conditional on any $g,\theta $.

\begin{lemma}%
	\label{l:RPW finite stochastic decreasing}
	Let $f$ be the RPW. For $\ell _{1}<\ell _{2}$ and
	$t_{1},\dots ,t_{n}\in \mathbb{R}^{2}$,
	\begin{equation}
	\label{e:Stochastic domination vector}
	\left (\tilde{f}_{\ell _{1}}(t_{1})-\ell _{1},\dots ,\tilde{f}_{\ell _{1}}(t_{n})-
	\ell _{1}\middle |g,\theta \right ) \succ \left (\tilde{f}_{\ell _{2}}(t_{1})-
	\ell _{2},\dots ,\tilde{f}_{\ell _{2}}(t_{n})-\ell _{2}\middle |g,
	\theta \right ).
	\end{equation}
\end{lemma}

\begin{proof}
	By Theorem~\ref{t:Stochastic domination marginal} it is sufficient to show
	that the random vectors in \eqref{e:Stochastic domination vector} have
	the same copula (these copulas are uniquely defined on the same domain,
	and the stochastic domination of the marginals follows from Lemma~\ref{l:RPW pointwise stochastic decreasing}).
	
	Fix $\ell \in \mathbb{R}$ and
	$t_{1},\dots ,t_{n}\in \mathbb{R}^{2}$, and consider the copula
	\begin{align*}
	\text{Cop}_{Z} := \text{Cop}_{\tilde{f}_{\ell }(t_{1})-\ell ,\dots ,
		\tilde{f}_{\ell }(t_{n})-\ell \big |g,\theta } .
	\end{align*}
	By the definition of $a$ and $b$, we can express
	\begin{displaymath}
	\tilde{f}_{\ell }(t)-\ell =g(t)-\ell a(t,\theta )+\lambda _{\ell } b(t,
	\theta )
	\end{displaymath}
	for deterministic functions $a,b$. Since $g(t_{i})$,
	$\ell a(t_{i},\theta )$ and $b(t_{i},\theta )$ are constants under the
	conditioning on $(g, \theta )$, and since copulas are invariant under strictly
	increasing transformations,
	\begin{displaymath}
	\text{Cop}_{Z} = \text{Cop}_{\lambda _{\ell }\cdot \text{sign}(b(t_{1},
		\theta )),\dots ,\lambda _{\ell }\cdot \text{sign}(b(t_{n},\theta ))
		\big |g,\theta } =\text{Cop}_{\lambda _{\ell }\cdot \text{sign}(b(t_{1},
		\theta )),\dots ,\lambda _{\ell }\cdot \text{sign}(b(t_{n},\theta ))
		\big |\theta },
	\end{displaymath}
	where the last equality holds since $g$ is independent of
	$\lambda =\lambda _{\ell }$ and $\theta $. Notice that the random vector
	$(\lambda _{\ell }\cdot \text{sign}(b(t_{1},\theta )),\dots ,\lambda _{
		\ell }\cdot \text{sign}(b(t_{n},\theta ))\big |\theta )$ consists of
	$\lambda _{\ell }$ multiplied by a constant vector (with elements taking
	values $1$, $-1$ or $0$). Hence by considering the alternative characterisation
	of a copula in \eqref{e:copula}, it is clear that $\text{Cop}_{Z}$ does
	not depend on the distribution of $\lambda _{\ell }$, and so is independent
	of~$\ell $.
\end{proof}

\subsubsection{Stochastic monotonicity in the general isotropic case}
\label{ss:mongen}

The copula argument in the RPW case relies crucially on the degeneracy
of RPW, which implies that after conditioning on $\theta $ and $g$, the
field depends only on the single random variable $\lambda $. For general
isotropic fields, $\tilde{f}_{\ell }$ is defined in terms of two eigenvalues,
so this argument fails. Instead we use a different method that works with
the finite-dimensional projections directly.

Let $f$ satisfy Assumptions~\ref{a:minimal},~\ref{a:non-degenerate gradient}
and~\ref{a:Monotonicity}. Recall from Proposition~\ref{p:ftildeiso} that
\begin{displaymath}
\tilde{f}_{\ell }(t)=g(t)+\ell \alpha (t)+\lambda _{1} b_{1}(t,\theta )+
\lambda _{2} b_{2}(t,\theta )
\end{displaymath}
for $\alpha , b_{1}, b_{2}$ as stated in the proposition (recall that
$b_{1}$ and $b_{2}$ are defined in terms of $\beta $). The role of Assumption~\ref{a:Monotonicity}
is to ensure the following inequalities hold for
$\alpha , b_{1}, b_{2}$:

\begin{lemma}
	\label{l:genmonassump}
	Let $f$ satisfy Assumptions~\ref{a:minimal},~\ref{a:non-degenerate gradient}
	and~\ref{a:Monotonicity}. For all $t \in \mathbb{R}^{2}$ and
	$\theta \in [0,2\pi )$,
	\begin{equation*}
	b_{1}(t,\theta ) \ge 0 \ , \quad b_{2}(t, \theta ) \ge 0 \quad
	\text{and} \quad \alpha (t) \le 1 .
	\end{equation*}
\end{lemma}
\begin{proof}
	From the definition of $b_{1}(t,\theta )$ and $\beta $, it is immediate
	that $b_{1}(t,0)$ is the quantity given in Assumption~\ref{a:Monotonicity}
	to be non-negative for all values of $t$. Since $f$ is isotropic,
	$b_{1}$ is non-negative for all values of $\theta $. (By the identity
	$\cos (\theta )=\sin (\theta +\pi /2)$, this also means that $b_{2}$ is
	non-negative.) Similarly, $\alpha $ is the function given in Assumption~\ref{a:Monotonicity}
	to be bounded above by $1$.
\end{proof}

\begin{lemma}%
	\label{l:General field finite stochastic decreasing}
	Let $f$ satisfy Assumptions~\ref{a:minimal},~\ref{a:non-degenerate gradient}
	and~\ref{a:Monotonicity}. For any
	$t_{1},\dots ,t_{n}\in \mathbb{R}^{2}$ and
	$u_{1},\dots ,u_{n}\in \mathbb{R}$,
	\begin{displaymath}
	\mathbb{P} \left ( \tilde{f}_{\ell }(t_{i})-\ell \leq u_{i}\;\forall i=1,
	\dots ,n\;\middle |g,\theta \right )
	\end{displaymath}
	is non-decreasing in $\ell \in \mathbb{R}$.
\end{lemma}
\begin{proof}
	Since the $u_{i}$ are arbitrary, we may assume $g(t_{i})=0$ for all
	$i$. We define the region
	\begin{displaymath}
	A_{\ell }=\left \{  (x,y)\in \mathbb{R}^{2}:xb_{1}(t_{i},\theta )+yb_{2}(t_{i},
	\theta )\leq u_{i}+(1-\alpha (t_{i}))\ell ,\;\;\forall i=1,\dots ,n
	\right \}
	\end{displaymath}
	so that $\tilde{f}_{\ell }(t_{i})-\ell \leq u_{i}$ for all $i$ if and only
	if $(\lambda _{1},\lambda _{2})\in A_{\ell }$. It is enough to prove that
	the probability of the latter event is non-decreasing in $\ell $ because
	$(\lambda _{1},\lambda _{2})$ is independent of $(g,\theta )$. Given the
	density of $(\lambda _{1},\lambda _{2})$ in Proposition~\ref{p:ftildeiso},
	\begin{align*}
	\frac{d}{d\ell }\mathbb{P}\left ((\lambda _{1},\lambda _{2})\in A_{\ell }\right )&=\frac{d}{d\ell }
	\frac{\int _{A_{\ell }}q_{\ell }(x,y)\;dxdy}{\int _{\mathbb{R}^{2}}q_{\ell }(x,y)\;dxdy}
	\\
	&=
	\frac{\int _{\mathbb{R}^{2}}q_{\ell }\;dxdy\cdot \frac{d}{d\ell }\int _{A_{\ell }}q_{\ell }\;dxdy-\int _{A_{\ell }}q_{\ell }\;dxdy\cdot \frac{d}{d\ell }\int _{\mathbb{R}^{2}}q_{\ell }\;dxdy}{\left (\int _{\mathbb{R}^{2}}q_{\ell }\;dxdy\right )^{2}}
	.
	\end{align*}
	Since $b_{1},b_{2}\geq 0$ and $\alpha \leq 1$ by Lemma~\ref{l:genmonassump},
	$A_{\ell }$ is non-decreasing in $\ell $, and so for this derivative to be
	non-negative it is sufficient that
	\begin{equation}
	\label{e:Eigenvalue derivation 1}
	\frac{\int _{A_{\ell }}\frac{d}{d\ell }q_{\ell }(x,y)\;dxdy}{\int _{A_{\ell }}q_{\ell }(x,y)\;dxdy}
	\geq
	\frac{\int _{\mathbb{R}^{2}}\frac{d}{d\ell }q_{\ell }(x,y)\;dxdy}{\int _{\mathbb{R}^{2}}q_{\ell }(x,y)\;dxdy}.
	\end{equation}
	By direct evaluation
	\begin{displaymath}
	\frac{d}{d\ell }q_{\ell }(x,y)=\frac{\mu (1+\tau )}{\sigma ^{2}}(x+y-2
	\mu \ell )q_{\ell }(x,y),
	\end{displaymath}
	where $\mu $, $\tau $ and $\sigma ^{2}$ are defined in \eqref{e:density parameters}. Since $\mu <0$, \eqref{e:Eigenvalue derivation 1} is equivalent to
	\begin{equation}
	\label{e:Eigenvalue derivation 2}
	\mathbb{E}\left (\lambda _{1}+\lambda _{2}|(\lambda _{1},\lambda _{2})
	\in A_{\ell }\right )\leq \mathbb{E}\left (\lambda _{1}+\lambda _{2}
	\right ).
	\end{equation}
	To complete the proof of the lemma, we show that this inequality holds
	for any possible region~$A_{\ell }$. Since the shape of $A_{\ell }$ might be
	quite complicated (see Figure~\ref{Fig_5} for a typical example), we divide
	the analysis into three cases and in each case show that conditioning on
	$(\lambda _{1},\lambda _{2})$ being contained in some simple region can
	only increase the expectation of $\lambda _{1}+\lambda _{2}$ relative to
	conditioning on $(\lambda _{1},\lambda _{2})\in A_{\ell }$.
	
	\begin{figure}[h!]
		\centering
		\begin{tikzpicture}[scale=0.15]
		\draw[->] (-20,0) -- (20,0);
		\draw[->](0,-20)--(0,20);
		\node[below] at (20,0) {$x$};
		\node[right] at (0,20) {$y$};
		\draw[thick](-25,15)--(-15,13)--(-9,9)--(-3,2)--(-3,-22);
		\path[fill=gray,opacity=0.5](-25,15)--(-15,13)--(-9,9)--(-3,2)--(-3,-22)--plot[smooth, tension=.7] coordinates {(-3,-22)(-12,-22)(-23,-5)(-28,8)(-25,15)};
		\draw (-4,9)--(-7,3);
		\node[above] at (-4,9) {$A_\ell$};
		\end{tikzpicture}
		\caption{A typical example of the region $A_{\ell }$.}%
		\label{Fig_5}
	\end{figure}
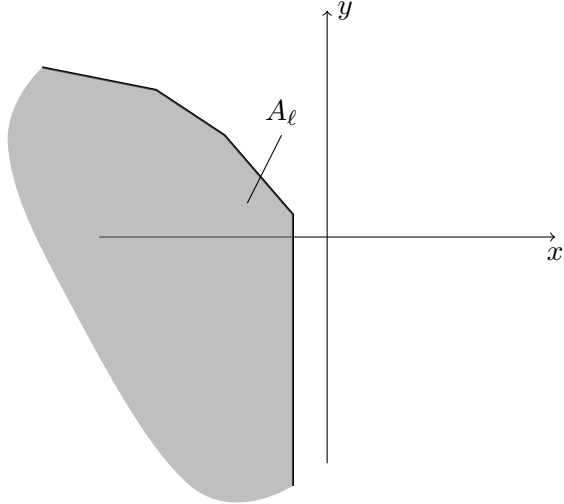
	
	(Case 1). Suppose $b_{1}(t_{i},\theta )/b_{2}(t_{i},\theta )= 1$ for all
	$i$, so that the boundary of $A_{\ell }$, denoted $\partial A_{\ell }$, is
	a line of the form $x+y=c$ for some $c\in \mathbb{R}$. (Note that
	$A_{\ell }$ is a subset of the entire plane not just the upper-left quadrant
	which is the support of $(\lambda _{1},\lambda _{2})$.) Then trivially
	\begin{displaymath}
	\mathbb{E}\left (\lambda _{1}+\lambda _{2}|(\lambda _{1},\lambda _{2})
	\in A_{\ell }\right )=\mathbb{E}\left (\lambda _{1}+\lambda _{2}|
	\lambda _{1}+\lambda _{2}\leq c\right )\leq \mathbb{E}\left (\lambda _{1}+
	\lambda _{2}\right )
	\end{displaymath}
	and so \eqref{e:Eigenvalue derivation 2} is verified in this case.
	
	(Case 2). Now suppose that
	$b_{1}(t_{i},\theta )/b_{2}(t_{i},\theta )\geq 1$ for all $i$ and that
	this inequality is strict for some $i$ (allowing for the degenerate case
	that $b_{2}(t_{i},\theta )=0$). Let
	$c^{*}=\mathbb{E}\left (\lambda _{1}+\lambda _{2}|(\lambda _{1},
	\lambda _{2})\in A_{\ell }\right )$ and then we note that the line
	$x+y=c^{*}$ must intersect $\partial A_{\ell }$ at precisely one point (since
	the distribution of $(\lambda _{1},\lambda _{2})$ is continuous) which
	we denote by $(d_{1},d_{2})$. We now consider the region
	$\{x\leq d_{1}\}$, conditioning on $(\lambda _{1},\lambda _{2})$ lying
	in this region weakly increases the probability that
	$\lambda _{1}+\lambda _{2}=c$ for $c\geq c^{*}$ and weakly decreases this
	probability for $c<c^{*}$ (see Figure~\ref{Fig_6}). Therefore
	\begin{equation}
	\label{e:Eigenvalue derivation 3}
	\mathbb{E}\left (\lambda _{1}+\lambda _{2}|(\lambda _{1},\lambda _{2})
	\in A_{\ell }\right )\leq \mathbb{E}(\lambda _{1}+\lambda _{2}|\lambda _{1}
	\leq d_{1}).
	\end{equation}
	
	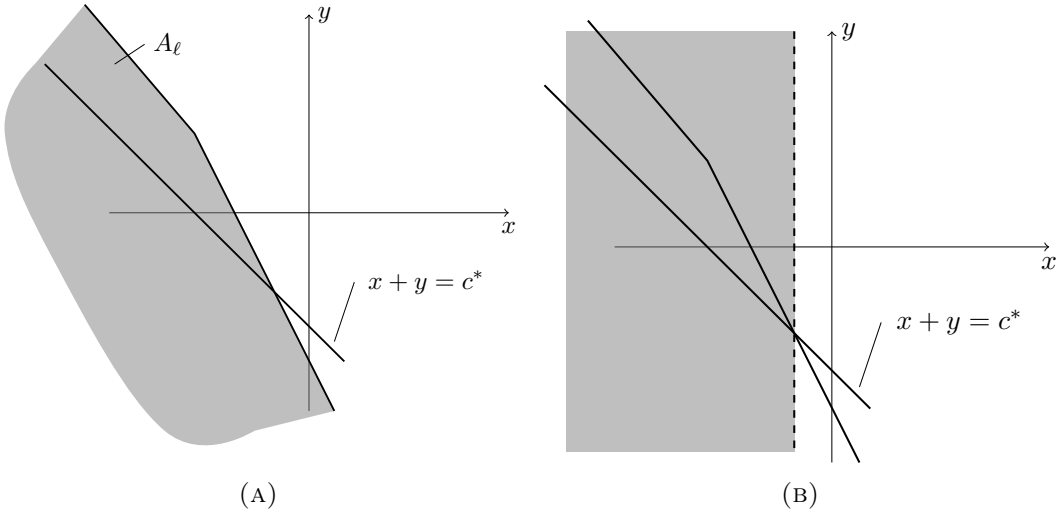
\begin{figure}[h!]
		\centering
		\begin{subfigure}[t]{0.45\textwidth}
			\resizebox{\linewidth}{!}{
				\begin{tikzpicture}[scale=0.15]
				\draw[->] (-20,0) -- (20,0);
				\draw[->](0,-20)--(0,20);
				\node[below] at (20,0) {$x$};
				\node[right] at (0,20) {$y$};
				\begin{scope}[xshift=-70]
				\path[fill=gray,opacity=0.5](-20,21)--(-9,8)--(5,-20)--plot[smooth, tension=.7] coordinates {(-3,-22)(-12,-22)(-23,-5)(-28,8)(-25,15)};
				\draw[thick](-20,21)--(-9,8)--(5,-20);
				\draw (5,-13)--(7,-7);
				\node[right] at (7.5,-7) {$x+y=c^*$};
				\draw [thick](-24,15)--(6,-15);
				\draw (-17,15)--(-14,17);
				\node[right] at (-14,17) {$A_\ell$};
				\end{scope}
				\end{tikzpicture}}
			\caption{}%
			\label{Fig_6a}
		\end{subfigure}
		\begin{subfigure}[t]{0.45\textwidth}
			\resizebox{\linewidth}{!}{
				\begin{tikzpicture}[scale=0.15]
				\draw[->] (-20,0) -- (20,0);
				\draw[->](0,-20)--(0,20);
				\node[below] at (20,0) {$x$};
				\node[right] at (0,20) {$y$};
				\begin{scope}[xshift=-70]
				\path[fill=gray,opacity=0.5] (-1,20)--(-1,-19)--(-22,-19)--(-22,20);
				\draw[thick](-20,21)--(-9,8)--(5,-20);
				\draw (5,-13)--(7,-7);
				\node[right] at (7.5,-7) {$x+y=c^*$};
				\draw [thick](-24,15)--(6,-15);
				\draw[dashed,thick] (-1,20)--(-1,-19);
				\end{scope}
				\end{tikzpicture}}
			\caption{}%
			\label{Fig_6b}
		\end{subfigure}
		\caption{When $A_{\ell }$ takes the form shown in~(a) (Case 2),
			we condition on the region shown in~(b), which weakly increases
			the mean of $\lambda _{1}+\lambda _{2}$. }%
		\label{Fig_6}
	\end{figure}
	
	If $b_{1}(t_{i},\theta )/b_{2}(t_{i},\theta )\leq 1$ for all $i$ and this
	inequality is strict for some $i$, then an entirely analogous argument
	shows that for some $d_{2}$
	\begin{equation}
	\label{e:Eigenvalue derivation 4}
	\mathbb{E}\left (\lambda _{1}+\lambda _{2}|(\lambda _{1},\lambda _{2})
	\in A_{\ell }\right )\leq \mathbb{E}(\lambda _{1}+\lambda _{2}|\lambda _{2}
	\leq d_{2}).
	\end{equation}
	
	(Case 3). Suppose that for some $i$ and $j$,
	$b_{1}(t_{i},\theta )/b_{2}(t_{i},\theta )<1<b_{1}(t_{j},\theta )/b_{2}(t_{j},
	\theta )$. Defining $c^{*}$ as before we note that the line
	$x+y=c^{*}$ must intersect $A_{\ell }$ (by definition of $c^{*}$) and so
	must intersect $\partial A_{\ell }$ at two points (since the distribution
	of $(\lambda _{1},\lambda _{2})$ has no atoms). We denote these points
	by $(d_{1},d_{2})$ and $(e_{1},e_{2})$ and without loss of generality take
	$d_{1}<e_{1}$. We now consider the region
	$\{x\leq e_{1}\}\cap \{y\leq d_{2}\}$. Reasoning as before, conditioning
	on $(\lambda _{1},\lambda _{2})$ lying in this region weakly increases
	the probability that $\lambda _{1}+\lambda _{2}=c$ for $c\geq c^{*}$ and
	weakly decreases this probability for $c<c^{*}$ (see Figure~\ref{Fig_7}).
	So in this case
	\begin{equation}
	\label{e:Eigenvalue derivation 5}
	\mathbb{E}\left (\lambda _{1}+\lambda _{2}|(\lambda _{1},\lambda _{2})
	\in A_{\ell }\right )\leq \mathbb{E}(\lambda _{1}+\lambda _{2}|\lambda _{1}
	\leq e_{1},\lambda _{2}\leq d_{2}).
	\end{equation}
	
	\begin{figure}[h!]
		\centering
		\begin{subfigure}[t]{0.45\textwidth}
			\resizebox{\linewidth}{!}{
				\begin{tikzpicture}[scale=0.15]
				\path[fill=gray,opacity=0.5](-25,15)--(-9,7)--(3,-20)--plot[smooth, tension=.7] coordinates {(3,-20)(-3,-22)(-12,-22)(-23,-5)(-28,8)(-25,15)};
				\draw[->] (-20,0) -- (20,0);
				\draw[->](0,-20)--(0,20);
				\node[below] at (20,0) {$x$};
				\node[right] at (0,20) {$y$};
				\draw[thick](-25,15)--(-9,7)--(3,-20);
				\draw (-4,9)--(-7,3);
				\node[above] at (-4,9) {$A_\ell$};
				\draw (5,-10)--(7,-7);
				\node[above] at (7.5,-7) {$x+y=c^*$};
				\draw [thick](-22,16)--(6,-12);
				\end{tikzpicture}}
			\caption{}%
			\label{Fig_7a}
		\end{subfigure}
		\begin{subfigure}[t]{0.45\textwidth}
			\resizebox{\linewidth}{!}{
				\begin{tikzpicture}[scale=0.15]
				\path[fill=gray,opacity=0.5] (-25,11)--(-6,11)--(-6,-22)--(-25,-22);
				\draw[->] (-20,0) -- (20,0);
				\draw[->](0,-20)--(0,20);
				\node[below] at (20,0) {$x$};
				\node[right] at (0,20) {$y$};
				\draw[thick](-25,15)--(-9,7)--(3,-20);
				\draw (5,-10)--(7,-7);
				\node[above] at (7.5,-7) {$x+y=c^*$};
				\draw [thick](-22,16)--(6,-12);
				\draw[thick,dashed](-25,11)--(-6,11)--(-6,-22);
				\end{tikzpicture}}
			\caption{}%
			\label{Fig_7b}
		\end{subfigure}
		\caption{When $A_{\ell }$ takes the form shown in~(a) (Case 3),
			we condition on the region shown in~(b), which weakly increases
			the mean of $\lambda _{1}+\lambda _{2}$. }%
		\label{Fig_7}
	\end{figure}
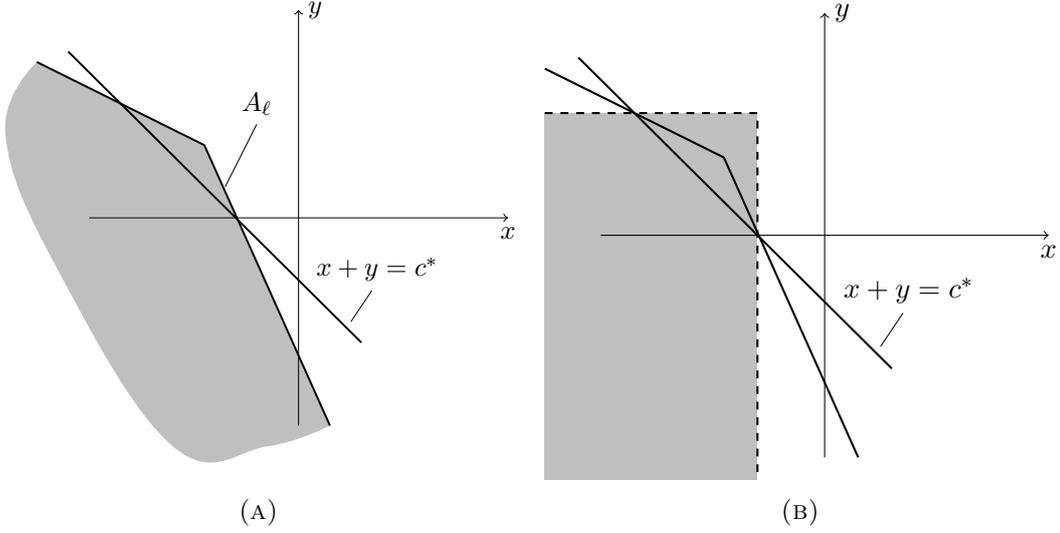
	
	From \eqref{e:Eigenvalue derivation 3}, \eqref{e:Eigenvalue derivation 4} and \eqref{e:Eigenvalue derivation 5} we see that in order to complete the
	proof of the lemma, we need only verify \eqref{e:Eigenvalue derivation 2} (or, equivalently, \eqref{e:Eigenvalue derivation 1}) when $A_{\ell }$ is of the form
	$\{\lambda _{1}\leq c_{1}\}$, $\{\lambda _{2}\leq c_{2}\}$ or
	$\{\lambda _{1}\leq c_{1},\lambda _{2}\leq c_{2}\}$. Furthermore, since
	$\lambda _{2}\geq 0\geq \lambda _{1}$, we may assume $c_{1}\leq 0$ and
	$c_{2}\geq 0$.
	
	Gromov's theorem (\cite[Theorem 1.3]{estrada2017hopital}) states that if
	$h_{1},h_{2}$ are integrable on $[a,b]$ such that $h_{2}>0$ and
	$h_{1}(x)/h_{2}(x)$ is non-increasing in $x$ then
	$\int _{a}^{c} h_{1}(x)\;dx/\int _{a}^{c} h_{2}(x)\;dx$ is non-increasing
	in $c$. Applying this to
	\begin{displaymath}
	h_{1}(x):=\int _{-\infty }^{c_{2}}\frac{d}{d\ell }q_{\ell }(x,y)\;dy
	\quad \text{and} \quad h_{2}(x):=\int _{-\infty }^{c_{2}} q_{\ell }(x,y)
	\;dy
	\end{displaymath}
	we see that, provided $h_{1}(x)/h_{2}(x)$ is non-increasing, we have
	\begin{displaymath}
	\frac{\int _{-\infty }^{c_{1}}\int _{-\infty }^{c_{2}} \frac{d}{d\ell }q_{\ell }(x,y)\;dydx}{\int _{-\infty }^{c_{1}}\int _{-\infty }^{c_{2}} q_{\ell }(x,y)\;dydx}
	\end{displaymath}
	is non-increasing in $c_{1}$. Then taking $c_{1}\to \infty $, we see that
	we need only verify \eqref{e:Eigenvalue derivation 1} when $A_{\ell }$ is
	of the form $\{\lambda _{1}\leq c_{1}\}$ or
	$\{\lambda _{2}\leq c_{2}\}$.
	
	It remains to show that $h_{1}(x)/h_{2}(x)$ is non-increasing in $x$, which
	is equivalent to
	$\mathbb{E}(\lambda _{1}+\lambda _{2}|\lambda _{1}=x,\lambda _{2}
	\leq c_{2})$ being non-decreasing in $x$ for $x<0$. Using the joint density
	of $(\lambda _{1},\lambda _{2})$ given in Proposition~\ref{p:ftildeiso}
	\begin{align*}
	\mathbb{E}(\lambda _{1}+\lambda _{2}|\lambda _{1}=x,\lambda _{2}\leq c_{2})=
	\frac{\int _{0}^{c_{2}} y(y^{2}-x^{2})\exp (-\frac{1}{2\sigma ^{2}}(y-m)^{2})\;dy}{\int _{0}^{c_{2}} y(y-x)\exp (-\frac{1}{2\sigma ^{2}}(y-m)^{2})\;dy}%
	\end{align*}
	where $m=(1+\tau )\mu \ell -\tau x$. We next differentiate this expression
	with respect to $x$. To simplify the resulting expression, we let
	$Z$ be a random variable with density proportional to
	$\exp (-\frac{1}{2\sigma ^{2}}(y-m)^{2})\mathds{1}_{y\in [0,c_{2}]}$ (i.e.\ $Z$
	is a truncated normal variable). Then using the Leibniz rule for differentiating
	integrals:
	\begin{multline*}
	\left (
	\frac{\int _{0}^{c_{2}} y(y-x)\exp (-\frac{1}{2\sigma ^{2}}(y-m)^{2})\;dy}{\int _{0}^{c_{2}}\exp (-\frac{1}{2\sigma ^{2}}(y-m)^{2})\;dy}
	\right )^{2}\left (\frac{d}{dx}\mathbb{E}(\lambda _{1}+\lambda _{2}|
	\lambda _{1}=x,\lambda _{2}\leq c_{2})\right )
	\\
	=(\mathbb{E}(Z^{2})-x\mathbb{E}(Z))\left (-3\tau \mathbb{E}\left (Z^{2}
	\right )-2x\mathbb{E}\left (Z\right )+\tau x^{2}+
	\frac{\tau c_{2}(c_{2}^{2}-x^{2})e^{-\frac{(c_{2}-m)^{2}}{2\sigma ^{2}}}}{\int _{0}^{c_{2}}\exp (-(y-m)^{2}/(2\sigma ^{2}))\;dy}
	\right )
	\\
	-\left (\mathbb{E}(Z^{3})-x^{2}\mathbb{E}(Z)\right )\left (-(2\tau +1)
	\mathbb{E}(Z)+\tau x+
	\frac{\tau c_{2}(c_{2}-x)e^{-\frac{(c_{2}-m)^{2}}{2\sigma ^{2}}}}{\int _{0}^{c_{2}}\exp (-(y-m)^{2}/(2\sigma ^{2}))\;dy}
	\right ).
	\end{multline*}
	This expression can be rearranged to take the form
	\begin{align*}
	&\underbrace{(2\tau +1)\mathbb{E}\left (Z^{3}\right )\mathbb{E}\left (Z\right )-3\tau \mathbb{E}\left (Z^{2}\right )^{2}}_{A}+
	\underbrace{\tau \mathbb{E}\left (c_{2}^{2}Z^{2}(c_{2}-Z)\right )\frac{e^{-\frac{(c_{2}-m)^{2}}{2\sigma ^{2}}}}{\int _{0}^{c_{2}}e^{-\frac{(y-m)^{2}}{2\sigma ^{2}}}\;dy}}_{B}
	\\
	&-x\left (
	\underbrace{(2-3\tau )\mathbb{E}\left (Z^{2}\right )\mathbb{E}\left (Z\right )+\tau \mathbb{E}\left (Z^{3}\right )}_{C}+
	\underbrace{\tau \mathbb{E}\left (c_{2}Z\left (c_{2}^{2}-Z^{2}\right )\right )\frac{e^{-\frac{(c_{2}-m)^{2}}{2\sigma ^{2}}}}{\int _{0}^{c_{2}}e^{-\frac{(y-m)^{2}}{2\sigma ^{2}}}\;dy}}_{D}
	\right )
	\\
	&+x^{2}\left (
	\underbrace{\tau \mathbb{E}\left (Z^{2}\right )+(1-2\tau )\mathbb{E}\left (Z\right )^{2}}_{E}+
	\underbrace{\tau c_{2}\mathbb{E}\left (Z\left (c_{2}-Z\right )\right )\frac{e^{-\frac{(c_{2}-m)^{2}}{2\sigma ^{2}}}}{\int _{0}^{c_{2}}e^{-\frac{(y-m)^{2}}{2\sigma ^{2}}}\;dy}}_{F}
	\right ).
	\end{align*}
	Then, since $x<0$, we need only verify that each of the terms $A$-$F$ are
	non-negative. We show this by using two facts: first, that
	$0\leq Z\leq c_{2}$ (which is true by definition) and second, that
	$0\leq \tau \leq 1$ (which holds because
	$\tau =(\chi ^{2}-1)/(3-\chi ^{2})$ and Assumption~\ref{a:Monotonicity}
	implies that $1\leq \chi ^{2}\leq 2$).
	
	These two facts immediately show that $B$, $D$ and $F$ are non-negative.
	Furthermore
	\begin{align*}
	A&\geq (2\tau +1)\left (\mathbb{E}\left (Z^{3}\right )\mathbb{E}(Z)-
	\mathbb{E}\left (Z^{2}\right )^{2}\right )
	\\
	C&=2(1-\tau )\mathbb{E}\left (Z^{2}\right )\mathbb{E}\left (Z\right )+
	\tau \mathrm{Cov}\left (Z^{2},Z\right )
	\\
	E&=(1-\tau )\mathbb{E}\left (Z\right )^{2}+\tau \mathrm{Var}(Z)
	\end{align*}
	(where the first inequality uses $\tau \leq 1$). Applying the Cauchy-Schwarz
	inequality to $Z^{2}=Z^{3/2}Z^{1/2}$ shows that $A$ is non-negative. Using
	the fact that $Z\geq 0$ (so that $Z^{2}$ is an increasing function of
	$Z$) implies that $C\geq 0$. Since $\tau \in [0,1]$ we see that
	$E\geq 0$.
	
	Using Gromov's theorem in the same way as above, shows that in order to
	verify \eqref{e:Eigenvalue derivation 2} for $A_{\ell }$ of the form
	$\{\lambda _{2}\leq c_{2}\}$ or $\{\lambda _{1}\leq c_{1}\}$ it is enough
	to show that
	$\mathbb{E}(\lambda _{1}+\lambda _{2}|\lambda _{2}=c_{2})$ and
	$\mathbb{E}(\lambda _{1}+\lambda _{2}|\lambda _{1}=c_{1})$ are non-decreasing
	in $c_{2}>0$ and $c_{1}<0$ respectively. This can be proven using a near
	identical calculation to that for
	$\frac{d}{dx}\mathbb{E}(\lambda _{1}+\lambda _{2}|\lambda _{1}=x,
	\lambda _{2}\leq c_{2})$ (the only change is the region on which $Z$ is
	truncated, which means there will be no terms analogous to $B$, $D$ and
	$F$ above). This completes the proof of the lemma.
\end{proof}
\begin{remark}%
	\label{r:general monotonicity conditions}
	In Assumption~\ref{a:Monotonicity}, we impose the condition that
	$\chi \geq 1$. The only point in this paper at which we use this condition
	is in the proof of Lemma~\ref{l:General field finite stochastic decreasing},
	in order to show that
	\begin{displaymath}
	\mathbb{E}\left (\lambda _{1}+\lambda _{2}|\lambda _{1}=x,\lambda _{2}
	\leq c_{2}\right ),\quad \mathbb{E}\left (\lambda _{1}+\lambda _{2}|
	\lambda _{1}=x\right ),\quad \text{and}\quad \mathbb{E}\left (\lambda _{1}+
	\lambda _{2}|\lambda _{2}=x\right )
	\end{displaymath}
	are non-decreasing in $x$ (for all $c_{2}\geq 0$). Therefore, if an alternative
	method was found to verify this property (or to verify that \eqref{e:Eigenvalue derivation 2} holds for $A_{\ell }$ of the form
	$\{\lambda _{1}\leq c_{1}\}$, $\{\lambda _{2}\leq c_{2}\}$ and
	$\{\lambda _{1}\leq c_{1},\lambda _{2}\leq c_{2}\}$) for fields with
	$\chi <1$, then our results (including Theorem~\ref{t:Monotonicity}) would
	also hold for such fields.
	
	We expect that it should be possible to extend our results in this way.
	In the proof of Lemma~\ref{l:General field finite stochastic decreasing}
	we use $\chi \geq 1$ (or equivalently, $\tau \geq 0$) to show that
	$A$-$F$ are non-negative. If we explicitly evaluate these terms using the
	higher order moments of a truncated normal distribution, then numerical
	calculations suggest that $A+B$, $C+D$ and $E+F$ are non-negative for all
	relevant values of $\tau $, (i.e.\ including negative values) which would
	be sufficient to prove this lemma in such cases. We do not attempt to prove
	this analytically, because the algebraic expressions involved in these
	calculation are quite long and we are primarily interested in the case
	of the BF field, for which $\tau =0$.
\end{remark}

\subsection{Proof of Theorem~\ref{t:Monotonicity}}

We now use Lemmas~\ref{l:RPW finite stochastic decreasing} and~\ref{l:General field finite stochastic decreasing}
to complete the proof of Theorem~\ref{t:Monotonicity}, treating the RPW
case and the general case simultaneously.

\begin{proof}[Proof of Theorem~\ref{t:Monotonicity}]
	We begin by fixing a realisation of $g$ and $\theta $. Let
	$A(\epsilon ,R)$ denote the annulus on the plane centred at the origin
	with inner radius $\epsilon $ and outer radius $R$. We discretise this
	region by considering the points with polar coordinates
	\begin{displaymath}
	\left (r_{i}^{(n)},\omega _{j}^{(n)}\right ):=(r_{i},\omega _{j}):=(
	\epsilon +i2^{-n}(R-\epsilon ),\theta +j2^{-n}2\pi )
	\end{displaymath}
	for $i,j=0,1,\dots ,2^{n}$. We consider these points as a graph by placing
	an edge between $(r_{i_{1}},\omega _{j_{1}})$ and
	$(r_{i_{2}},\omega _{j_{2}})$ if and only if
	$\lvert i_{1}-i_{2}\rvert +\lvert j_{1}-j_{2}\rvert =1$. We define a site
	percolation model by declaring the vertex $(r_{i},\omega _{j})$ open if
	$\tilde{f}_{\ell }(r_{i},\omega _{j})-\ell <0$ (so an edge is open precisely
	when both of its vertices are open). Let $S_{\epsilon ,R,n,\ell }$ denote
	the event that there is an open path between $(\epsilon ,\theta )$ and
	$(\epsilon ,\theta +\pi )$ in this percolation model.
	
	Let $S_{\epsilon ,R,\ell }$ denote the event that
	$\{\tilde{f}_{\ell }<\ell \}\cap A(\epsilon ,R)$ contains a path joining
	$(\epsilon ,\theta )$ to $(\epsilon ,\theta +\pi )$. We claim that with
	probability one,
	\begin{equation}
	\label{e:Monotonicity}
	\mathds{1}_{S_{\epsilon ,R,\ell }}=\lim _{n\to \infty }\mathds{1}_{S_{
			\epsilon ,R,n,\ell }}.
	\end{equation}
	Since $\tilde{f}_{\ell }$ has no critical points at level $\ell $ away from
	the origin (Lemma~\ref{l:no four arm saddles conditional field}), the level
	set $\{\tilde{f}_{\ell }=\ell \}\cap A(\epsilon ,R)$ consists of
	$C^{2+\eta }$ curves. So in particular, if there exists a path in
	$\{\tilde{f}_{\ell }<\ell \}\cap A(\epsilon ,R)$ joining
	$(\epsilon ,\theta )$ to $(\epsilon ,\theta +\pi )$, then for $n$ sufficiently
	large we may assume this path lies on the graph with vertices
	$(r_{i}^{(n)},\omega _{j}^{(n)})$ as defined above. Hence
	$\mathds{1}_{S_{\epsilon ,R,\ell }}\leq \liminf _{n}\mathds{1}_{S_{
			\epsilon ,R,n,\ell }}$. If there is no path in
	$\{\tilde{f}_{\ell }<\ell \}\cap A(\epsilon ,R)$ joining
	$(\epsilon ,\theta )$ to $(\epsilon ,\theta +\pi )$ then there are three
	possibilities: (1) $\tilde{f}_{\ell }-\ell $ is non-negative at
	$(\epsilon ,\theta )$ or $(\epsilon ,\theta +\pi )$; (2) there exists a
	path in $\{\tilde{f}_{\ell }\geq \ell \}\cap A(\epsilon ,R)$ joining
	$(\epsilon ,\omega _{i})$ to $(\epsilon ,\omega _{j})$ for some
	$\omega _{i}\in (\theta ,\theta +\pi )$ and
	$\omega _{j}\in (\theta -\pi ,\theta )$, (here we note that by Lemma~\ref{l:no four arm saddles conditional field},
	$\tilde{f}_{\ell }|_{\partial B(\epsilon )}$ has no local extrema at level
	$\ell $ and we assume that $n$ is sufficiently large to find such
	$\omega _{i},\omega _{j}$); or (3) there exist two paths in
	$\{\tilde{f}_{\ell }\geq \ell \}\cap A(\epsilon ,R)$ which join
	$(\epsilon ,\omega _{i})$ and $(\epsilon ,\omega _{j})$ respectively to
	$\partial B(R)$ for $\omega _{i},\omega _{j}$ as before (See Figure~\ref{Fig_8}).
	In this case, by Lemma~\ref{l:no four arm saddles conditional field} we
	may assume that the paths intersect $\partial B(R)$ at different points.
	\begin{figure}[h!]
		\centering
		\begin{subfigure}[t]{0.45\textwidth}
			\begin{tikzpicture}[scale=0.1]
			\draw[fill=gray,fill opacity=0.5] plot[smooth cycle, tension=.7] coordinates {(0,0)(-3,8)(0,12)(6,14)(15,12)(20,0)(17,-5)(12,-10)(4,-15)(2,-5)(0,0)(3,-2)(7,-7)(12,-4)(14,0)(13,5)(7,7)};
			\draw[dashed,fill=white] (0,0) circle (5);
			\draw[dashed] (0,0) circle (30);
			\node[right] at (0:5) {$(\epsilon,\theta)$};
			\fill (0:5) circle (10pt);
			\fill (180:5) circle (10pt);
			\node[left] at (180:5) {$(\epsilon,\theta+\pi)$};
			\draw (90:5)--(135:8);
			\node[left] at (135:8) {$(\epsilon,\omega_i)$};
			\fill (90:5) circle (10pt);
			\draw (-50:5)--(-120:8);
			\node[left] at (-120:8) {$(\epsilon,\omega_j)$};
			\fill (-50:5) circle (10pt);
			\draw (0,16)--(2,12);
			\node[above] at (0,16) {$\left\{\tilde{f}_\ell\geq\ell\right\}$};
			\node[left,above] at (135:32) {$A_{\epsilon,R}$};
			\draw[fill=gray,fill opacity=0.5] plot[smooth cycle, tension=.7] coordinates {(-18,13)(-15,13)(-15,10)(-17,6)(-19,5)};
			\end{tikzpicture}
			\caption{}%
			\label{Fig_8a}
		\end{subfigure}
		\begin{subfigure}[t]{0.45\textwidth}
			\begin{tikzpicture}[scale=0.1]
			\node[left,above] at (135:32) {$A_{\epsilon,R}$};
			\draw[dashed] (0,0) circle (30);
			\clip circle (30);
			\draw[fill=gray,fill opacity=0.5] plot[smooth cycle, tension=.7] coordinates {(0,0)(-3,8)(-5,12)(-4,17)(-4,22)(-5,27)(-6,33)(0,32)(7,30)(12,28)(17,27)(16,25)(14,23)(12,20)(9,17)(7,14)};
			\draw[fill=gray,fill opacity=0.5] plot[smooth cycle, tension=.7] coordinates {(17,-5)(12,-10)(4,-15)(2,-5)(0,0)(3,-2)(7,-7)(12,-4)(16,-3)(22,-3)(26,-4)(30,-5)(32,-9)(27,-9)(23,-8)};
			\draw[dashed,fill=white] (0,0) circle (5);
			\node[right] at (0:5) {$(\epsilon,\theta)$};
			\fill (0:5) circle (10pt);
			\fill (180:5) circle (10pt);
			\node[left] at (180:5) {$(\epsilon,\theta+\pi)$};
			\draw (90:5)--(135:8);
			\node[left] at (135:8) {$(\epsilon,\omega_i)$};
			\fill (90:5) circle (10pt);
			\draw (-50:5)--(-120:8);
			\node[left] at (-120:8) {$(\epsilon,\omega_j)$};
			\fill (-50:5) circle (10pt);
			\draw (-11,15)--(0,13);
			\node[above] at (-11,15) {$\left\{\tilde{f}_\ell\geq\ell\right\}$};
			\draw[fill=white] plot[smooth cycle, tension=.7] coordinates {(0,20)(4,23)(4,20)};
			\end{tikzpicture}
			\caption{}%
			\label{Fig_8b}
		\end{subfigure}
		\caption{Two of the three ways in which $S_{\epsilon ,R,\ell }$ can fail,
			corresponding to cases (2) and (3) above respectively.}%
		\label{Fig_8}
	\end{figure}
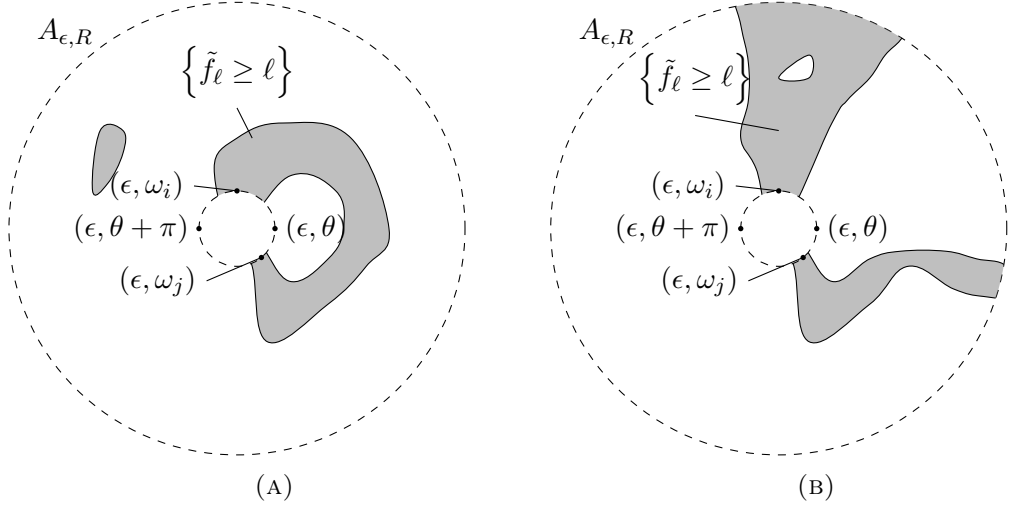
	
	In each of these cases, for all $n$ large enough we can construct corresponding
	paths on the discrete lattice as above which block a discrete path from
	joining $(\epsilon ,\theta )$ to $(\epsilon ,\theta +\pi )$ in
	$\{\tilde{f}_{\ell }<\ell \}$ and so $S_{\epsilon ,R,n,\ell }$ cannot occur
	for sufficiently large $n$. Therefore
	$\mathds{1}_{S_{\epsilon ,R,\ell }}\geq \limsup _{n}\mathds{1}_{S_{
			\epsilon ,R,n,\ell }}$, completing the proof of the claim.
	
	Since $S_{\epsilon ,R,n,\ell }$ depends on only finitely many points of
	$\tilde{f}_{\ell }$ and is a decreasing event, by Lemma~\ref{l:RPW finite stochastic decreasing}
	for the RPW and Lemma~\ref{l:General field finite stochastic decreasing}
	for general fields
	\begin{displaymath}
	\mathbb{P}\left (S_{\epsilon ,R,n,\ell _{1}}|g,\theta \right )\leq
	\mathbb{P}\left (S_{\epsilon ,R,n,\ell _{2}}|g,\theta \right )
	\end{displaymath}
	for any $\ell _{1}<\ell _{2}$. Then by \eqref{e:Monotonicity} and the bounded
	convergence theorem
	\begin{equation}
	\label{e:Monotonicity 2}
	\mathbb{P}\left (S_{\epsilon ,R,\ell _{1}}|g,\theta \right )\leq
	\mathbb{P}\left (S_{\epsilon ,R,\ell _{2}}|g,\theta \right ).
	\end{equation}
	
	Now let $S_{R,\ell }$ be the event that $\tilde{f}_{\ell }$ has an $R$-lower
	connected saddle point at the origin. Conditional on $\theta $, if this
	event occurs then so must $S_{\epsilon ,R,\ell }$ for $\epsilon $ sufficiently
	small. Conversely, if the saddle point at the origin is not $R$-lower connected,
	then it must be four-arm in $B(R)$ or $R$-upper connected. In both of these
	cases, $S_{\epsilon ,R,\ell }$ cannot occur for $\epsilon $ sufficiently
	small. We conclude that
	$\mathds{1}_{S_{R,\ell }}=\lim _{\epsilon \to 0}\mathds{1}_{S_{
			\epsilon ,R,\ell }}$ and by applying the bounded convergence theorem to \eqref{e:Monotonicity 2} we see that
	\begin{equation}
	\label{e:Monotonicity 3}
	\mathbb{P}\left (S_{R,\ell _{1}}|g,\theta \right )\leq \mathbb{P}
	\left (S_{R,\ell _{2}}|g,\theta \right ).
	\end{equation}
	
	Finally we let $S_{\ell }$ be the event that $\tilde{f}_{\ell }$ has a lower
	connected saddle point at the origin and note that trivially
	$S_{\ell }=\cup _{R}S_{R,\ell }$. Applying this to \eqref{e:Monotonicity 3} shows that
	\begin{displaymath}
	\mathbb{P}\left (S_{\ell _{1}}|g,\theta \right )\leq \mathbb{P}\left (S_{
		\ell _{2}}|g,\theta \right ).
	\end{displaymath}
	Integrating over realisations of $g$ and $\theta $ implies that
	$\mathbb{P}(S_{\ell _{1}})\leq \mathbb{P}(S_{\ell _{2}})$ and so by definition
	(see the proof of Theorem~\ref{t:Differentiability equivalence})
	\begin{displaymath}
	\frac{p_{s^{-}}^{*}(\ell _{1})}{p_{s}(\ell _{1})}\leq
	\frac{p_{s^{-}}^{*}(\ell _{2})}{p_{s}(\ell _{2})}.
	\end{displaymath}
	A near identical argument shows that
	$p_{s^{+}}^{*}(\ell )/p_{s}(\ell )$ is non-increasing in $\ell $.
\end{proof}

\subsection{Remaining results}

We now prove the remaining results stated in Section~\ref{ss:Monotonicity},
namely Corollaries~\ref{c:Twice differentiable} and~\ref{c:Differentiability equivalence}
and Propositions~\ref{p:RPW monotonicity}--\ref{p:Positivity RPW}.

\begin{proof}[Proof of Corollary~\ref{c:Twice differentiable}]
	Since $p_{s^{-}}^{*}/p_{s}$ is monotone it has at most a countable number
	of discontinuities, all of which are jump discontinuities. By the continuity
	of $p_{s}$, the same is true of $p_{s^{-}}^{*}$. Since $c_{ES}$ is absolutely
	continuous (Theorem~\ref{t:integral equality}) it is differentiable almost
	everywhere (see \cite[Theorem 7.18]{rudin1987real}) with derivative
	$p_{s^{-}}^{*}-p_{m^{+}}$. The density $p_{m^{+}}$ is derived explicitly
	in~\cite{cheng2015expected} and is continuously differentiable. It also
	follows from monotonicity that $p_{s^{-}}^{*}/p_{s}$ is differentiable
	almost everywhere, and since $p_{s}$ is smooth (again, from
	\cite{cheng2015expected}) the same is true of $p_{s^{-}}^{*}$, thus showing
	that $c_{ES}$ is twice differentiable almost everywhere. A similar proof
	applies to $c_{LS}$.
\end{proof}

\begin{proof}[Proof of Corollary~\ref{c:Differentiability equivalence}]
	Since the equivalence of $(2)$--$(4)$ follows from Theorem~\ref{t:integral equality},
	and $(1)$ implies $(2)$ by Theorem~\ref{t:Differentiability equivalence},
	it remains to show that $(2)$ implies $(1)$. Now suppose there exists a
	version of $p_{s^{-}}$ which is continuous on $(a,b)$, denoted
	$\tilde{p}_{s^{-}}$. Then
	$\tilde{p}_{s^{-}}/p_{s}=p_{s^{-}}^{*}/p_{s}$ almost everywhere, and since
	the former is continuous and the latter is monotone, this equality must
	hold pointwise on $(a,b)$, so $p_{s^{-}}^{*}$ is continuous on
	$(a,b)$. We note that $\tilde{p}_{s^{+}}:=p_{s}-\tilde{p}_{s^{-}}$ defines
	a continuous version of $p_{s^{+}}$ and arguing as above then shows that
	$p_{s^{+}}^{*}$ is continuous on $(a,b)$. Therefore the almost everywhere
	equality $p_{s^{-}}^{*}+p_{s^{+}}^{*}=p_{s}$ is in fact true for all points
	in $(a,b)$, and by \eqref{e:Differentiability equivalent 1}
	$\tilde{f}_{\ell }$ almost surely has no infinite four-arm saddle at the
	origin for all $\ell \in (a,b)$.
\end{proof}

\begin{proof}[Proof of Proposition~\ref{p:Positivity RPW}]
	We use the `barrier method', that is, we show that the probability of having
	at least one component of $\{f\geq \ell \}$ contained in $B(r)$ is strictly
	positive for some fixed $r>0$. By linearity of expectation and stationarity
	of $f$, this shows that
	$\liminf _{R\to \infty }\mathbb{E}(N_{ES}(R,\ell ))/R^{2}>0$, so in particular
	$c_{ES}(\ell )>0$.
	
	It is known that the RPW has the orthogonal expansion
	\begin{displaymath}
	f(x)=\sum _{m\in \mathbb{Z}} a_{m} J_{\lvert m\rvert }(r)e^{im\theta }
	\end{displaymath}
	where $(r,\theta )$ represents $x$ in polar coordinates, $J_{m}$ is the
	$m$-th Bessel function and $a_{m}=b_{m}+ic_{m}=\overline{a_{-m}}$ with
	$b_{0}$, $(\sqrt{2}b_{m})_{m\in \mathbb{N}}$ and $(\sqrt{2}c_{m})_{m\in \mathbb{N}}$ independent standard (real) Gaussians and
	$c_{0}=0$. (This function is clearly Gaussian and can be shown to have
	the correct covariance structure using Graf's addition theorem for Bessel
	functions.) Let $r$ be the minimiser of $J_{0}$, so $r\approx 3.83$ and
	$J_{0}(r)<-0.4$. We note that by considering the power series for the Bessel
	functions, it can be shown that for $x\in [0,4]$,
	$\lvert J_{m}(x)\rvert \leq e^{4}(2^{m}/m!)$. Finally we note that
	$J_{m}$ is bounded in absolute value by $1$ for any $m$. Now consider the
	event that
	\begin{displaymath}
	a_{0}>\min \{\lvert \ell \rvert ,1\}\quad \text{and}\quad \lvert a_{-1}
	\rvert ,\lvert a_{-2}\rvert ,\lvert a_{1}\rvert ,\lvert a_{2}\rvert
	\leq C_{1}\quad \text{and}\quad \forall \lvert m\rvert >2,\; \lvert a_{m}
	\rvert \leq C_{2}(m!)/4^{m}.
	\end{displaymath}
	It is easily seen that this event has positive probability, and for appropriately
	chosen constants $C_{1},C_{2}>0$, we see that on this event
	$f(0)>\ell $ and $f(x)<\ell $ for any $x$ such that
	$\lvert x\rvert =r$. Therefore $f$ has a component of
	$\{f\geq \ell \}$ contained in $B(r)$ with positive probability, completing
	the proof of the result.
\end{proof}

\begin{proof}[Proof of Proposition~\ref{p:RPW monotonicity}]
	By Corollary~\ref{c:Differentiability equivalence} we may take
	$p_{s^{-}}(\ell )/p_{s}(\ell )$ to be non-decreasing. In
	\cite{cheng2015expected} it is shown that for the RPW
	\begin{align*}
	p_{m^{+}}(x)&=\frac{1}{4\sqrt{2}\pi ^{3/2}}\left ((x^{2}-1)e^{-
		\frac{x^{2}}{2}}+e^{-\frac{3x^{2}}{2}}\right )\mathds{1}_{x\geq 0}
	\\
	p_{s}(x)&=\frac{1}{4\sqrt{2}\pi ^{3/2}}e^{-\frac{3x^{2}}{2}}.
	\end{align*}
	In particular, $p_{m^{+}}(x)=0$ for $x<0$, so by Theorem~\ref{t:integral equality}
	for $\ell ^{\prime }<\ell \leq 0$
	\begin{displaymath}
	c_{ES}(\ell ^{\prime })-c_{ES}(\ell )=\int _{\ell ^{\prime }}^{\ell }-p_{s^{-}}(x)
	\;dx.
	\end{displaymath}
	Taking $\ell ^{\prime }\to -\infty $ shows that for $\ell <0$
	\begin{displaymath}
	c_{ES}(\ell )=\int _{-\infty }^{\ell }p_{s^{-}}(x)\;dx.
	\end{displaymath}
	By Proposition~\ref{p:Positivity RPW} this is positive for every
	$\ell <0$, so in particular there must exist arbitrarily negative
	$x$ such that $p_{s^{-}}(x)>0$. Since
	$p_{s^{-}}(\ell )/p_{s}(\ell )$ is non-decreasing, we conclude that
	$p_{s^{-}}$ is strictly positive for all $\ell \in \mathbb{R}$. Since
	$p_{s^{-}}(x)=p_{s^{+}}(-x)$ we also see that $p_{s^{+}}(x)>0$ for all
	$x$ and since $p_{s^{-}}+p_{s^{+}}=p_{s}$ we see that
	$0<p_{s^{-}}(x)/p_{s}(x)<1$ for all $x\in \mathbb{R}$. Finally, we note
	that there must exist a sequence $\ell _{n}>0$ with $\ell _{n}\to 0$ such
	that $p_{s^{-}}(\ell _{n})/p_{s}(\ell _{n})\geq 1/2$ for all~$n$. Indeed,
	if this were not true, by monotonicity, there would exist a neighbourhood
	of $0$ on which $p_{s^{-}}/p_{s}<1/2$ and by symmetry
	$p_{s^{+}}/p_{s}<1/2$ on a possibly smaller neighbourhood, but then there
	would exist a set of positive measure on which
	$p_{s^{-}}+p_{s^{+}}<p_{s}$ giving a contradiction.
	
	For $\ell ^{\prime }\leq \ell $ and $\epsilon >0$
	\begin{align*}
	\frac{1}{\epsilon }\int _{\ell }^{\ell +\epsilon }
	\frac{p_{s^{-}}(\ell ^{\prime })}{p_{s}(\ell ^{\prime })}p_{s}(x)-p_{m^{+}}(x)
	\;dx&\leq \frac{1}{\epsilon }\int _{\ell }^{\ell +\epsilon }p_{s^{-}}(x)-p_{m^{+}}(x)
	\;dx
	\\
	&\leq \frac{1}{\epsilon }\int _{\ell }^{\ell +\epsilon }
	\frac{p_{s^{-}}(\ell +\epsilon )}{p_{s}(\ell +\epsilon )}p_{s}(x)-p_{m^{+}}(x)
	\;dx.
	\end{align*}
	By Theorem~\ref{t:integral equality} and continuity of $p_{s}$ we therefore
	see that
	\begin{displaymath}
	\frac{p_{s^{-}}(\ell ^{\prime })}{p_{s}(\ell ^{\prime })}p_{s}(\ell )-p_{m^{+}}(
	\ell )\leq D_{+}c_{ES}(\ell )\leq D^{+}c_{ES}(\ell )\leq
	\frac{p_{s^{-}}(\ell +\epsilon )}{p_{s}(\ell +\epsilon )}p_{s}(\ell )-p_{m^{+}}(
	\ell ).
	\end{displaymath}
	Since $p_{s^{-}}/p_{s}<1$, evaluating the final inequality using the explicit
	forms of $p_{s}$ and $p_{m^{+}}$ shows that $D^{+}c_{ES}(\ell )<0$ whenever
	$\ell \geq 1$. Since $p_{s^{-}}>0$ and $p_{m^{+}}(\ell )=0$ for
	$\ell \leq 0$, taking $\ell ^{\prime }=\ell $ in the first inequality shows
	that $D_{+}c_{ES}(\ell )>0$ for $\ell \leq 0$. If $\ell >0$ then we may
	take $\ell ^{\prime }=\ell _{n}$ as defined above for sufficiently large
	$n$. Then by evaluating the densities explicitly we see that
	$1/2p_{s}(\ell )-p_{m^{+}}(\ell )>0$ for $\ell \in (0,0.876]$ thus completing
	the proof of the statements for $c_{ES}$.
	
	Since $p_{m^{-}}(x)=0$ for $x>0$, we see from Theorem~\ref{t:integral equality}
	that
	\begin{displaymath}
	\frac{c_{LS}(\ell +\epsilon )-c_{LS}(\ell )}{\epsilon }=
	\frac{c_{ES}(\ell +\epsilon )-c_{ES}(\ell )}{\epsilon }-
	\frac{1}{\epsilon }\int _{\ell }^{\ell +\epsilon }p_{s^{+}}(x)\;dx
	\end{displaymath}
	for $\ell >0$. As $D^{+}c_{ES}(\ell )<0$ for $\ell \geq 1$ and
	$p_{s^{+}}\geq 0$, taking the limit superior here shows that
	$D^{+}c_{LS}(\ell )<0$ (for $\ell \geq 1$).
\end{proof}

\begin{proof}[Proof of Propositions~\ref{p:BF monotonicity} and~\ref{p:Isotropic monotonicity}]
	By Theorem~\ref{t:differentiability of c_{LS}}, both $c_{ES}$ and
	$c_{LS}$ are differentiable and so by Theorem~\ref{t:integral equality}
	\begin{equation}
	\label{e:Isotropic monotonicity}
	\begin{aligned}
	c_{ES}^{\prime }(\ell )&=p_{s^{-}}(\ell )-p_{m^{+}}(\ell )\leq p_{s}(
	\ell )-p_{m^{+}}(\ell )
	\\
	c_{LS}^{\prime }(\ell )&=p_{m^{-}}(\ell )+p_{s^{-}}(\ell )-p_{m^{+}}(
	\ell )-p_{s^{+}}(\ell )\leq p_{m^{-}}(\ell )+p_{s}(\ell )-p_{m^{+}}(
	\ell ).
	\end{aligned}
	\end{equation}
	The densities $p_{m^{-}}$, $p_{s}$ and $p_{m^{+}}$ were derived for isotropic
	fields satisfying (a weaker version of) Assumption~\ref{a:minimal} in
	\cite{cheng2015expected}. In the proof of
	\cite[Corollary~1.19]{Beliaev2018Number} it is shown that both right hand
	expressions in \eqref{e:Isotropic monotonicity} are strictly negative whenever
	$\ell >\sqrt{2}/\chi $ (with $\chi $ defined prior to the statement of
	this proposition). We note that this is a sufficient condition for the
	derivatives to be negative, chosen for its simplicity. For many fields,
	the derivatives will be negative on a larger region and this can be found
	by using the densities specified in \cite{cheng2015expected} with the appropriate
	value of $\chi $. Specifically, these densities are given in terms of
	$\chi $ and $\xi ^{2}:=-k^{\prime }(0)/k^{\prime \prime }(0)$ by
	\begin{align*}
	p_{m^{+}}(x)=p_{m^{-}}(-x)&=\frac{1}{\pi \xi ^{2}}\Bigg (\chi ^{2}(x^{2}-1)
	\phi (x)\Phi \left (\frac{\chi x}{\sqrt{2-\chi ^{2}}}\right )+
	\frac{\chi x\sqrt{2-\chi ^{2}}}{2\pi }e^{-\frac{x^{2}}{2-\chi ^{2}}}
	\\
	&\qquad \qquad \qquad \qquad +
	\frac{\sqrt{2}}{\sqrt{\pi (3-\chi ^{2})}}e^{-
		\frac{3x^{2}}{2(3-\chi ^{2})}}\Phi \left (
	\frac{\chi x}{\sqrt{(3-\chi ^{2})(2-\chi ^{2})}}\right )\Bigg )
	\\
	p_{s}(x)&=\frac{1}{\pi \xi ^{2}}
	\frac{\sqrt{2}}{\sqrt{\pi (3-\chi ^{2})}}e^{-
		\frac{3x^{2}}{2(3-\chi ^{2})}}
	\end{align*}
	where $\phi $ and $\Phi $ denote the standard normal probability density
	and cumulative density respectively. For the Bargmann-Fock field, (for
	which $\chi =1$) substituting these densities into \eqref{e:Isotropic monotonicity} shows that
	$c_{ES}^{\prime }(\ell ),c_{LS}^{\prime }(\ell )<0$ for $\ell \geq 1.03$ improving
	on the general bound $\ell >\sqrt{2}/\chi =\sqrt{2}$.
	
	Finally we note that $c_{ES}^{\prime }(0)=p_{s^{-}}(0)-p_{m^{+}}(0)$, and
	by the identities $p_{s^{-}}(x)=p_{s^{+}}(-x)$,
	$p_{s^{-}}+p_{s^{+}}=p_{s}$ almost everywhere and the fact these densities
	are all continuous, we see that $p_{s^{-}}(0)=p_{s}(0)/2$. Evaluating the
	densities given in \cite{cheng2015expected} at zero shows that
	$p_{s}(0)/2>p_{m^{+}}(0)$ so we conclude that $c_{ES}^{\prime }(0)>0$. Since
	$c_{ES}$ is continuously differentiable, we can extend this to a neighbourhood
	of the origin.
	
	By Theorem~\ref{t:Monotonicity}, $p_{s^{-}}(\ell )/p_{s}(\ell )$ is non-decreasing
	and so for $\ell >0$
	\begin{displaymath}
	\frac{p_{s^{-}}(\ell )}{p_{s}(\ell )}\geq \frac{p_{s-}(0)}{p_{s}(0)}=
	\frac{1}{2}.
	\end{displaymath}
	Therefore
	$c_{ES}^{\prime }(\ell )\geq p_{s}(\ell )/2-p_{m^{+}}(\ell )$ for
	$\ell \geq 0$. Evaluating the densities above then gives an explicit constant
	$C$ such that this expression is strictly positive for $\ell \leq C$. In
	the case of the Bargmann-Fock field, $C=0.64$.
\end{proof}

\appendix

\section{Non-degeneracy}
\label{a:nondeg}

We verify some claims about the non-degeneracy of Gaussian fields:

\begin{lemma}
	\label{a:nondegen}
	Let $f$ be a $C^{2}$, stationary, planar Gaussian field. Then the spectral
	measure $\mu $ being supported on the union of two lines through the origin
	is equivalent to the Gaussian vector $\nabla ^{2} f(0)$ being degenerate.
\end{lemma}

\begin{proof}
	By \cite[Chapter 5]{RFG}, for any $s,t\in \mathbb{R}^{2}$ and
	$\alpha ,\beta ,\gamma ,\delta \in \mathbb{N}\cup \{0\}$ (such that the
	following derivatives are defined)
	\begin{displaymath}
	\mathbb{E}\left (
	\frac{\partial ^{\alpha +\beta }}{\partial t_{1}^{\alpha }\partial t_{2}^{\beta }}f(t)
	\overline{\frac{\partial ^{\gamma +\delta }}{\partial s_{1}^{\gamma }\partial s_{2}^{\delta }}f(s)}
	\right )=\int _{\mathbb{R}^{2}}(-ix_{1})^{\alpha }(-ix_{2})^{\beta }e^{-it
		\cdot x}\overline{(-ix_{1})^{\gamma }(-ix_{2})^{\delta }e^{-is\cdot x}}\;d
	\mu (x)
	\end{displaymath}
	where $\mu $ is the spectral measure of $f$. Then for
	${a}\in \mathbb{R}^{3}$,
	\begin{displaymath}
	\mathbb{E}\left (({a}\cdot \nabla ^{2} f(0))^{2}\right )=\int _{
		\mathbb{R}^{2}}\lvert a_{1} x_{1}^{2}+a_{2}x_{2}^{2}+a_{3}x_{1}x_{2}
	\rvert ^{2}\;d\mu (x).
	\end{displaymath}
	
	If $\nabla ^{2} f(0)$ is degenerate, then we may choose ${a}\neq {0}$ such
	that this expression is zero, and hence the integrand is identically zero
	on the support of $\mu $. Hence the support of $\mu $ is contained in the
	zero set of this binary quadratic form which is contained in the union
	of two lines through the origin.
	
	Conversely if the support of $\mu $ is contained in the union of two lines
	through the origin, then we may choose ${a}\neq {0}$ such that the zero
	set of $a_{1} x_{1}^{2}+a_{2}x_{2}^{2}+a_{3}x_{1}x_{2}$ is equal to this
	union. Hence the integral above will be zero and $\nabla ^{2} f(0)$ will
	be degenerate.
\end{proof}

\begin{lemma}
	\label{a:nondegen2}
	Let $f:\mathbb{R}^{2}\to \mathbb{R}$ be a Gaussian field which is stationary
	and centred with $\mathrm{Var}(f(0))=1$ and covariance function
	$K\in C^{4+\eta ^{\prime }}$. If the support of the spectral measure
	$\mu $ contains a centred ellipse (or circle), then Assumptions~\ref{a:minimal}
	and~\ref{a:non-degenerate gradient} hold. Moreover, if the support of
	$\mu $ contains an open set then, for any distinct
	$t_{1},\dots ,t_{n}\subset \mathbb{R}^{2}$, the vector
	\begin{displaymath}
	(f(t_{1}),\dots ,f(t_{n}),\nabla f(t_{1}),\dots ,\nabla f(t_{n}),
	\nabla ^{2} f(t_{1}),\dots ,\nabla ^{2} f(t_{n}))
	\end{displaymath}
	is non-degenerate (so in particular, Assumptions~\ref{a:minimal} and~\ref{a:non-degenerate gradient}
	hold).
\end{lemma}
We note that these results could be proven under much weaker conditions
on the support of the spectral measure using the arguments we give below.
We do not attempt to formulate the most general conditions.
\begin{proof}
	First consider the case that the support of $\mu $ contains an ellipse/circle.
	Let ${a}\in \mathbb{R}^{9}$ and
	\begin{displaymath}
	{w}:=(f(t),\nabla f(t),f(0),\nabla f(0),\nabla ^{2} f(0)).
	\end{displaymath}
	By the same arguments as in the proof of Lemma~\ref{a:nondegen}
	\begin{displaymath}
	\mathbb{E}\left (({a}\cdot {w})^{2}\right )=\int _{\mathbb{R}^{2}}
	\left \lvert {a}\cdot (-e^{-it\cdot x},ix_{1}e^{-it\cdot x},ix_{2}e^{-it
		\cdot x}, -1,ix_{1},ix_{2},x_{1}^{2},x_{1}x_{2},x_{2}^{2})\right
	\rvert ^{2}d\mu (x).
	\end{displaymath}
	If Assumption~\ref{a:non-degenerate gradient} does not hold, then there
	exists a choice of ${a}$ such that this expectation is zero and one of
	the first three elements of ${a}$ is non-zero. Hence the integrand above
	must be identically zero on the support of $\mu $. By considering the real
	and imaginary parts explicitly, the zero set of this integrand cannot contain
	an ellipse/circle centred at the origin and so neither can the support
	of $\mu $. By a near-identical argument, and Lemma~\ref{a:nondegen},
	$f$ also satisfies the non-degeneracy conditions of Assumption~\ref{a:minimal}.
	(The other conditions are satisfied by the premise of this lemma.)
	
	By a completely analogous argument we see that if
	\begin{displaymath}
	\left (f(t_{1}),\dots ,f(t_{n}),\nabla f(t_{1}),\dots ,\nabla f(t_{n}),
	\nabla ^{2} f(t_{1}),\dots ,\nabla ^{2} f(t_{n})\right )
	\end{displaymath}
	is degenerate then some non-trivial linear combination of
	\begin{align*}
	&e^{-it_{1}\cdot x},\dots ,e^{-it_{n}\cdot x},
	\\
	&ix_{1}e^{-it_{1}\cdot x},\dots ,ix_{1}e^{-it_{n}\cdot x},ix_{2}e^{-it_{1}
		\cdot x},\dots ,ix_{2}e^{-it_{n}\cdot x}
	\\
	&x_{1}^{2}e^{-it_{1}\cdot x},\dots ,x_{1}^{2}e^{-it_{n}\cdot x},x_{2}^{2}e^{-it_{1}
		\cdot x},\dots ,x_{2}^{2}e^{-it_{n}\cdot x},x_{1}x_{2}e^{-it_{1}
		\cdot x},\dots ,x_{1}x_{2}e^{-it_{n}\cdot x}
	\end{align*}
	is identically zero on the support of $\mu $. Since the $t_{i}$ are distinct,
	we see that the support of $\mu $ cannot contain an open set.
\end{proof}

\begin{lemma}
	\label{a:nondegen3}
	Let $f$ be a Gaussian field satisfying Assumption~\ref{a:minimal}. Then
	the density of saddle points $p_{s}$ defined in Proposition~\ref{p:density existence}
	is non-zero for all $\ell \in \mathbb{R}$.
\end{lemma}
\begin{proof}
	By the Kac-Rice theorem (Corollary 11.2.2 of \cite{RFG}), and the independence
	of $\nabla f(0)$ and $(f(0), \nabla ^{2} f(0))$,
	\begin{equation}
	\label{e:ps}
	p_{s}(\ell ) = \mathbb{E}\left [\left \lvert \det \nabla ^{2} f(0)
	\right \rvert \mathds{1}_{\det \nabla ^{2} f(0)<0} \, \middle | \, f(0)=
	\ell \right ] p_{f(0)}(\ell ).
	\end{equation}
	We note that by Gaussian regression
	\begin{displaymath}
	\mathrm{Cov}\left (\nabla ^{2} f(0) \, \middle | \, f(0) = \ell
	\right )=\mathrm{Cov}\left (\nabla ^{2} f(0)\right )-\mathrm{Cov}\left (
	\nabla ^{2} f(0),f(0)\right )\mathrm{Cov}\left (\nabla ^{2} f(0),f(0)
	\right )^{t}
	\end{displaymath}
	where $\mathrm{Cov}\left (\nabla ^{2} f(0)\right )$ is a three by three
	matrix and $\mathrm{Cov}\left (\nabla ^{2} f(0),f(0)\right )$ is a three-dimen\-sional
	row vector. Since we assume that $\nabla ^{2} f(0)$ is non-degenerate,
	the conditional covariance matrix above is the difference between a rank
	three and rank one matrix, so must have rank at least two. Therefore
	$(\nabla ^{2} f(0)|f(0)=\ell )$ must be supported on either a two or three
	dimensional (affine) subspace of $\mathbb{R}^{3}$. This implies that the
	support of
	\begin{equation}
	\label{e:det}
	\left (\det \nabla ^{2} f(0) \, \middle | \, f(0) = \ell \right )
	\end{equation}
	is $\mathbb{R}$, and hence by \eqref{e:ps} $p_{s}(\ell )>0$ for all
	$\ell \in \mathbb{R}$.
\end{proof}

\printbibliography

\end{document}